\documentclass[10pt,a4paper,english]{amsart}
\usepackage{babel}
\usepackage[numbers]{natbib}
\usepackage{amsfonts, amsmath, wasysym}
\usepackage{amssymb, amsthm}
\usepackage{mathrsfs}
\usepackage{verbatim}
\usepackage{color}
\newfont{\ffont}{cmr10}

\allowdisplaybreaks
\numberwithin{equation}{section}

\newcommand{\cB}{\mathcal{B}}
\newcommand{\cG}{\mathcal{G}}
\newcommand{\cP}{\mathcal{P}}
\newcommand{\cF}{\mathcal{F}}
\newcommand{\cS}{\mathcal{S}}
\newcommand{\cE}{\mathcal{E}}

\newcommand{\cJ}{\mathcal{J}}
\newcommand{\cI}{\mathcal{I}}
\newcommand{\cD}{\mathcal{D}}

\newcommand{\cM}{\mathcal{M}}
\newcommand{\cN}{\mathcal{N}}

\newcommand{\Ex}{\mathbb{E}}

\newcommand{\vNT}{\overline{\otimes}}

\newcommand{\ot}{\otimes}
\newcommand{\tr}{\tau}

\newcommand{\om}{\omega}
\newcommand{\Om}{\Omega}
\newcommand{\si}{\sigma}
\newcommand{\Si}{\Sigma}
\newcommand{\R}{\mathbb{R}}

\newcommand{\bP}{\mathbb{P}}
\newcommand{\bF}{\mathbb{F}}
\newcommand{\E}{\mathbb{E}}

\newcommand{\eps}{\varepsilon}
\newcommand{\ti}{\times}

\newcommand{\Del}{\Delta}

\newcommand{\xspace}{\hbox{\kern-2.5pt}}
\newcommand\tnorm[1]{\left\vert\xspace\left\vert\xspace\left\vert\mskip2mu
#1\mskip2mu \right\vert\xspace\right\vert\xspace\right\vert}

\newcommand{\ud}[0]{\,\mathrm{d}}

\begin{document}

\newtheorem{definition}{Definition}[section]
\newtheorem{theorem}[definition]{Theorem}
\newtheorem{conjecture}[definition]{Conjecture}
\newtheorem{proposition}[definition]{Proposition}
\newtheorem{corollary}[definition]{Corollary}
\newtheorem{remark}[definition]{Remark}
\newtheorem{lemma}[definition]{Lemma}
\newtheorem{assumption}[definition]{Assumption}
\newtheorem{example}{Example}[section]
\newtheorem{exercise}{Exercise}[section]

\title[$L^q$-valued Burkholder-Rosenthal inequalities]{$L^q$-valued Burkholder-Rosenthal inequalities and sharp estimates for stochastic integrals}
\author{Sjoerd Dirksen}
\address{RWTH Aachen University, 52062 Aachen, Germany}
\email{dirksen@mathc.rwth-aachen.de}

\author{Ivan Yaroslavtsev}
\address{Delft Institute of Applied Mathematics\\
Delft University of Technology \\ P.O. Box 5031\\ 2600 GA Delft\\The
Netherlands}
\email{I.S.Yaroslavtsev@tudelft.nl}


\begin{abstract}
We prove sharp maximal inequalities for $L^q$-valued stochastic integrals with respect to any Hilbert space-valued local martingale. Our proof relies on new Burkholder-Rosenthal type inequalities for martingales taking values in an $L^q$-space.
\end{abstract}

\maketitle

\tableofcontents

\section{Introduction}

This work is motivated by the semigroup approach to stochastic partial differential equations. In this approach one first reformulates an SPDE as a stochastic ordinary differential equation in a suitable infinite-dimensional state space $X$ and then establishes existence, uniqueness and regularity properties of a mild solution via a fixed point argument. An important ingredient for this argument is a maximal inequality for the $X$-valued stochastic convolution associated with the semigroup generated by the operator in the stochastic evolution equation. The semigroup approach for equations driven by Gaussian noise in Hilbert spaces is well-established and can be found in \cite{DaZ92}. This theory has more recently been developed in two directions. Firstly, the theory for equations driven by Gaussian noise has been extended to the context of UMD Banach spaces, see e.g.\ \cite{NVW07,NVW12}. In particular, the latter results cover $L^q$-spaces and Sobolev spaces and, as a consequence, allow to achieve better regularity results than the Hilbert space theory. Secondly, there has been increased interested in equations driven by discontinuous noise, e.g.\ Poisson- and L\'evy-type noise \cite{BrH09,FTP10,KoH16,MPR10,MPZ13,MaR10,PeZ07}. The latter results are mostly restricted to the Hilbert space setting. The development of this theory in a non-Hilbertian setting is hindered by the fact that maximal inequalities for vector-valued stochastic convolutions with respect to discontinuous noise are not yet well-understood. In general, only some non-sharp maximal estimates based on geometric assumptions on the Banach space are available \cite{DMN12,ZBH17}. In fact, even the theory for `vanilla' stochastic integrals (corresponding to the trivial semigroup) is incomplete. Sharp maximal inequalities for $L^q$-valued stochastic integrals with respect to Poisson random measures were obtained only recently \cite{Dir14}.\par     
The main purpose of the present paper is to contribute to the foundation of the semigroup approach by proving sharp estimates for $L^q$-valued stochastic integrals with respect to general Hilbert-space valued local martingales. In our main result, Theorem~\ref{thm:mainSI}, we identify a suitable norm $\tnorm{\cdot}_{M,p,q}$ so that, for any elementary predictable processes $\Phi$ with values in the bounded operators from $H$ into $L^q(S)$,
\begin{equation}
\label{eqn:ItoIsomorphism}
c_{p,q} \tnorm{\Phi}_{ M,p,q} \leq \Big(\E\sup_{0\leq s\leq t}\Big\|\int_0^s \Phi \ dM\Big\|_{L^q(S)}^p\Big)^{\frac{1}{p}} \leq C_{p,q} \tnorm{\Phi}_{ M,p,q},
\end{equation}
with universal constants $c_{p,q}$, $C_{p,q}$ depending only on $p$ and $q$. Let us emphasize two important points. Firstly, the norm $\tnorm{\cdot}_{ M,p,q}$ can be computed in terms of \emph{predictable} quantities, which is important for applications. Secondly, we call the estimates in \eqref{eqn:ItoIsomorphism} `sharp' as these inequalities are two-sided and therefore identify the largest possible class of \mbox{$L^p$-sto}\-chas\-ti\-cally integrable processes. We do not require the constants $c_{p,q}$ and $C_{p,q}$ to be sharp or even to depend optimally on $p$ and $q$. For applications to stochastic evolution equations, the precise constants in fact do not play a role. In forthcoming work together with Marinelli \cite{DMY17}, we show that the upper bound \eqref{eqn:ItoIsomorphism} can be transferred to a large class of stochastic convolutions and apply these new estimates to obtain improved well-posedness and regularity results for the associated stochastic evolution equations in $L^q$-spaces.\par 

Let us roughly sketch our approach to \eqref{eqn:ItoIsomorphism}. As a starting point, we use a classical result due to Meyer \cite{Mey76} and Yoeurp \cite{Yoe76} to decompose the integrator as a sum of three local martingales $M=M^c+M^q+M^a$, where $M^c$ is continuous, $M^q$ is purely discontinuous and quasi-left continuous, and $M^a$ is purely discontinuous with accessible jumps. Sharp bounds for stochastic integrals with respect to continuous local martingales were already obtained in a more general setting \cite{VY}.\par 

\smallskip

To estimate the integral with respect to $M^a$ we prove, more generally, sharp bounds for an arbitrary purely discontinuous $L^q$-valued local martingale with accessible jumps in Theorem~\ref{thm:acessjumpsL^q}. To establish this result we first show that such a process can be represented as an essentially discrete object, namely a sum of jumps occurring at predictable times. Using an approximation argument, the problem can then be further reduced to proving \emph{Burkholder-Rosenthal type inequalities} for $L^q$-valued discrete-time martingales. In general, if $1\leq p<\infty$ and $X$ is a Banach space, we understand under Burkholder-Rosenthal inequalities estimates for $X$-valued martingale difference sequences $(d_i)$ of the form
\begin{equation}
\label{eqn:questionBR} c_{p,X} \tnorm{(d_i)}_{p,X} \leq \Big(\Ex\Big\|\sum_i d_i\Big\|_X^p\Big)^{\frac{1}{p}} \leq C_{p,X} \tnorm{(d_i)}_{p,X},
\end{equation}
where $\tnorm{\cdot}_{p,X}$ is a suitable norm on $(d_i)$ which can be computed explicitly in terms of the \emph{predictable moments} of the individual differences $d_i$. In the scalar-valued case, these type of inequalities were proven by Burkholder \cite{Bur73}, following work of Rosenthal \cite{Ros70} in the independent case: for $2\leq p<\infty$
\begin{equation}
\label{eqn:ClassicalBRIntro}
\Big(\E\Big|\sum_{i=1}^n d_i\Big|^p\Big)^{\frac{1}{p}} \eqsim_p \max\Big\{\Big(\sum_{i=1}^n\E|d_i|^p\Big)^{\frac{1}{p}}, \Big(\E\Big(\sum_{i=1}^n \E_{i-1}|d_i|^2\Big)^{\frac{p}{2}}\Big)^{\frac{1}{p}}\Big\}.
\end{equation}
Here we write $A\lesssim_{\alpha} B$ if there is a constant $c_{\alpha}>0$ depending only on $\alpha$ such that $A\leq c_{\alpha} B$ and write $A\eqsim_{\alpha} B$ if both $A\lesssim_{\alpha} B$ and $B\lesssim_{\alpha} A$ hold. To state our $L^q$-valued extension, we fix a filtration $\bF=(\cF_i)_{i\geq 0}$, denote by $(\E_i)_{i\geq 0}$ the associated sequence of conditional expectations and set $\E_{-1}:=\E$. Let $(S, \Sigma,\rho)$ be any measure space. Let us introduce the following norms on the linear space of all finite sequences $(f_i)$ of random variables in $L^{\infty}(\Om;L^q(S))$. Firstly, for $1\leq p,q<\infty$ we set
\begin{equation}
\label{eqn:SpqNorm}
\|(f_i)\|_{S_q^p}  = \Big(\E\Big\|\Big(\sum_i \E_{i-1}|f_i|^2\Big)^{\frac{1}{2}}\Big\|_{L^q(S)}^p\Big)^{\frac{1}{p}},
\end{equation}
From the work of Junge on conditional sequence spaces \cite{Jun02} one can deduce that this expression is a norm, see Section~\ref{sec:NCLq} for more details. We let $S_{q}^p$ denote the completion with respect to this norm. Furthermore, we define
\begin{equation}
\label{eqn:DpqNorms}
\begin{split}
\|(f_i)\|_{D_{q,q}^p} & = \Big(\E\Big(\sum_i \E_{i-1}\|f_i\|_{L^q(S)}^q\Big)^{\frac{p}{q}}\Big)^{\frac{1}{p}}, \\
\|(f_i)\|_{D_{p,q}^p} & = \Big(\sum_i\E\|f_i\|_{L^q(S)}^p\Big)^{\frac{1}{p}}.
\end{split}
\end{equation}
Clearly these expressions define two norms and we let $D_{p,q}^p$ and $D_{q,q}^p$ denote the completions in these norms. Although these spaces depend on the filtration~$\bF$, we will suppress this from the notation. We let $\hat{S}_q^p$, $\hat{D}_{q,q}^p$ and $\hat{D}_{p,q}^p$ denote the closed subspaces spanned by all martingale difference sequences in the above spaces.  
\begin{theorem}
\label{thm:summaryBRIntro} Let $1<p,q<\infty$ and let $S$ be any measure space. If $(d_i)$ is an $L^q(S)$-valued martingale difference sequence, then
\begin{equation}
\label{eqn:summaryBR} \Big(\E\Big\|\sum_i d_i\Big\|_{L^q(S)}^p\Big)^{\frac{1}{p}} \eqsim_{p,q} \|(d_i)\|_{\hat{s}_{p,q}},
\end{equation}
where $\hat{s}_{p,q}$ is given by
\begin{align*}
\hat{S}_q^p\cap \hat{D}_{q,q}^p \cap \hat{D}_{p,q}^p & \ \ \mathrm{if} \ \ 2\leq q\leq p<\infty;\\
\hat{S}_q^p\cap (\hat{D}_{q,q}^p + \hat{D}_{p,q}^p) & \ \ \mathrm{if} \ \ 2\leq p\leq q<\infty;\\
(\hat{S}_q^p \cap \hat{D}_{q,q}^p) + \hat{D}_{p,q}^p & \ \ \mathrm{if} \ \ 1<p<2\leq q<\infty;\\
(\hat{S}_q^p + \hat{D}_{q,q}^p) \cap \hat{D}_{p,q}^p & \ \ \mathrm{if} \ \ 1<q<2\leq p<\infty;\\
\hat{S}_q^p + (\hat{D}_{q,q}^p \cap \hat{D}_{p,q}^p) & \ \ \mathrm{if} \ \ 1<q\leq p\leq 2;\\
\hat{S}_q^p + \hat{D}_{q,q}^p + \hat{D}_{p,q}^p & \ \ \mathrm{if} \ \ 1<p\leq q\leq 2.
\end{align*}
Consequently, if $\cF=\si(\cup_{i\geq 0}\cF_i)$, then the map $f \mapsto (\E_i f - \E_{i-1} f)_{i\geq 0}$ induces an isomorphism between $L^p_0(\Om;L^q(S))$, the subspace of mean-zero random variables in $L^p(\Om;L^q(S))$, and $s_{p,q}$.
\end{theorem}
The result in Theorem~\ref{thm:summaryBRIntro} can be extended to martingales taking values in noncommutative $L^q$-spaces. We discuss this in Section~\ref{sec:NCLq}.\par
Let us say a few words about the proof of Theorem~\ref{thm:summaryBRIntro}. We derive the upper bound in \eqref{eqn:summaryBR} from the known special case that the $d_i$ are independent \cite{Dir14} by applying powerful decoupling techniques due to Kwapie\'{n} and Woyczy\'{n}ski \cite{KwW92}. In the scalar-valued case this route was already traveled by Hitczenko \cite{Hit90} to deduce the optimal order of the constant in the classical Burkholder-Rosenthal inequalities (\ref{eqn:ClassicalBRIntro}) from the one already known for martingales with independent increments. The lower bound in \eqref{eqn:summaryBR} is derived by using a duality argument. For this purpose, we show that for $1<p,q<\infty$ the spaces $s_{p,q}$ satisfy the duality relation
\begin{equation*}
(s_{p,q})^* = s_{p',q'}, \ \tfrac{1}{p} + \tfrac{1}{p'} = 1, \ \tfrac{1}{q}+\tfrac{1}{q'} = 1.
\end{equation*}
The only non-trivial step in proving this duality is to show that $(D_{q,q}^p)^*=D_{q',q'}^{p'}$. In Section~\ref{sec:dualHp} we prove a more general result: we show that if $X$ is a reflexive separable Banach space, then for the space $H^{s_q}_p(X)$ of all adapted $X$-valued sequences $(f_i)$ such that 
$$\|(f_i)\|_{H^{s_q}_p(X)}  = \Big(\E\Big(\sum_i \E_{i-1}\|f_i\|_X^q\Big)^{\frac{p}{q}}\Big)^{\frac{1}{p}}<\infty,$$
the identity $(H^{s_q}_p(X))^* = H^{s_{q'}}_{p'}(X^*)$ holds isomorphically with constants depending only on $p$ and $q$. Somewhat surprisingly, this result only seems to be known in the literature if $X=\R$ and either $1<p\leq q<\infty$ or $2\leq q\leq p<\infty$ (see \cite{Weisz}).\par

\smallskip

Let us now discuss our approach to the integral of $\Phi$ with respect to $M^q$, the purely discontinuous quasi-left continuous part of $M$. We first show that this integral can be represented as an integral with respect to $\bar{\mu}^{M^q}$, the compensated version of the random measure $\mu^{M^q}$ that counts the jumps of $M^q$. In Theorem~\ref{thm:mainintranmeas} we then prove the following sharp estimates for integrals with respect to $\bar{\mu}=\mu-\nu$, where $\mu$ is any integer-valued random measure that has a compensator $\nu$ that is non-atomic in time. This result covers $\mu^{M^q}$ as a special case. To formulate our result, let $(J,\mathcal{J})$ be a measurable space and $\widetilde{\mathcal P}$ be the predictable $\si$-algebra on $\R_+\times \Omega\times J$. For $1<p,q<\infty$ we define the spaces $\hat{\mathcal S}_q^p$, $\hat{\mathcal D}_{q,q}^p$ and $\hat{\mathcal D}_{p,q}^p$ as the Banach spaces of all $\widetilde{\mathcal P}$-measurable functions $F: \mathbb R_+ \times\Omega \times J \to L^q(S)$ for which the corresponding norms
\begin{equation*}
\begin{split}
\|F\|_{\hat{\mathcal S}_q^p} &:= \Bigl( \mathbb E \Bigl\| \Bigl(\int_{\mathbb R_+\times J}|F|^2 \ud \nu \Bigr)^{\frac 12}\Bigr\|^p_{L^q(S)} \Bigr)^{\frac 1p},\\
\|F\|_{\hat{\mathcal D}_{q,q}^p} &:=  \Bigl( \mathbb E \Bigl( \int_{\mathbb R_+\times J}\|F\|^q_{L^q(S)} \ud \nu \Bigr)^{\frac pq} \Bigr)^{\frac 1p},\\
\|F\|_{\hat{\mathcal D}_{p,q}^p} &:=  \Bigl( \mathbb E  \int_{\mathbb R_+\times J}\|F\|^p_{L^q(S)} \ud \nu \Bigr)^{\frac 1p}
\end{split}
\end{equation*}
are finite. 
\begin{theorem}\label{thm:mainintranmeasIntro}
Fix $1<p,q<\infty$. Let $\mu$ be an optional $\widetilde{\mathcal P}$-$\si$-finite random measure on $\mathbb R_+\times J$ and suppose that its compensator $\nu$ is non-atomic in time. Then for any $\widetilde{\mathcal P}$-measurable $F:\mathbb R_+ \times \Omega \times J \to L^q(S)$,
 \begin{equation*}
  \Bigl( \mathbb E \sup_{0\leq s\leq t}\Big\|\int_{[0,s] \times J} F(u,x) \bar{\mu}(\ud u, \ud x)\Big\|^p_{L^q(S)} \Bigr)^{\frac 1p} \eqsim_{p,q}\|F\mathbf 1_{[0,t]}\|_{\mathcal I_{p,q}},
 \end{equation*}
where $\mathcal I_{p,q}$ is given by
\begin{equation*}
 \begin{split}
   \hat{\mathcal S}^p_q \cap \hat{\mathcal D}^p_{q,q} \cap \hat{\mathcal D}^p_{p,q} & \text{ if } \; 2\leq q\leq p<\infty,\\
 \hat{\mathcal S}^p_q \cap (\hat{\mathcal D}^p_{q,q} + \hat{\mathcal D}^p_{p,q}) & \text{ if } \; 2\leq p\leq q<\infty,\\
 (\hat{\mathcal S}^p_q \cap \hat{\mathcal D}^p_{q,q}) + \hat{\mathcal D}^p_{p,q} & \text{ if } \; 1<p<2\leq q<\infty,\\
 (\hat{\mathcal S}^p_q + \hat{\mathcal D}^p_{q,q}) \cap \hat{\mathcal D}^p_{p,q} & \text{ if } \; 1<q<2\leq p<\infty,\\
 \hat{\mathcal S}^p_q + (\hat{\mathcal D}^p_{q,q} \cap \hat{\mathcal D}^p_{p,q}) & \text{ if } \; 1<q\leq p\leq 2,\\
 \hat{\mathcal S}^p_q + \hat{\mathcal D}^p_{q,q} + \hat{\mathcal D}^p_{p,q} & \text{ if } \; 1<p\leq q \leq 2.
 \end{split}
\end{equation*}
\end{theorem}
In the scalar-valued case this result is due to A.A.\ Novikov \cite{Nov75}. In the special case that $\mu$ is a Poisson random measure, Theorem~\ref{thm:mainintranmeas} was obtained in \cite{Dir14}. A very different proof of the upper bounds in Theorem~\ref{thm:mainintranmeas}, based on tools from stochastic analysis, was discovered independently of our work in \cite{Mar13}.\par 
The proof of the upper bounds in Theorem~\ref{thm:mainintranmeasIntro} relies on the Burkholder-Rosenthal inequalities in Theorem~\ref{thm:summaryBRIntro}, a Banach space-valued extension of Novikov's inequality in the special case that $\nu (\mathbb R_+ \times J)\leq 1$ a.s.\ (Proposition~\ref{prop:Nov}), and a time-change argument. For the lower bounds, the non-trivial work is to show that 
$$(\hat{\mathcal S}_q^p)^* = \hat{\mathcal S}_{q'}^{p'}, \qquad (\hat{\mathcal D}_{q,q}^p)^* = \hat{\mathcal D}_{q',q'}^{p'}$$
hold isomorphically with constants depending only on $p$ and $q$. These duality statements are derived in Appendix~\ref{sec:appDuals}.\par

\medskip

Our paper is structured as follows. In Section~\ref{sec:B-RProof} we prove Theorem~\ref{thm:summaryBRIntro}. Section~\ref{sec:dualHp} contains the proof of the duality for the space $H^{s_q}_p(X)$. In Section~\ref{sec:Ito} we prove the sharp bounds \eqref{eqn:ItoIsomorphism}. In particular, Sections~\ref{sec:ItoAJ}, \ref{sec:ItoPDQLC} and \ref{sec:ItoC} are dedicated to integration with respect to local martingales with accessible jumps, purely discontinuous quasi-left continuous local martingales and continuous local martingales, respectively. These three parts can be read independently of each other. Finally, in Section~\ref{sec:NCLq} we discuss the extension of Theorem~\ref{thm:summaryBRIntro} to martingales taking values in a noncommutative $L^q$-space. 

\section{Preliminaries}

Throughout, $(\Om,\cF,\bP)$ denotes a complete probability space. If $X$ and $Y$ are Banach spaces, then $\mathcal{L}(X,Y)$ denotes the Banach space of bounded linear operators from $X$ into $Y$.\par  
In the following, we will frequently use duality arguments for sums and intersections of Banach spaces. Let us recall some basic facts in this direction. If $(X,Y)$ is a compatible couple of Banach spaces, i.e., $X,Y$ are continuously embedded in a Hausdorff topological vector space, then their intersection $X\cap Y$ and sum $X+Y$ are Banach spaces under the norms
$$\|z\|_{X\cap Y} = \max\{\|z\|_{X},\|z\|_{Y}\}$$
and
$$\|z\|_{X+Y} = \inf\{\|x\|_{X} + \|y\|_{Y}:z=x+y, \ x \in X, \ y \in Y\}.$$
If $X\cap Y$ is dense in both $X$ and $Y$, then
\begin{equation}
\label{eqn:sumIntersectionDuality}
(X\cap Y)^* = X^* + Y^*, \qquad (X+Y)^* = X^*\cap Y^*
\end{equation}
hold isometrically. The duality brackets under these identifications are given by
$$\langle x^*,x\rangle = \langle x^*|_{X\cap Y},x\rangle \qquad (x^* \in X^* + Y^*)$$
and
\begin{equation}
\label{eqn:dualofsum}
\langle x^*,x\rangle = \langle x^*,y\rangle + \langle x^*,z\rangle \qquad (x^* \in X^*\cap Y^*, \ x=y+z \in X+Y),
\end{equation}
respectively, see e.g. \cite[Theorem I.3.1]{KPS82}.\par
The following observation facilitates a duality argument that we will use repeatly below. We provide a proof for the convenience of the reader.
\begin{lemma}\label{lemma:jkstaff}
 Let $X$ and $Y$ be Banach spaces, $X$ be reflexive, $U$ be a dense linear subspace of $X$, and let $V$ be a dense linear subspace of $X^*$. Consider $j_0 \in \mathcal L(U,Y)$ and $k_0 \in \mathcal L(V, Y^*)$ so that $\textnormal{ran } j_0$ is dense in $Y$ and $\langle x^*, x\rangle = \langle k_0(x^*), j_0(x)\rangle$ for each $x\in U, x^*\in V$. Then
 \begin{itemize}
  \item[(i)] there exists $j\in\mathcal L(X, Y)$, $k \in \mathcal L(X^*, Y^*)$ such that $j|_{U} = j_0$, $k|_{V} = k_0$,
  \item[(ii)] $\textnormal{ran } j = Y$, $\textnormal{ran } k = Y^*$, in particular $k$ and $j$ are invertible, and
  \item[(iii)] for each $x\in X$ and $x^* \in X^*$
  \begin{equation}\label{eq:equivfornorms}
   \begin{split}
    \frac{1}{\|k\|}\|x\|&\leq \|j(x)\|\leq \|j\|\|x\|,\\
    \frac{1}{\|j\|}\|x^*\|&\leq \|k(x^*)\|\leq \|k\|\|x^*\|.
   \end{split}
  \end{equation}
 \end{itemize}
\end{lemma}
\begin{proof} (i) holds due to the continuity of $j_0$ and $k_0$, and as a consequence
\begin{equation*}
 \langle x^*, x\rangle = \langle k(x^*), j(x)\rangle, \;\; x\in X, x^* \in X^*.
\end{equation*}
 Notice that $j$ and $k$ are embeddings. Indeed, $\langle x^*, x\rangle = \langle k(x^*), j(x)\rangle$ for each $x\in X$, $x^*\in X^*$, so for each nonzero $x\in X$, $x^*\in X^*$ both $j(x)$ and $k(x^*)$ define nonzero linear functionals on $Y^*$ and $Y$ respectively, hence they are nonzero.

For (ii), fix any $y^*\in Y^*$. Since $j\in \mathcal L(X, Y)$, we can define $x^*\in X^*$ by 
$$\langle x^*, x\rangle := \langle y^*, j(x)\rangle, \qquad x\in X.$$ 
Since $\langle x^*, x\rangle = \langle k(x^*), j(x)\rangle$ and hence $\langle y^*-k(x^*),j(x)\rangle=0$ for any $x\in X$, we conclude by density of $\textnormal{ran } j$ that $y^*=k(x^*)$. Thus $\text{ran }k = Y^*$ and $k$ is invertible by the bounded inverse theorem. Using reflexivity of $X$ one can similarly prove the statement for $j$. To prove (iii), we note that for each $x\in X$
\begin{align*}
  \|j(x)\| &= \sup_{x^* \in X^*, \|k(x^*)\|=1}\langle k(x^*), j(x)\rangle = \sup_{x^*\in X^*, \|k(x^*)\|=1}\langle x^*, x\rangle\\
  &\geq \sup_{x^* \in X^*, \|x^*\|=\frac 1{\|k\|}}\langle x^*, x\rangle = \frac 1{\|k\|} \|x\|,
\end{align*}
and obviously $\|j(x)\|\leq \|j\|\|x\|$. The estimates for $k$ are derived similarly.
\end{proof}

\section{$L^q$-valued Burkholder-Rosenthal inequalities} 
\label{sec:B-RProof}

In this section we prove Theorem~\ref{thm:summaryBRIntro}. Our starting point is the following  $L^q$-valued version of the classical Rosenthal inequalities \cite{Ros70}. For $1\leq p,q<\infty$ let $S_q$ and $D_{p,q}$ be the spaces of all sequences of $L^q(S)$-valued random variables such the respective norms
\begin{equation}
\label{eqn:normsRosCom}
\begin{split}
\|(f_i)\|_{S_{q}} & = \Big\|\Big(\sum_i \E|f_i|^2\Big)^{\frac{1}{2}}\Big\|_{L^q(S)},\\
\|(f_i)\|_{D_{p,q}} & = \Big(\sum_i\E\|f_i\|_{L^q(S)}^p\Big)^{\frac{1}{p}}
\end{split}
\end{equation}
are finite. Note that the following result corresponds to a special case of Theorem~\ref{thm:summaryBRIntro}, in which the martingale differences $d_i$ are independent.
\begin{theorem}
\cite{Dir14} \label{thm:summaryRosIntro} Let $1<p,q<\infty$ and let $(S,\Si,\si)$ be a measure space. If $(\xi_i)$ is a sequence of independent, mean-zero random variables taking values in $L^q(S)$, then
\begin{equation}
\label{eqn:summaryRos} \Big(\E\Big\|\sum_i \xi_i\Big\|_{L^q(S)}^p\Big)^{\frac{1}{p}} {\eqsim}_{p,q} \|(\xi_i)\|_{s_{p,q}},
\end{equation}
where $s_{p,q}$ is given by
\begin{align*}
S_{q} \cap D_{q,q} \cap D_{p,q} & \ \ \mathrm{if} \ \ 2\leq q\leq p<\infty;\\
S_{q} \cap (D_{q,q} + D_{p,q}) & \ \ \mathrm{if} \ \ 2\leq p\leq q<\infty;\\
(S_{q} \cap D_{q,q}) + D_{p,q} & \ \ \mathrm{if} \ \ 1<p<2\leq q<\infty;\\
(S_{q} + D_{q,q}) \cap D_{p,q} & \ \ \mathrm{if} \ \ 1<q<2\leq p<\infty;\\
S_{q} + (D_{q,q} \cap D_{p,q}) & \ \ \mathrm{if} \ \ 1<q\leq p\leq 2;\\
S_{q} + D_{q,q} + D_{p,q} & \ \ \mathrm{if} \ \ 1<p\leq q\leq 2.
\end{align*}
Moreover, the estimate $\lesssim_{p,q}$ in (\ref{eqn:summaryRos}) remains valid if $p=1$, $q=1$ or both.
\end{theorem}
To derive the upper bound in Theorem~\ref{thm:summaryBRIntro} we use the following decoupling techniques from \cite{KwW92}. Let $(\Om,\cF,\bP)$ be a complete probability space, let $(\cF_i)_{i\geq 0}$ be a filtration and let $X$ be a (quasi-)Banach space. Two $(\cF_i)_{i\geq 1}$-adapted sequences $(d_i)_{i\geq 1}$ and
$(e_i)_{i\geq 1}$ of $X$-valued random variables are called \emph{tangent} if for every $i\geq 1$ and $A \in \cB(X)$
\begin{equation}
\label{eqn:tangent}
\bP(d_i \in A|\cF_{i-1}) = \bP(e_i \in A|\cF_{i-1}).
\end{equation}
An $(\cF_i)_{i\geq 1}$-adapted sequence $(e_i)_{i\geq 1}$ of $X$-valued random variables is said to satisfy \emph{condition (CI)} if, firstly, there is a
sub-$\si$-algebra $\cG\subset\cF_{\infty}=\si(\cup_{i\geq 0}\cF_i)$ such that for every $i\geq 1$ and $A \in \cB(X)$,
\begin{equation}
\label{eqn:CI}
\bP(e_i \in A|\cF_{i-1}) = \bP(e_i \in A|\cG)
\end{equation}
and, secondly, $(e_i)_{i\geq 1}$ consists of $\cG$-independent random variables, i.e.\ for all $n\geq 1$ and $A_1,\ldots,A_n \in \cB(X)$,
$$\E(\mathbf 1_{e_1 \in A_1}\cdot\ldots\cdot \mathbf 1_{e_n \in A_n}|\cG) = \E(\mathbf 1_{e_1 \in A_1}|\cG)\cdot\ldots\cdot\E(\mathbf 1_{e_n \in A_n}|\cG).$$
It is shown in \cite{KwW92} that for every $(\cF_i)_{i\geq 1}$-adapted sequence $(d_i)_{i\geq 1}$ there exists an $(\cF_i)_{i\geq 1}$-adapted sequence $(e_i)_{i\geq
1}$ on a possibly enlarged probability space which is tangent to $(d_i)_{i\geq 1}$ and satisfies condition (CI). This sequence is called a \emph{decoupled tangent sequence} for $(d_i)_{i\geq 1}$ and is unique in law.\par 
To derive the upper bound in Theorem~\ref{thm:summaryBRIntro} for a given martingale difference sequence $(d_i)_{i\geq 1}$ we apply Theorem~\ref{thm:summaryRosIntro} conditionally to its decoupled tangent sequence $(e_i)_{i\geq 1}$. For this approach to work, we will need to relate various norms on $(d_i)_{i\geq 1}$ and $(e_i)_{i\geq 1}$. One of these estimates can be formulated as a Banach space property. Following~\cite{CoV12}, we say that a (quasi-)Banach space $X$ satisfies \emph{the $p$-decoupling property} if for some $0<p<\infty$ there is a constant $C_{p,X}$ such that for any complete probability space $(\Om,\cF,\bP)$, any filtration $(\cF_i)_{i\geq 0}$, and any $(\cF_{i})_{i\geq 1}$-adapted sequence $(d_i)_{i\geq 1}$ in $L^p(\Om,X)$,
\begin{equation}
\label{eqn:decouple}
\Big(\E\Big\|\sum_{i=1}^n d_i\Big\|_X^p\Big)^{\frac{1}{p}} \leq C_{p,X} \Big(\E\Big\|\sum_{i=1}^n e_i\Big\|_X^p\Big)^{\frac{1}{p}},
\end{equation}
for all $n\geq 1$, where $(e_i)_{i\geq 1}$ is the decoupled tangent sequence of $(d_i)_{i\geq 1}$. It is shown in \cite[Theorem 4.1]{CoV12} that this property is independent of $p$, so we may simply say that $X$ satisfies the \emph{decoupling property} if it satisfies the $p$-decoupling property for some (then all) $0<p<\infty$. Known examples of spaces satisfying the decoupling property are the $L^q(S)$-spaces for any $0<q<\infty$ and UMD Banach spaces. If $X$ is a UMD Banach space, then one can also \emph{recouple}, meaning that for all $1<p<\infty$ there is a constant $c_{p,X}$ such that for any martingale difference sequence $(d_i)_{i\geq 1}$ and any associated decoupled tangent sequence $(e_i)_{i\geq 1}$,
\begin{equation}
\label{eqn:recouple}
\Big(\E\Big\|\sum_{i=1}^n e_i\Big\|_X^p\Big)^{\frac{1}{p}} \leq c_{p,X} \Big(\E\Big\|\sum_{i=1}^n d_i\Big\|_X^p\Big)^{\frac{1}{p}}.
\end{equation}
Conversely, if both (\ref{eqn:decouple}) and (\ref{eqn:recouple}) hold for some (then all) $1<p<\infty$, then $X$ must be a UMD space. This equivalence is independently due to McConnell \cite{McC89} and Hitczenko \cite{HitUP}.\par
To further relate a sequence with its decoupled tangent sequence we use the following technical observation, which is a special case of \cite[Lemma 2.7]{CoV12}.
\begin{lemma}
\label{lem:tanFun} Let $X$ be a (quasi-)Banach space and for every $i\geq 1$ let $h_i:X\rightarrow X$ be a Borel measurable function. Let $(d_i)_{i\geq 1}$ be an $(\cF_i)_{i\geq 1}$-adapted sequence and $(e_i)_{i\geq 1}$ a decoupled tangent sequence. Then $(h_i(e_i))_{i\geq 1}$ is a decoupled tangent sequence for $(h_i(d_i))_{i\geq 1}$.
\end{lemma}
We are now ready to prove the announced result.
\begin{proof}
\emph{(Of Theorem~\ref{thm:summaryBRIntro})}
{\it Step 1: upper bounds.} We will only give a proof in the case $1 \leq q\leq 2\leq p<\infty$. The other cases are proved analogously. Let us write $\E_{\cG} = \E(\cdot|\cG)$ for brevity. By density we may assume that the $d_i$ take values in $L^q(S)\cap L^{\infty}(S)$. Fix an arbitrary decomposition $d_i = d_{i,1} + d_{i,2}$, where $d_{i,1},d_{i,2}$ are $L^q(S)\cap L^{\infty}(S)$-valued martingale difference sequences. Let $e_i = (e_{i,1},e_{i,2})$ be the decoupled tangent sequence for the martingale difference sequence $(d_{i,1},d_{i,2})$ which takes values in $(L^q(S)\cap L^{\infty}(S))\ti (L^q(S)\cap L^{\infty}(S))$. Lemma~\ref{lem:tanFun} implies that $d_{i,\alpha}$ is the decoupled tangent sequence for $e_{i,\alpha}$, $\alpha=1,2$, and $e_{i,1}+e_{i,2}$ is the decoupled tangent sequence for $d_i$. By the decoupling property for $L^q(S)$,
\begin{align*}
\Big(\E\Big\|\sum_i d_i\Big\|_{L^q(S)}^p\Big)^{\frac{1}{p}} & \lesssim_{p,q} \Big(\E\Big\|\sum_i e_{i,1}+e_{i,2}\Big\|_{L^q(S)}^p\Big)^{\frac{1}{p}}.
\end{align*}
Since the summands $e_{i,1}+e_{i,2}$ are $\cG$-conditionally independent and $\cG$-mean zero, we can apply Theorem~\ref{thm:summaryRosIntro} conditionally to find, a.s., 
\begin{align*}
& \Big(\E_{\cG}\Big\|\sum_i e_{i,1}+e_{i,2}\Big\|_{L^q(S)}^p\Big)^{\frac{1}{p}} \\
& \qquad \lesssim_{p,q} \max\Big\{\Big\|\Big(\sum_i \E_{\cG}|e_{i,1}|^2\Big)^{\frac{1}{2}}\Big\|_{L^q(S)} \\
& \qquad \qquad \qquad \qquad + \Big(\sum_i \E_{\cG}\|e_{i,2}\|_{L^q(S)}^q\Big)^{\frac{1}{q}}, \Big(\sum_i
\E_{\cG}\|e_{i,1}+e_{i,2}\|_{L^q(S)}^p\Big)^{\frac{1}{p}}\Big\}.
\end{align*}
Now we take $L^p$-norms on both sides and apply the triangle inequality to obtain
\begin{align*}
& \Big(\E\Big\|\sum_i d_i\Big\|_{L^q(S)}^p\Big)^{\frac{1}{p}} \\
& \qquad \lesssim_{p,q}
\max\Big\{\Big(\E\Big\|\Big(\sum_i \E_{\cG}|e_{i,1}|^2\Big)^{\frac{1}{2}}\Big\|_{L^q(S)}^p\Big)^{\frac{1}{p}} \\
& \qquad \qquad \qquad \qquad + \Big(\E\Big(\sum_i \E_{\cG}\|e_{i,2}\|_{L^q(S)}^q\Big)^{\frac{p}{q}}\Big)^{\frac{1}{p}}, \Big(\sum_i
\E\|e_{i,1}+e_{i,2}\|_{L^q(S)}^p\Big)^{\frac{1}{p}}\Big\} 
\end{align*}
By the properties \eqref{eqn:CI} and \eqref{eqn:tangent} of a decoupled tangent sequence, 
$$\E_{\cG}|e_{i,1}|^2 = \E_{i-1}|e_{i,1}|^2 = \E_{i-1}|d_{i,1}|^2,$$ 
and therefore
$$\Big(\sum_i \E_{\cG} |e_{i,1}|^2\Big)^{\frac{1}{2}} = \Big(\sum_i \E_{i-1}|d_{i,1}|^2\Big)^{\frac{1}{2}} .$$
Similarly,
$$\E_{\cG}\|e_{i,2}\|_{L^q(S)}^q = \E_{i-1}\|d_{i,2}\|_{L^q(S)}^q.$$
We conclude that
\begin{align*}
& \Big(\E\Big\|\sum_i d_i\Big\|_{L^q(S)}^p\Big)^{\frac{1}{p}} \\
& \qquad \lesssim_{p,q}
 \max\Big\{\Big(\E\Big\|\Big(\sum_i \E_{i-1}|d_{i,1}|^2\Big)^{\frac{1}{2}}\Big\|_{L^q(S)}^p\Big)^{\frac{1}{p}} \\
& \qquad \qquad \qquad \qquad + \Big(\E\Big(\sum_i \E_{i-1}\|d_{i,2}\|_{L^q(S)}^q\Big)^{\frac{p}{q}}\Big)^{\frac{1}{p}}, \Big(\sum_i
\E\|d_i\|_{L^q(S)}^p\Big)^{\frac{1}{p}}\Big\}.
\end{align*}
Taking the infimum over all decompositions as above yields the inequality `$\lesssim_{p,q}$' in~(\ref{eqn:summaryBR}).\par
{\it Step 2: lower bounds}. We deduce the lower bounds by duality. Since $(S^p_q)^*=S^{p'}_{q'}$ (by \eqref{eqn:SqdualNC}), $(D^p_{p,q})^*=D^{p'}_{p',q'}$, and $(D^p_{q,q})^*=D^{p'}_{q',q'}$ (by Theorem~\ref{thm:(H^{s_q}_p(X))^*= H^{s_{q'}}_{p'}(X^*)} below) hold isomorphically with constants depending only on $p$ and $q$, it follows from \eqref{eqn:sumIntersectionDuality} that $s_{p,q}^* = s_{p',q'}$ with duality bracket
$$\langle(f_i),(g_i)\rangle = \sum_i \E\langle f_i, g_i\rangle  \qquad ((f_i) \in s_{p,q}, \ (g_i)\in s_{p',q'}).$$
Let $\hat{x}^* \in (\hat{s}_{p,q})^*$. Define the map $P:s_{p,q}\to\hat{s}_{p,q}$ by 
$$P((f_i)) = (\Delta_i f_i),$$
where $\Del_i:=\E_i-\E_{i-1}$. By the triangle inequality and Jensen's inequality one readily sees that $P$ is a bounded projection. As a consequence, we can define $x^*\in s_{p,q}^*$ by $x^*=\hat{x}^*\circ P$. Let $(g_i)\in s_{p',q'}$ be such that 
$$x^*((f_i)) = \sum_i \E\langle f_i, g_i\rangle \qquad ((f_i) \in s_{p,q}).$$
Then, for any $(f_i) \in  \hat{s}_{p,q}$,
$$\hat{x}^*((f_i)) = \sum_i \E\langle f_i, g_i\rangle = \sum_i \E\langle f_i, \Del_i g_i\rangle = \langle(f_i),P(g_i)\rangle.$$
This shows that $(\hat{s}_{p,q})^*=\hat{s}_{p',q'}$ isomorphically. Let $U$ and $V$ be the dense linear subspaces spanned by all finite martingale difference sequences in $\hat{s}_{p,q}$ and $\hat{s}_{p',q'}$, respectively. Define   
$$Y=\overline{\text{span}}\Big\{\sum_i d_i \ : \ (d_i) \in U\Big\}\subset L^p(\Om;L^q(S)).$$
By Step 1, we can define two maps $j_0 \in \mathcal L(U,Y)$, $k_0 \in \mathcal L(V, Y^*)$ by 
$$j_0((d_i))  = \sum_i d_i, \qquad k_0((\tilde{d}_i))  = \sum_i \tilde{d}_i.$$
By the martingale difference property,
\begin{equation}\label{eq:sumofexpect=expectofsum}
 \langle j_0((d_i)), k_0((\tilde{d}_i))\rangle  = \E\Big\langle \sum_i d_i,  \sum_i \tilde{d}_i\Big\rangle = \sum_i \E\langle d_i,\tilde{d}_i\rangle = \langle (d_i), (\tilde{d}_i)\rangle.
\end{equation}
The lower bounds now follow immediately from Lemma~\ref{lemma:jkstaff}.\par
For the final assertion of the theorem, suppose that $\cF=\si(\cup_{i\geq 0}\cF_i)$. Let $f\in L_0^p(\Omega;L^q(S))$ and define $f_n=\E_n f$. Then $\lim_{n\rightarrow \infty} f_n=f$ (see e.g.\ \cite[Theorem 3.3.2]{HNVW1}). Conversely, let $(f_n)_{n\geq 1}$ be a martingale with $\sup_{n\geq 1}\|f_n\|_{L^p(\Om;L^q(S))}<\infty$. By reflexivity of $L^q(S)$ we have $L^p(\Omega;L^q(S)) = (L^{p'}(\Omega;L^{q'}(S)))^*$ and hence its unit ball is weak$^*$-compact. Let $f$ be the weak$^*$-limit of $(f_n)$. It is easy to check that $f_n=\E_n f$. In conclusion, any martingale difference sequence $(d_i)_{i\geq 0}$ of a bounded martingale in $L^p(\Om;L^q(S))$ corresponds uniquely to an $f \in L^p(\Om;L^q(S))$ such that
$$f - \E f = \sum_i d_i, \qquad d_i = \E_i f - \E_{i-1} f.$$
The two-sided inequality (\ref{eqn:summaryBR}) now implies that the map $f\mapsto (\E_i f - \E_{i-1} f)_{i\geq 0}$ is a linear isomorphism between $L^p_0(\Om;L^q(S))$ and $\hat{s}_{p,q}$, with constants depending only on $p$ and $q$.
\end{proof}

\begin{remark}\label{rem:oddmartingales}
 Let $1<p,q<\infty$. Define $\hat S^{p,odd}_q$, $\hat D^{p,odd}_{q,q}$ and $\hat D^{p,odd}_{p,q}$ as the closed subspaces of $\hat S^{p}_q$, $\hat D^{p}_{q,q}$ and $\hat D^{p}_{p,q}$, respectively, spanned by all $L^q$-valued martingale difference sequences $(d_i)_{i\geq 0}$ such that $d_{2i} = 0$ for each $i\geq 0$. By the proof of Theorem~\ref{thm:summaryBRIntro}, any $L^q$-valued martingale difference sequence $(d_i)_{i\geq 0}$ such that $d_{2i}=0$ for each $i\geq 0$ satisfies
 \begin{equation*}
 \Big(\E\Big\|\sum_i d_i\Big\|_{L^q(S)}^p\Big)^{\frac{1}{p}} {\eqsim}_{p,q} \|(d_i)\|_{\hat{s}^{odd}_{p,q}},
\end{equation*}
where $\hat{s}^{odd}_{p,q}$ is given by
\begin{align*}
\hat{S}_q^{p,odd}\cap \hat{D}_{q,q}^{p,odd} \cap \hat{D}_{p,q}^{p,odd} & \ \ \mathrm{if} \ \ 2\leq q\leq p<\infty;\\
\hat{S}_q^{p,odd}\cap (\hat{D}_{q,q}^{p,odd} + \hat{D}_{p,q}^{p,odd}) & \ \ \mathrm{if} \ \ 2\leq p\leq q<\infty;\\
(\hat{S}_q^{p,odd} \cap \hat{D}_{q,q}^{p,odd}) + \hat{D}_{p,q}^{p,odd} & \ \ \mathrm{if} \ \ 1<p<2\leq q<\infty;\\
(\hat{S}_q^{p,odd} + \hat{D}_{q,q}^{p,odd}) \cap \hat{D}_{p,q}^{p,odd} & \ \ \mathrm{if} \ \ 1<q<2\leq p<\infty;\\
\hat{S}_q^{p,odd} + (\hat{D}_{q,q}^{p,odd} \cap \hat{D}_{p,q}^{p,odd}) & \ \ \mathrm{if} \ \ 1<q\leq p\leq 2;\\
\hat{S}_q^{p,odd} + \hat{D}_{q,q}^{p,odd} + \hat{D}_{p,q}^{p,odd} & \ \ \mathrm{if} \ \ 1<p\leq q\leq 2.
\end{align*}
This fact will be used in the proof of Theorem~\ref{thm:acessjumpsinTL^q}.
\end{remark}
\begin{remark}
Let us compare our result to the literature. As was mentioned in the introduction, the scalar-valued version of Theorem~\ref{thm:summaryBRIntro} is due to Burkholder \cite{Bur73}, following work of Rosenthal \cite{Ros70}. A version for noncommutative martingales, as well as a version of (\ref{eqn:ClassicalBRIntro}) for $1<p\leq 2$, was obtained by Junge and Xu \cite{JuX03}. Various upper bounds for the moments of a martingale with values in a uniformly $2$-smooth (or equivalently, cf.\ \cite{Pis75}, martingale type $2$) Banach space were obtained by Pinelis \cite{Pin94}, with constants of optimal order. For instance, if $2\leq p<\infty$ then (\cite{Pin94}, Theorem 4.1)
\begin{equation}
\label{eqn:Pinelis}
\Big(\Ex\Big\|\sum_i d_i\Big\|_X^p\Big)^{\frac{1}{p}} \lesssim p(\E\sup_i\|d_i\|_X^p)^{\frac{1}{p}} + \sqrt{p}\tau_2(X)\Big(\E\Big(\sum_i \E_{i-1}\|d_i\|_X^2\Big)^{\frac{p}{2}}\Big)^{\frac{1}{p}},
\end{equation}
where $\tau_2(X)$ is the $2$-smoothness constant of $X$. As already remarked in \cite{Pin94}, due to the presence of the second term on the right hand side this type of inequality cannot hold in a Banach space which is not $2$-uniformly smooth (or equivalently, has martingale type $2$). On the other hand, one can show that the reverse inequality holds (with different constants) if and only if $X$ is $2$-uniformly convex (or equivalently, has martingale cotype $2$). Thus, a two-sided inequality involving the norm on the right hand side of (\ref{eqn:Pinelis}) can only hold in a space with both martingale type and cotype equal to $2$. Such a space is necessarily isomorphic to a Hilbert space by a~well-known result of Kwapie\'{n} (see e.g.\ \cite{AlK06}, Theorem 7.4.1). 

It should be mentioned that the dependence of the implicit constants on $p$ and $q$ in \eqref{eqn:summaryBR} is not optimal. We leave it as an interesting open problem to determine the optimal dependence on the constants.
\end{remark}

\section{The dual of $H^{s_q}_p(X)$}
\label{sec:dualHp}

In the proof of Theorem~\ref{thm:summaryBRIntro} we used the fact that $(D^p_{q,q})^*=D^{p'}_{q',q'}$ holds isomorphically (with constants depending only on $p$ and $q$) for all $1<p,q<\infty$. In this section we will prove a more general statement.\par
Let $(\Omega, \mathcal F, \mathbb P)$ be a probability space with a filtration $\mathbb F = (\mathcal F_k)_{k\geq 0}$, $X$ be a Banach space, and let $1<p,q<\infty$. For an adapted sequence $f=(f_k)_{k\geq 0}$ of $X$-valued random variables we define
$$
s_q^n(f):=\Bigl( \sum_{k=0}^n\mathbb E_{k-1}\|f_k\|^q \Bigr)^{1/q},\;\;\;s_q(f):=\Bigl( \sum_{k=0}^{\infty}\mathbb E_{k-1}\|f_k\|^q \Bigr)^{1/q},
$$
where $\mathbb E_{k} = \mathbb E(\cdot|\mathcal F_k)$, $\mathbb E_{-1} = \mathbb E$. We let $H^{s_q}_p(X)$ be the space of all adapted sequences $f=(f_k)_{k\geq 0}$ satisfying
$$
\|f\|_{H^{s_q}_p(X)}:= (\mathbb E s_q(f)^p)^{1/p}<\infty.
$$
Similarly we define $H^{s^n_q}_p(X)$. We will prove the following result, which was only known before if $X=\R$ and either $1<p\leq q<\infty$ or $2\leq q\leq p<\infty$ (see \cite[Theorem 15]{Weisz} and the remark following it).
\begin{theorem}\label{thm:(H^{s_q}_p(X))^*= H^{s_{q'}}_{p'}(X^*)}
 Let $X$ be a reflexive separable Banach space, $1<p,q<\infty$. Then $\bigl(H^{s_q}_p(X)\bigr)^* = H^{s_{q'}}_{p'}(X^*)$ isomorphically. The isomorphism is given by
 \begin{equation}\label{eq:isom}
  g\mapsto F_g,\;\;\; F_g(f) = \mathbb E \Bigl(\sum_{k=0}^{\infty} \langle f_k,g_k\rangle\Bigr)\;\;\; \bigl(f\in H^{s_q}_p(X), g\in H^{s_{q'}}_{p'}(X^*)\bigr),
 \end{equation}
 and
  \begin{equation}\label{eq:estonnorminfinity}
  \min\Bigl\{\frac{q}{p},\frac{q'}{p'}\Bigr\}\|g\|_{H^{s_{q'}}_{p'}(X^*)} \leq\|F_g\|_{(H^{s_q}_p(X))^*}\leq\|g\|_{H^{s_{q'}}_{p'}(X^*)}.
 \end{equation}
In particular, $H^{s_q}_p(X)$ is a reflexive Banach space.
\end{theorem}
To prove this result, we will first extend an argument of Cs\"org{\H{o}} \cite{Cso} to show that $\bigl(H_p^{s_q^n}(X)\bigr)^*$ and $H^{s_{q'}^n}_{p'}(X^*)$ are isomorphic if $1<p,q<\infty$, with isomorphism constants depending on $p,q$ and $n$. In particular, this shows that $H_p^{s_q^n}(X)$ is reflexive. In a second step, we exploit this reflexivity to show that the isomorphism constants do not depend on $n$. The proof of this result, Theorem~\ref{thm:findimrefl}, relies on an argument of Weisz \cite{Weisz}. After this step, it is straightforward to deduce Theorem~\ref{thm:(H^{s_q}_p(X))^*= H^{s_{q'}}_{p'}(X^*)}.\par
We start by introducing an operator that serves as a replacement for the sign-function in a vector-valued context.
\begin{lemma}\label{lemma:P_eps}
 Let $X$ be a Banach space with a separable dual. Fix $\eps>0$. Then there \mbox{exists} a discrete-valued Borel-measurable function $P_{\eps} :X^* \to X$ such that \mbox{$\|P_{\eps}(x^*)\| = 1$} and 
 \begin{equation}\label{eq:lem:P_eps}
  (1-{\eps})\|x^*\|\leq \langle P_{\eps}x^*,x^*\rangle \leq \|x^*\|
 \end{equation}
 for each $x^*\in X^*$.
\end{lemma}
\begin{proof}
Let $(x_n^*)_{n\geq 1}$ be a dense subset of the unit sphere of $X^*$. For each $n\geq 1$ define $U_n = U \bigcap B(x_n^*,\frac{\eps}{2})$, where $B(y^*,r)$ denotes the ball in $X^*$ with radius~$r$ and center $y^*$. Define $V_1 = U_1$ and
$$V_n = U_n\setminus \Big(\bigcup_{k=1}^{n-1} V_{k}\Big), \qquad n\geq 2.$$ 
For each $n\geq 1$ one can find an $x_n \in X$ such that $\|x_n\|=1$ and $\langle x_n, x_n^*\rangle\geq 1-\frac{\eps}{2}$. Now define 
 $$
 P_{\eps}(x^*) := \sum_{n=1}^{\infty} \mathbf 1_{V_n}\Bigl(\frac{x^*}{\|x^*\|}\Bigr)x_n,\;\;\; x^* \in X^*.
 $$ 
 This function is Borel since the $V_n$ are Borel sets. As the $V_n$ form a disjoint cover of the unit sphere, for every $x^*\in X^*$ there exists a unique $n=n(x^*)$ so that $x^*/\|x^*\| \in V_n$. Hence, $\|P_{\eps}(x^*)\| = 1$ and
 \[
  \langle P_{\eps}(x^*), x^*\rangle = \|x^*\|\Bigl\langle x_n, \frac{x^*}{\|x^*\|}\Bigr\rangle \geq \|x^*\|\langle x_n, x_n^*\rangle - \frac{\eps}2 \|x^*\| \geq (1-\eps) \|x^*\|,
 \]
 so \eqref{eq:lem:P_eps} follows.
\end{proof}
\begin{theorem}
\label{thm:CsorgoArg}
  Let $X$ be a reflexive separable Banach space, $1<p,q<\infty$, $n\geq 0$. Then $\bigl(H^{s_q^n}_p(X)\bigr)^* = H^{s_{q'}^n}_{p'}(X^*)$ isomorphically (with constants depending on $p,q$ and~$n$). The isomorphism is given by
 \begin{equation}
  g\mapsto F_g,\;\;\; F_g(f) = \mathbb E \Bigl(\sum_{k=0}^{n} \langle f_k,g_k\rangle\Bigr)\;\;\; \bigl(f\in H^{s_q^n}_p(X), g\in H^{s_{q'}^n}_{p'}(X^*)\bigr).
 \end{equation}
In particular, $H^{s_q^n}_p(X)$ is a reflexive Banach space. 
\end{theorem}
\begin{proof}
The main argument is inspired by the proof of \cite[Theorem 1]{Cso}. By the conditional H\"{o}lder inequality and the usual version of H\"{o}lder's inequality,
 \begin{equation}\label{eq:proofofthm:CsorgoArg}
  \begin{split}
     |F_g(f)| &\leq \mathbb E \Bigl(\sum_{k=0}^{n} \mathbb E_{k-1}(\|f_k\|\|g_k\|)\Bigr)\\
   &\leq \mathbb E\Bigl(\sum_{k=0}^{n} (\mathbb E_{k-1}\|f_k\|^q)^{1/q}(\mathbb E_{k-1}\|g_k\|^{q'})^{1/{q'}}\Bigr)\\
   &\leq \|f\|_{H^{s_q^n}_p(X)}\|g\|_{H^{s_{q'}^n}_{p'}(X^*)}.
  \end{split}
 \end{equation}
Hence, the functional $F_g$ is bounded and $\|F_g\|\leq \|g\|_{H^{s_{q'}^n}_{p'}(X^*)}$. 

To prove that $\|F_g\|\gtrsim_{p,q,n}\|g\|_{H^{s_{q'}^n}_{p'}(X^*)}$ we need to construct an appropriate $f \in H^{s_q^n}_p(X)$ with 
$$\|f\|_{H^{s_q^n}_p(X)}{{\lesssim}}_{p,q,n} 1, \qquad \langle F_g,f\rangle\gtrsim_{p,q,n} \|g\|_{H^{s_{q'}^n}_{p'}(X^*)}.$$
Fix $0<\eps<1$. We define $f$ by setting 
$$
f_k:=P_{\eps} g_k \frac {\|g_k\|^{q'-1}}{\|g\|^{p'-1}_{H^{s_{q'}^n}_{p'}(X^*)}} (\mathbb E_{k-1}\|g_k\|^{q'})^{\frac{p'-q'}{q'}},\qquad 0\leq k\leq n
$$
where $P_{\eps}$ is as in Lemma~\ref{lemma:P_eps}. Using $pp'=p+p'$ and $qq'=q+q'$ we find
\begin{align*}
 \|f\|_{H^{s_q^n}_p(X)}^p  &= \mathbb E \Bigl( \sum_{k=0}^n\mathbb E_{k-1}\|f_k\|^q \Bigr)^{p/q} { =} \frac{1}{\|g\|^{p(p'-1)}_{H^{s_{q'}^n}_{p'}(X^*)}}\mathbb E \Bigl( \sum_{k=0}^n(\mathbb E_{k-1}\|g_k\|^{q'})^{{ \frac {(p'-1)q}{q'}}} \Bigr)^{\frac pq}\\
 &\eqsim_{n,p,q}\frac{1}{\|g\|^{p'}_{H^{s_{q'}^n}_{p'}(X^*)}}\mathbb E \Bigl( \sum_{k=0}^n(\mathbb E_{k-1}\|g_k\|^{q'}) \Bigr)^{\frac{p'}{q'}} = 1,
\end{align*}
so $f\in H^{s_q^n}_p(X)$.
Moreover,
\begin{align*}
 \langle F_g,f\rangle & \geq (1-\eps) \frac {1}{\|g\|^{p'-1}_{H^{s_{q'}^n}_{p'}(X^*)}}\mathbb E \sum_{k=0}^n \|g_k\|^{q'} (\mathbb E_{k-1}\|g_k\|^{q'})^{\frac{p'-q'}{q'}}\\
 & = (1-\eps) \frac {1}{\|g\|^{p'-1}_{H^{s_{q'}^n}_{p'}(X^*)}}\mathbb E \sum_{k=0}^n (\mathbb E_{k-1}\|g_k\|^{q'})^{\frac{p'}{q'}}\\
 &\eqsim_{p,q,n} (1-\eps) \frac {1}{\|g\|^{p'-1}_{H^{s_{q'}^n}_{p'}(X^*)}}\mathbb E  \Big( \sum_{k=0}^n \mathbb E_{k-1}\|g_k\|^{q'}\Big)^{\frac{p'}{q'}}
  = \|g\|_{H^{s_{q'}^n}_{p'}(X^*)},
\end{align*}
as desired, since $\eps$ was arbitrary and can be chosen, say, $\frac 12$.

Now we will show that every $F \in \bigl(H^{s_q^n}_p(X)\bigr)^*$ is equal to $F_g$ for a suitable $g \in H^{s_{q'}^n}_{p'}(X^*)$. For this purpose we consider the disjoint direct sum of $(\Om,\cF_k,\bP)$, $k=0,\ldots,n$. Formally, we set $\Om_k=\Om\times \{k\}$, $\widetilde{\cF}_k=\cF_k\times \{k\}$ and define a probability measure $\bP_k$ on $\widetilde\cF_k$ by $\bP_k(A\times\{k\})=\bP(A)$. Now the disjoint direct sum $(\Om^n,\cF^n,\bP^n)$ is defined by 
$$\Om^n = \bigcup_{k=0}^n \Om_k, \qquad \cF^n=\{A\in \Om^n \ : \ A\cap\Om_k\in \widetilde{\cF}_k,\text{ for all } 1\leq k\leq n\}$$
and 
$$\bP^n(A) = \sum_{k=0}^N \bP_k(A\cap \Om_k), \qquad A\in \cF^n.$$
Let $P_k:(\Om,\cF_k)\to (\Om^n,\cF^n)$, $P_k(\om)=(\om,k)$, be the measurable bijection between $(\Om,\cF_k)$ and its disjoint copy.  We can now define an $X^*$-valued set function $\mu$ by
$$
\langle \mu(A),x \rangle :=F\bigl((x \cdot \mathbf 1_{P_k^{-1}(A\cap \Omega_k)})_{k=0}^n\bigr), \qquad A\in \cF^n, \ x\in X. 
$$
We will show that $\mu$ is $\si$-additive, absolutely continuous with respect to $\bP^n$ and of finite variation. Let us first show that $\mu$ is of finite variation. Let $(A_m)_{m=1}^M \subset \mathcal F^n$ be disjoint such that $\cup_m A_m = \Omega^n$. Then
\begin{align*}
 \sum_{m=1}^M \|\mu({A_m})\| 
 &=\sum_{m=1}^M \sup_{x_m\in X:\|x_m\|=1}  F\bigl((x_m \cdot \mathbf 1_{P_k^{-1}(A_m\cap \Omega_k)})_{k=0}^n\bigr)\\
 &= \sup_{(x_m)_{m=1}^M\subset X:\|x_m\|=1} \sum_{m=1}^M F\bigl((x_m \cdot \mathbf 1_{P_k^{-1}(A_m\cap \Omega_k)})_{k=0}^n\bigr)\\
 &= \sup_{(x_m)_{m=1}^M\subset X:\|x_m\|=1}F\Bigl(\Bigl( \sum_{m=1}^M  x_m \cdot \mathbf 1_{P_k^{-1}(A_m\cap \Omega_k)}\Bigr)_{k=0}^n\Bigr)\\
 &\leq \|F\| \sup_{(x_m)_{m=1}^M\subset X:\|x_m\|=1}\Bigl\|\Bigl( \sum_{m=1}^M  x_m \cdot \mathbf 1_{P_k^{-1}(A_m\cap \Omega_k)}\Bigr)_{k=0}^n\Bigr\|_{H^{s_q^n}_p(X)}\\
 &= \|F\|\sup_{(x_m)_{m=1}^M\subset X:\|x_m\|=1}\Bigl(\mathbb E\Bigl(\sum_{k=0}^n \mathbb E_{k-1} \Big\|\sum_{m=1}^M x_m \mathbf 1_{P_k^{-1}(A_m\cap \Omega_k)}\Big\|^q\Bigr)^{\frac{p}{q}}\Bigr)^{\frac 1p}\\
  &\leq \|F\| \biggl[\mathbb E\Bigl(\sum_{k=0}^n \mathbb E_{k-1} \Big(\sum_{m=1}^M\mathbf 1_{P_k^{-1}(A_m\cap \Omega_k)}\Big)^q\Bigr)^{\frac{p}{q}}\biggr]^{\frac 1p}\\
  &=\|F\| \Bigl(\mathbb E\Bigl(\sum_{k=0}^n \mathbb E_{k-1} \mathbf 1_{\Omega}\Bigr)^{\frac{p}{q}}\Bigr)^{\frac 1p}= \|F\| (n+1)^{\frac 1q}
\end{align*}
Now let us prove the $\sigma$-additivity. Obviously $\mu$ is additive. Let $(A_m)_{m\geq 0}\subset \mathcal F_n$ be such that $A_m \searrow \varnothing$. Then
\begin{align*}
 \|\mu(A_m)\| &= \sup_{x\in X:\|x\|=1}|F((x \cdot \mathbf 1_{P_k^{-1}(A_m\cap \Omega_k)})_{k=0}^n)|\\
 &\leq \|F\|\sup_{x\in X:\|x\|=1}\|(x \cdot \mathbf 1_{P_k^{-1}(A_m\cap \Omega_k)})_{k=0}^n\|_{H^{s_q^n}_p(X)}\\
 & = \|F\|\Bigl(\mathbb E\Bigl(\sum_{k=0}^n \mathbb E_{k-1}\mathbf 1_{P_k^{-1}(A_m\cap \Omega_k)}\Bigr)^{\frac{p}{q}}\Bigr)^{\frac 1p}\to 0 \text{ as } m \to \infty,
\end{align*}
by the monotone convergence theorem. This computation also shows that $\mu$ is absolutely continuous with respect to $\bP^n$.

Since $X$ is reflexive, $X^*$ has the Radon-Nikodym property (see e.g.\ \cite[Theorem~1.3.21]{HNVW1}). Thus, there exists a $g\in L^1(\Om^n;X^*)$ so that 
$$\mu(A) = \int_A g \ud \mathbb P^n = \sum_{k=0}^n  \int_{A\cap \Om_k} g \ud \mathbb P_k.$$
If we now define $g_k:=g\circ P_k$ then $g_k$ is $\cF_k$-measurable and 
$$\mu(A) = \sum_{k=0}^n  \int_{P_k^{-1}(A\cap \Om_k)} g_k \ud \mathbb P.$$
It now follows for $f=(f_k)_{k=0}^n \in H^{s_q^n}_p(X)$ with $f_k$ bounded for all $k=0,\ldots,n$ that 
\begin{equation}\label{eq:functional}
 F(f) = F_g(f)= \mathbb E\sum_{k=0}^n \langle f_k,g_k\rangle.
\end{equation}
Now fix a general $f \in H^{s_q^n}_p(X)$. Fix $0<\eps<1$ and let $h:=(h_k)_{k=0}^n = (\|f_k\|P_{\eps} g_k)_{k=0}^n$. Define $h^m:=(h^m_k)_{k=0}^n = (h_k \mathbf 1_{\|h_k\|\leq m})_{k=0}^n$ for each $m\geq 1$. Then formula \eqref{eq:functional} holds for $h^m$. But $F(h^m)\to F(h)$ as $m$ goes to infinity, so by the monotone convergence theorem $F(h) = \mathbb E\sum_{k=0}^n \langle h_k,g_k\rangle$. This shows that
\begin{equation}
\label{eqn:fkgkBd}
\E\sum_{k=0}^n |\langle f_k,g_k\rangle| \leq \E\sum_{k=0}^n \|f_k\|\|g_k\|\leq (1-\eps)^{-1} \E\sum_{k=0}^n\langle h_k,g_k\rangle<\infty.
\end{equation}
Now consider $f^m := (f^m_k)_{k=0}^n = (f_k \mathbf 1_{\|f_k\|\leq m})_{k=0}^n$. Since \eqref{eq:functional} holds for $f^m$ and $F(f^m)\to F(f)$, we can use \eqref{eqn:fkgkBd} and the dominated convergence theorem to conclude that $f$ satisfies \eqref{eq:functional}.

It remains to prove that $g \in H^{s_{q'}^n}_{p'}(X^*)$. For each $m\geq 1$ we consider the approximation $g^m := (g_k\mathbf 1_{\|g_k\|\leq m})_{k=0}^n$. Then $\|g^m\|_{H^{s_{q'}^n}_{p'}(X^*)}\lesssim_{p,q,n} \|F_{g^m}\|\leq \|F\|$. Therefore by the monotone convergence theorem $\|g\|_{H^{s_{q'}^n}_{p'}(X^*)}\lesssim_{p,q,n} \|F\|$.
\end{proof}
One can easily show the following simple lemma.
\begin{lemma}\label{lemma:simple}
 Let $X$ and $Y$ be reflexive Banach spaces such that $X^*$ is isomorphic to $Y$ and
 $$
 a\|x^*\|_Y \leq \|x^*\|_{X^*}\leq b\|x^*\|_{Y},\;\;\; x^* \in X^*.
 $$
 Then $Y^*$ is isomorphic to $X^{**} = X$ and
 $$
 a\|x\|_{X}\leq \|x\|_{Y^*}\leq b\|x\|_X,\;\;\; x\in X.
 $$
\end{lemma}
\begin{theorem}\label{thm:findimrefl}
 Let $X$ be a reflexive separable Banach space, $1<p,q<\infty$, $n\geq 0$. Then 
 \begin{equation}\label{eq:estonnorm}
  \min\Bigl\{\frac{q}{p},\frac{q'}{p'}\Bigr\}\|g\|_{H^{s_{q'}^n}_{p'}(X^*)} \leq\|F_g\|_{\bigl(H^{s_q^n}_p(X)\bigr)^*}\leq\|g\|_{H^{s_{q'}^n}_{p'}(X^*)}.
 \end{equation}
 \end{theorem}
\begin{proof}
We already proved in Theorem~\ref{thm:CsorgoArg} that $H^{s_q^n}_p(X)$ is reflexive, so by Lemma \ref{lemma:simple} it is enough to show \eqref{eq:estonnorm} for $p\leq q$. It was already noted in \eqref{eq:proofofthm:CsorgoArg} that $\|F_g\|\leq \|g\|_{H^{s_{q'}^n}_{p'}(X^*)}$.
It is sufficient to show \eqref{eq:estonnorm} for a bounded $g$. The following construction is in essence the same as in \cite[Theorem 15]{Weisz}. Set
$$
(v_k)_{k=0}^n =\Biggl(\frac{(s_{q'}^{k}(g))^{p'-q'}}{\|g\|^{p'-1}_{H^{s_{q'}^n}_{p'}(X^*)}}\Biggl)_{k=0}^n.
$$
Fix $0<\eps<1$. Let us define $h \in H^{s_q^n}_p(X)$ by setting
$$
h_k = v_k\|g_k\|^{q'-1}P_{\eps}g_k,
$$
where $P_{\eps}:X^*\to X$ is as given in Lemma \ref{lemma:P_eps}. Then
$$
(s_q^n(h))^q \leq \sum_{k=0}^n \frac {(s_{q'}^{k}(g))^{qp'-qq'}}{\|g\|^{qp'-q}_{H^{s_{q'}^n}_{p'}(X^*)}}\mathbb E_{k-1}\|g_k\|^{q'} \leq \frac {(s_{q'}^{n}(g))^{qp'-(q-1)q'}}{\|g\|^{qp'-q}_{H^{s_{q'}^n}_{p'}(X^*)}}.
$$
and therefore
$$
\mathbb E(s_q^n(h))^p  \leq \frac {\mathbb E(s_{q'}^{n}(g))^{(qp'-(q-1)q')\frac p{q}}}{\|g\|^{pp'-p}_{H^{s_{q'}^n}_{p'}(X^*)}} = 1.
$$
As a consequence,
\begin{equation}\label{eq:thmforn}
 \begin{split}
  \|F_g\| &\geq |\langle F_g, h\rangle| \\
  &\geq (1-\eps) \frac{1}{\|g\|^{p'-1}_{H^{s_{q'}^n}_{p'}(X^*)}}\mathbb E \sum_{k=0}^n (s_{q'}^{k}(g))^{p'-q'} \mathbb E_{k-1}\|g_k\|^{q'}\\
 & = (1-\eps) \frac{1}{\|g\|^{p'-1}_{H^{s_{q'}^n}_{p'}(X^*)}}\mathbb E \sum_{k=0}^n (s_{q'}^{k}(g))^{p'-q'}((s_{q'}^{k}(g))^{q'}-(s_{q'}^{k-1}(g))^{q'}).
 \end{split}
\end{equation}
By the mean value theorem,
\begin{equation}
\label{eqn:MVTalpha}
x^{\alpha} - 1\leq \alpha (x-1)x^{\alpha -1},\;\;\;\; x,\alpha \geq 1.
\end{equation}
Applying this for $x = \frac{(s_{q'}^{k}(g))^{q'}}{(s_{q'}^{k-1}(g))^{q'}}\geq 1$ and $\alpha = \frac {p'}{q'}\geq 1$ we find
$$
\frac{q'}{p'}((s_{q'}^{k}(g))^{p'}-(s_{q'}^{k-1}(g))^{p'}) \leq ((s_{q'}^{k}(g))^{q'}-(s_{q'}^{k-1}(g))^{q'})(s_{q'}^{k}(g))^{p'-q'}.
$$
Combining this with \eqref{eq:thmforn} and letting $\eps \to 0$,
$$
\|F_g\|\geq \frac{q'}{p'\|g\|^{p'-1}_{H^{s_{q'}^n}_{p'}(X^*)}} \mathbb E(s_{q'}^{n}(g))^{p'} = \frac{q'}{p'} \|g\|_{H^{s_{q'}^n}_{p'}(X^*)}.
$$
\end{proof}
We can now deduce the main result of this section.
\begin{proof}[Proof of Theorem~\ref{thm:(H^{s_q}_p(X))^*= H^{s_{q'}}_{p'}(X^*)}]
Let $F \in (H^{s_q}_p(X))^*$. For every $n\geq 0$ there exists an 
$F_n \in \bigl(H^{s^n_q}_p(X)\bigr)^*$ such that $\langle F, f\rangle = \langle F_n, (f_k)_{k=0}^n\rangle$ for each $f \in H^{s_q}_p(X)$ satisfying $f_m = 0$ for all $m> n$. Thanks to Theorem~\ref{thm:CsorgoArg}, for each $n\geq 0$ there exists a $g^n = (g^n_k)_{k=0}^n \in H^{s^n_{q'}}_{p'}(X)$ such that $F_n = F_{g^n}$. Obviously $g_{k}^m = g_k^n$ for each $m,n\geq k$, so there exists a unique $g = (g_k)_{k=0}^{\infty}$ such that $g^n = (g_k)_{k=0}^n$. Moreover, Theorem \ref{thm:findimrefl} implies
 $$
 \min\Bigl\{\frac{q}{p},\frac{q'}{p'}\Bigr\}\|g^n\|_{H^{s^n_{q'}}_{p'}(X)} \leq \|F_n\|_{\bigl(H^{s^n_q}_p(X)\bigr)^*} \leq \|F\|_{(H^{s_q}_p(X))^*},
 $$
 so $g\in H^{s_{q'}}_{p'}(X)$ and 
 $$
 \min\Bigl\{\frac{q}{p},\frac{q'}{p'}\Bigr\}\|g\|_{H^{s_{q'}}_{p'}(X)}\leq \|F\|_{(H^{s_q}_p(X))^*}.
 $$
Now obviously $F = F_g$, as these two functionals coincide on the dense subspace of all finitely non-zero sequences in $H^{s_q}_p(X)$, and \eqref{eq:isom} and \eqref{eq:estonnorminfinity} hold. 
\end{proof}

\section{Sharp bounds for $L^q$-valued stochastic integrals}
\label{sec:Ito}

We now turn to proving sharp bounds for stochastic integrals. We start by setting notation and recalling some basic facts on martingales. 

\subsection{Preliminaries} Throughout, $H$ always denotes a Hilbert space. We write $\overline{\mathbb R}_+:= [0,+\infty]$. We let $(\Omega,\cF,\bP)$ be a complete probability space and let $\mathbb F = (\mathcal F_t)_{t\geq 0}$ be a filtration that satisfies the usual conditions. We write $\cP$ to denote the \emph{predictable $\si$-algebra} on $\R_+\ti \Om$, the $\si$-algebra generated by all c\`ag adapted processes. We use $\mathcal O$ to denote the {\it optional $\sigma$-algebra} $\mathbb R_+ \times \Omega$, the $\sigma$-algebra generated by all c\`adl\`ag adapted processes. We write $\widetilde{\mathcal P}= \mathcal P \otimes \mathcal J$ and $\widetilde {\mathcal O}:=\mathcal O \otimes \mathcal J$ for the induced $\sigma$-algebras on $\widetilde {\Omega} = \mathbb R_+ \times \Omega \times J$.

Let $M$, $N:\mathbb R_+ \times \Omega \to H$ be local martingales. The {\it covariation} $[M,N]:\mathbb R_+ \times \Omega \to \mathbb R$ and the {\it quadratic variation} $[M]:\mathbb R_+ \times \Omega \to \mathbb R_+$ are defined by
\begin{align*}
 [M,N]_t & = \mathbb P-\lim \sum_{i=1}^k \langle M_{t_i} - M_{t_{i-1}},N_{t_i} - N_{t_{i-1}} \rangle,\;\;\;t\geq 0,\\
 [M]_t &= [M,M]_t\;\;\;t\geq 0,
\end{align*}
where $0= t_0\leq t_1\leq \cdots \leq t_k = t$ is a partition and the limit in probability is taken as the mesh of the partition goes to $0$. It is well known that the covariation of any two local martingales exists and that $\langle M,N\rangle-[M,N]$ is a local martingale. Notice that for any orthogonal basis $(h_n)_{n\geq 1}$ of $H$, for any $t\geq 0$ a.s.
 \begin{equation}\label{eq:quadvarcoorwise}
  [M]_t = \sum_{n\geq 1} [\langle M, h_n\rangle]_t.
 \end{equation}

We will frequently use the \emph{Burkholder-Davis-Gundy inequality}: if $1\leq p<\infty$, $M$ is any local martingale with $M_0=0$ and $\tau$ is any stopping time, then 
$$\bigl(\E\sup_{0\leq t\leq \tau} \|M_t\|^p\bigr)^{1/p} \eqsim_p (\E[M]_\tau^{p/2})^{1/p}.$$
We refer to \cite{MarRock16} for a self-contained proof.

If $M:\mathbb R_+ \times \Omega\to H$ is a local martingale, then it has a c\`adl\`ag version \cite[Theorem I.9]{Prot}. Therefore a.s.\ for each stopping time~$\tau$ one can define the jump of $M$ at time $\tau$ by $\Delta M_{\tau}:=M_{\tau} - \lim_{\eps \to 0}M_{({\tau}-\eps)\vee 0}$. 

An increasing c\`adl\`ag process $A:\mathbb R_+ \times \Omega \to \mathbb R$ is called {\it pure jump} if for each $t\geq 0$ a.s.
\[
 A_t = A_0 + \sum_{0\leq s\leq t} \Delta A_s.
\]

An $H$-valued local martingale $M$ is called {\it purely discontinuous} if 
$[M]$ is a pure jump process a.s.
(An equivalent definition of a purely discontinuous local martingale will be given in Proposition \ref{thm:purdiscorthtoanycont1}.) The reader can find more on purely discontinuous local martingales in \cite[Chapter I.4]{JS} and \cite[Chapter 26]{Kal}. 

Let $\tau$ be a stopping time. 
 An $H$-valued local martingale is called {\it quasi-left continuous} if $\Delta M_{\tau} = 0$ a.s.\ on the set $\{\tau<\infty\}$ for each predictable stopping time~$\tau$ (see \cite[Chapter I.2]{JS} for more information). We call
$$[\tau]=\{((\omega,t) \in \Omega\times \R_+ \ : \ t=\tau(\omega)\}$$
the graph of $\tau$ (although it is strictly speaking, the restriction of the graph of $\tau$ to $\Omega\times \R_+$). A stopping time $\tau$ is called {\it predictable} if there exists an increasing sequence $(\tau_n)_{n\geq 1}$ of stopping times such that $\tau_n < \tau$ on $\{\tau>0\}$ for each $n\geq 1$ and $\tau_n \nearrow \tau$ a.s.\ as $n\to \infty$
(see \cite[Definition I.2.7]{JS} and \cite[Chapter 25]{Kal}). For a predictable stopping time $\tau$ we define $\mathcal F_{\tau-}$ by
\[
 F_{\tau-} := \sigma (\mathcal F_{\tau_n})_{n\geq 1}.
\]
Notice that $F_{\tau-}$ does not depend on the choice of the announcing sequence $(\tau_n)_{n\geq 0}$ (see for example \cite[Lemma 25.2(iii)]{Kal}). \mbox{An~$H$-va}\-lued local martingale is said to have {\it accessible jumps} if there exists a sequence of predictable stopping times $({\tau_n})_{n\geq 0}$ with disjoint graphs such that a.s. 
$$
\{t\in \mathbb R_+: \Delta M_t \neq 0\} \subset \cup_{n\geq 0}\{\tau_n\}
$$
(see \cite[p.499]{Kal} and \cite[Corollary 26.16]{Kal}). Let $X$ be a Banach space. We say that a local martingale $M:\mathbb R_+ \times \Omega \to X$ is quasi-left continuous (has accessible jumps), if $\langle M, x^*\rangle$ is quasi-left continuous (has accessible jumps) for each $x^* \in X^*$. Notice that these definitions coincide with the previous ones if $X$ is a Hilbert space.

\subsection{Decomposition of stochastic integrals}

The process $\Phi: \mathbb R_+ \times \Omega \to \mathcal L(H,X)$ is called {\it elementary predictable} if it is of the form
$$
\Phi(t,\omega) = \sum_{n=1}^N\sum_{m=1}^M \mathbf 1_{(t_{n-1},t_n]\times A_{mn}}(t,\omega)
\sum_{k=1}^K h_k \otimes x_{kmn},
$$
where $0 \leq t_0 < \ldots < t_n <\infty$,
$A_{1n},\ldots,A_{Mn}\in \mathcal F_{t_{n-1}}$ for each $n = 1,\ldots, N$ and $h_1,\ldots,h_K\in H$ are orthogonal.
For each elementary predictable
$\Phi$ and for any $H$-valued local martingale $M$ we define the stochastic integral with respect to $M$
as an element of $L^0(\Omega; L^{\infty}(\mathbb R_+;X))$ by
\begin{equation}\label{eq:intelpred}
 \int_0^t \Phi(s) \ud M(s) = \sum_{n=1}^N\sum_{m=1}^M \mathbf 1_{A_{mn}}
\sum_{k=1}^K \langle M(t_n\wedge t) - M(t_{n-1}\wedge t),h_k\rangle x_{kmn}.
\end{equation}
We will often write $\Phi \cdot M$ for the process $\int_0^\cdot \Phi(s) \ud M(s)$.\par
To prove sharp bounds for the stochastic integral, we will decompose it by decomposing the integrator $M$ into three parts. 

\begin{lemma}\label{lem:A^cA^qA^a}
 Let $A:\mathbb R_+\times \Omega \to \mathbb R_+$ be an increasing adapted c\`adl\`ag process, $A_0=0$ a.s. Then there exist unique increasing adapted c\`adl\`ag $A^c, A^q, A^a:\mathbb R_+ \times \Omega \to \mathbb R_+$ such that $A^c_0=A^q_0=A^a_0=0$, $A^c$ is continuous a.s., $A^q$ and $A^a$ are pure jump a.s., $A^q$ is quasi-left continuous, $A^a$~has accessible jumps, and $A = A^c + A^q + A^a$.
\end{lemma}
\begin{proof}
 The statement follows from \cite[Proposition 25.17]{Kal}.
\end{proof}

The following lemma is a generalization of both \cite[Theorem 26.14]{Kal} and \cite[Corollary 26.16]{Kal} to the Hilbert space-valued case.
\begin{lemma}[Decomposition of local martingales, Yoeurp, Meyer]\label{lem:YoeMayH}
 Let $M:\mathbb R_+ \times  \Omega \to H$ be a local martingale. Then there exists a unique decomposition $M = M^c + M^q + M^a$, where $M^c:\mathbb R_+ \times \Omega \to H$ is a continuous local martingale, $M^q, M^a:\mathbb R_+ \times \Omega \to H$ are purely discontinuous local martingales, $M^q$ is quasi-left continuous, $M^a$ has accessible jumps, $M^c_0 = M^q_0=0$, and then $[M^c] = [M]^c$, $[M^q] = [M]^q$ and $[M^a] = [M]^a$, where $[M]^c$, $[M]^q$ and $[M]^a$ are defined as in Lemma \ref{lem:A^cA^qA^a}.
\end{lemma}
We will refer to the decomposition in Lemma~\ref{lem:YoeMayH} as the \emph{canonical} decomposition of $M$. For the proof we will need the following lemmas.

\begin{lemma}[Truncation]\label{lem:trunc}
 Let $H$ be a Hilbert space, $M:\mathbb R_+ \times \Omega \to H$ be a local martingale. Then there exist local martingales $M'$, $M'':\mathbb R_+ \times \Omega \to H$ such that  $M=M'+M''$, $M'$ has locally integrable variation and $\|\Delta M''_t\|\leq 1$ a.s.\ for each $t\geq 0$.
\end{lemma}

\begin{proof}
 The proof is essentially the same as the one given in \cite[Lemma 26.5]{Kal} in the real-valued case.
\end{proof}

\begin{lemma}\label{lem:martoflocbddvar}
 Let $M:\mathbb R_+ \times \Omega \to \mathbb R$ be a local martingale of locally bounded variation. Then $M$ is purely discontinuous.
\end{lemma}

\begin{proof}
 The lemma follows from \cite[Theorem 26.6(viii)]{Kal}.
\end{proof}

\begin{proof}[Proof of Lemma \ref{lem:YoeMayH}]
The proof consists of two steps. In the first one we show that we can assume that $\|\Delta M_t\|\leq 1$ a.s.\ for each $t\geq 0$. In the second step we will show the statement in this particular case.

{\it Step 1.}
First decompose $M$ as in Lemma \ref{lem:trunc}. Notice that $M'$ is purely discontinuous. Indeed, for each $h\in H$ the martingale $\langle M', h\rangle$ is locally of bounded variation. Therefore due to Lemma \ref{lem:martoflocbddvar}, $[\langle M', h\rangle]$ is a pure jump process, and by \eqref{eq:quadvarcoorwise} for any orthogonal basis $(h_n)_{n\geq 1}$ of $H$ and for any $t\geq 0$
\begin{align*}
 [M']_t = \sum_{n\geq 1}[\langle M', h_n\rangle]_t = \sum_{n\geq 1}\sum_{0\leq s\leq t}\Delta[\langle M', h_n\rangle]_s &= \sum_{0\leq s\leq t}\sum_{n\geq 1}\Delta[\langle M', h_n\rangle]_s\\
 &=\sum_{0\leq s\leq t}\Delta[M']_s.
\end{align*}
Hence $[M']$ is pure jump, and $M'$ is purely discontinuous.

 The existence of a decomposition of $M'$ into a purely discontinuous quasi-left continuous part and a purely discontinuous part with accessible jumps follows analogously to \cite[Corollary 26.16]{Kal}. The uniqueness holds due to the uniqueness of the decomposition in the real-valued case.

{\it Step 2.} By Step 1 we can assume that $\|\Delta M_t\|\leq 1$ a.s.\ for each $t\geq 0$. By a localization argument we may assume that $M$ is an $L^2$-martingale. Without loss of generality assume also that $M_0=0$ and that $H$ is separable. Let $(h_n)_{n\geq 1}$ be an orthonormal basis of~$H$. For each $n\geq 1$ define a martingale $M^n:\mathbb R_+ \times \Omega \to \mathbb R$ by $M^n:= \langle M, h_n\rangle$. Then by \eqref{eq:quadvarcoorwise} for each $t\geq 0$ a.s.
 \begin{equation}\label{eq:quadvarcoorwise1}
  [M]_t = \sum_{n=1}^{\infty}[M^n]_t.
 \end{equation}
 For each $n\geq 1$ by \cite[Theorem~26.14]{Kal} and \cite[Corollary 26.16]{Kal} we can define martingales $M^{c, n}, M^{q, n}, M^{a,n}:\mathbb R_+ \times \Omega \to \mathbb R$ such that $M^{c, n}_0= M^{q, n}_0= M^{a,n}_0=0$ and $M^{c, n}$ is a continuous martingale, $M^{q, n}$ and $M^{a,n}$ are purely discontinuous, $M^{q,n}$ is quasi-left continuous, $M^{a,n}$ has accessible jumps, and $[M^n]^c = [M^{c,n}]$, $[M^n]^q = [M^{q,n}]$ and $[M^{n}]^a = [M^{a,n}]$. $[M^{c,n}]_t \leq [M^n]_t$ a.s.\ for all $t\geq 0$, so by the Burkholder-Davis-Gundy inequality and \eqref{eq:quadvarcoorwise1} the series $M^c_t :=\sum_{n=1}^{\infty} M^{c,n}_t h_n$ converges in $L^2(\Omega; H)$ for each $t\geq 0$. Moreover, since a conditional expectation is a bounded operator on $L^2(\Omega; H)$, $\mathbb E(M^c_t|\mathcal F_s) = M^c_s$ for each $t\geq s\geq 0$, so $M^c$ is a martingale. Obviously, $M^c_0=0$ and $M^c$ is continuous since $\langle M^c, h_n\rangle = M^{c,n}$ is continuous for each $n\geq 1$. $[M]^c = [M^c]$ since by \eqref{eq:quadvarcoorwise}.
 \[
  [M]^c = \sum_{n=1}^{\infty} [M^n]^c = \sum_{n=1}^{\infty} [M^{c,n}] = [M^c].
 \]
With the same argument we can construct $M^q$ and $M^a$. The uniqueness follows from the uniqueness of the decomposition $\langle M, h_n\rangle = M^{c, n} + M^{q, n} + M^{a,n}$ for each $n\geq 1$.
\end{proof}

Thanks to Lemma \ref{lem:YoeMayH} one can give an equivalent definition of a purely discontinuous local martingale, which is broadly used in the literature (e.g.\ in \cite{JS}).
 \begin{proposition}\label{thm:purdiscorthtoanycont1}
  A local martingale $M:\mathbb R_+\times \Omega \to H$ is purely discontinuous if and only if $\langle M,N\rangle$ is a local martingale for any continuous bounded $H$-valued martingale $N$ such that $N_0 = 0$. 
 \end{proposition}
 
 \begin{proof}
  One direction follows from \cite[Corollary 26.15]{Kal}. Indeed, if $M$ is purely discontinuous and $N$ is a continuous bounded martingale, then a.s.\ $[M,N] = 0$, and therefore $\langle M,N\rangle$ is a local martingale. 
 
Let us now prove the reverse implication. Without loss of generality assume that $M$ is a martingale. By Lemma \ref{lem:YoeMayH}, there exists a continuous martingale $N:\mathbb R_+ \times \Omega \to H$ such that $N_0=0$ and $M - N$ is purely discontinuous. Let $(\tau_n)_{n\geq 1}$ be an increasing sequence of stopping times such that $\tau_n \nearrow \infty$ as $n\to \infty$ and $N^{\tau_n}$ is a bounded continuous martingale for each $n\geq 1$. For any $n\geq 1$, $\langle M,N^{\tau_n}\rangle$ is a martingale by assumption and by the first part of the proof $\langle M-N,N^{\tau_n}\rangle$ is a martingale as well. Hence $\|N^{\tau_n}\|^2 = (\langle M,N^{\tau_n}\rangle - \langle M-N,N^{\tau_n}\rangle)^{\tau_n}$ is a nonnegative martingale that starts at zero and therefore a zero martingale. By letting $n\to\infty$ find $N=0$ a.s., so $M$ is purely discontinuous.
 \end{proof}
As we will see below (Lemma~\ref{lemma:intcontcontintdisdis}), if $M=M^c+M^q+M^a$ is the canonical decomposition of $M$, then the canonical decomposition of $\Phi\cdot M$ is given by
\begin{equation}
\label{eqn:decomSIPrel}
\Phi \cdot M = \Phi \cdot M^c + \Phi \cdot M^q + \Phi \cdot M^a.
\end{equation}
The following four subsections are dedicated to sharp estimates of the respective parts on the right hand side. In Section~\ref{sec:mainResult} we combine our work to estimate $\Phi\cdot M$.

\subsection{Purely discontinuous martingales with accessible jumps}
\label{sec:ItoAJ}

In this section we prove Burkholder-Rosenthal type inequalities for purely discontinuous martingales with accessible jumps. As an immediate consequence, we find sharp bounds for the accessible jump part in \eqref{eqn:decomSIPrel}.\par
As a first step, we will show that we can represent a purely discontinuous martingales with accessible jumps as a sum of jumps occurring at predictable times. We need the following simple observation.
\begin{lemma}\label{lemma:DeltaMisamart}
 Let $X$ be a Banach space, $1\leq p<\infty$, $M:\mathbb R_+ \times \Omega \to X$ be an $L^p$-martingale, $\tau$ be a~predictable stopping time. Then $(\Delta M_{\tau}\mathbf 1_{[0,t]}(\tau))_{t\geq 0}$ is an $L^p$-martingale as well.
\end{lemma}
\begin{proof}
By the definition of a predictable stopping time 
 there exists an increasing sequence of stopping times $(\tau_n)_{n\geq 0}$ such that $\tau_n < \tau$ a.s.\ for each $n\geq 0$ on $\{\tau>0\}$ and $\tau_n \nearrow \tau$ a.s.\ as $n\to \infty$. Then $M^{\tau}$, $M^{\tau_1},\ldots, M^{\tau_n},\ldots$ are $L^p$-martingales. Moreover, $M^{\tau}_t - M^{\tau_n}_t \to \Delta M_{\tau}\mathbf 1_{[0,t]}(\tau)$ is in $L^p(\Omega;X)$ for each $t\geq 0$ due to the fact that $\Delta M_{\tau} = \mathbb E(M|\mathcal F_{\tau})-\mathbb E(M|\mathcal F_{\tau-})$ and \cite[Corollary 2.6.30]{HNVW1}. Consequently, $(\Delta M_{\tau}\mathbf 1_{[0,t]}(\tau))_{t\geq 0}$ is an~$L^p$-martingale.
\end{proof}
 We can now prove the assertion.
 \begin{lemma}\label{lem:limofM^nisM}
  Let $1<p,q<\infty$, $M:\mathbb R_+ \times \Omega \to L^q(S)$ be a purely discontinuous \mbox{$L^p$-mar}\-tin\-gale with accessible jumps. Let $\mathcal T = (\tau_n)_{n=0}^\infty$ be any sequence of predictable stopping times with disjoint graphs that exhausts the jumps of $M$. Then for each $n\geq 0$
  \begin{equation}\label{eq:defofM^n}
   M^n_t = \sum_{k=0}^n \Delta M_{\tau_k} \mathbf 1_{[0, t]}(\tau_k)
  \end{equation}
defines an $L^p$-martingale. Moreover, for any $t\geq 0$, $\|M_{t}-M^n_{t}\|_{L^p(\Omega; L^q(S))}\to 0$ as $n\to \infty$. If $\sup_{t\geq 0} \E\|M_t\|^p<\infty$, then $\|M_{\infty}-M^n_{\infty}\|_{L^p(\Omega; L^q(S))}\to 0$ for $n\to \infty$.
 \end{lemma}
\begin{proof}
\emph{Step 1}. Let us first suppose that $M$ takes values in a finite-dimensional subspace of $L^q(S)$. Let $\tnorm{\cdot}$ be an equivalent Euclidean norm on this subspace. Then, by Lemma~\ref{lemma:DeltaMisamart}, \eqref{eq:defofM^n} defines an $L^p$-martingale.  
 Since $M$ is purely discontinuous, the Burkholder-Davis-Gundy inequality implies 
 \begin{equation}
 \label{eqn:BDGAppProof}
  \mathbb E \tnorm{(M-M^n)_t}^p\eqsim_{p}\mathbb E [M-M^n]_t^{\frac p2} = \mathbb E\Bigl(\sum_{k=n+1}^{\infty}{ \tnorm{\Delta M_{\tau_k}}}^2\mathbf 1_{[0,t]}(\tau_k)\Bigr)^{\frac p2}.
 \end{equation}
Since the sum on the right hand side is a.s.\ bounded by $[M]_t$ and monotonically vanishes as $N\to \infty$, the first convergence result follows. If $\sup_{t\geq 0} \E\|M_t\|^p<\infty$, then
 \[
  \mathbb E \tnorm{(M^n)_t}^p\eqsim_{p}\mathbb E [M^n]_t^{\frac p2} = \mathbb E\Bigl(\sum_{k=1}^{n}{ \tnorm{\Delta M_{\tau_k}}}^2\mathbf 1_{[0,t]}(\tau_k)\Bigr)^{\frac p2} \leq \E{[M]_t}^{p/2}.
 \]
 Thus $M_{\infty}^n=\lim_{t\to\infty} M_t^n$ exists in $L^p(\Om;L^q(S))$. The second convergence result now follows from the computation \eqref{eqn:BDGAppProof} for $t=\infty$.\par
\emph{Step 2}. Let $(x_m)_{m\geq 1}$ be a Schauder basis of $L^q(S)$. Then every $x\in L^q(S)$ has a unique decomposition $x = \sum_{m\geq 1} a_m x_m$, and according to \cite[Proposition 1.1.9]{AlK06} there exists $K>0$ such that for each $N\geq 1$
 \begin{equation}\label{eq:Schauderest}
   \Bigl\|\sum_{m=1}^N a_mx_m\Bigr\| \leq K\|x\|.
 \end{equation}
Let $P_N\in \mathcal L(L^q(S))$ be such that $P_Nx_m = x_m$ for $m\leq N$, and $P_N x_m=0$ for $m>N$. By Step 1, $P_N M^n$ defines an $L^p$-martingale for each $N\geq 1$ and since $P_N M^n_t\rightarrow M_t^n$ in $L^p(\Om;L^q(S))$ for each $t\geq 0$, $M^n$ is an $L^p$-martingale.\par
To prove convergence we use a result from \cite{Y17FourUMD}. Consider a purely discontinuous $L^p$-martingale $\widetilde{M}$ with accessible jumps. For each $n\geq 0$ both martingales $\widetilde{M}$ and $\widetilde{M}^n$ are purely discontinuous and $\widetilde{M}^n$ is weakly differentially subordinated to $\widetilde{M}$, i.e., for any $x^*\in L^{q'}(S)$ the process 
$$[\langle \widetilde{M},x^*\rangle] - [\langle \widetilde{M}^n,x^*\rangle]$$
is a.s.\ non-decreasing and $|\langle \widetilde{M}^n_0,x^*\rangle|\leq |\langle \widetilde{M}_0,x^*\rangle|$ a.s. By  \cite{Y17FourUMD},
\begin{equation}\label{eq:M^m_t^pleqM_t^p}
 \mathbb E\|\widetilde{M}^n_t\|^p \lesssim_{p,q}\mathbb E\|\widetilde{M}_t\|^p,\;\;\; t\geq 0.
\end{equation}
By \eqref{eq:M^m_t^pleqM_t^p} applied for the martingale $\widetilde{M}=(I-P_N)M$, we obtain
$$\mathbb E\|P_NM^n_t-M^n_t\|^p\lesssim_{p,q}\mathbb E\|P_NM_t- M_t\|^p$$ 
for each $N\geq 1$ and $n\geq 0$. Therefore,
\begin{align}\label{eq:limargforM^m}
  (\mathbb E\|M_t - M^n_t\|^p)^{\frac 1p}
  &\leq   (\mathbb E\|M_t - P_NM_t\|^p)^{\frac 1p}\nonumber\\
  &\quad+ (\mathbb E\|P_NM_t - P_NM^n_t\|^p)^{\frac 1p}+ (\mathbb E\|P_NM^n_t - M^n_t\|^p)^{\frac 1p}\nonumber\\
   &\lesssim_{p,q}   (\mathbb E\|M_t - P_NM_t\|^p)^{\frac 1p}+ (\mathbb E\|P_NM_t - P_NM^n_t\|^p)^{\frac 1p}.
\end{align}
For a fixed $\eps>0$, we can now first pick $N$ so that the first term on the right hand side of \eqref{eq:limargforM^m} is less than $\frac{\eps}2$, and subsequently use Step 1 to find an $n=n(N,\eps)$ so that the second term on the right hand side of \eqref{eq:limargforM^m} will be less than $\frac{\eps}2$. Hence for each $\eps>0$ there exists $n_{\eps}$ such that $ (\mathbb E\|M_t - M^{n_{\eps}}_t\|^p)^{\frac 1p}\lesssim_{p,q} \eps$, so $\mathbb E\|M^n_t-M_t\|^p\to 0$ as $n\to \infty$.\par
Finally, if $\sup_{t>0} E\|M_t\|^p<\infty$, then $\sup_{t>0} E\|M_t^n\|^p<\infty$ by \eqref{eq:M^m_t^pleqM_t^p}. Hence $M_{\infty}^n = \lim_{t\to \infty} M_t^n$ exists in $L^p(\Om;L^q(S))$ and \eqref{eq:M^m_t^pleqM_t^p} shows that 
$$\mathbb E\|P_NM^n_{\infty}-M^n_{\infty}\|^p\lesssim_{p,q}\mathbb E\|P_NM_{\infty}- M_{\infty}\|^p$$ 
for each $N\geq 1$ and $n\geq 0$. Repeating \eqref{eq:limargforM^m} for $t=\infty$ now yields the second convergence result.
 \end{proof}
\begin{definition}
For $1<p,q<\infty$ we define $\mathcal M^{\rm acc}_{p,q}$ as the linear space of all \mbox{$L^q(S)$-va}\-lued purely discontinuous $L^p$-martingales with accessible jumps, endowed with the norm $\|M\|_{\mathcal M^{\rm acc}_{p,q}} := \|M_{\infty}\|_{L^p(\Omega; L^q(S))}$. 
\end{definition}
\begin{proposition}\label{prop:purdisaccjumpsisBanachspace}
For any $1<p,q<\infty$ the space $\mathcal M^{\rm acc}_{p,q}$ is a Banach space. 
\end{proposition}
\begin{proof}
Let $(M^n)_{n\geq 1}$ be a Cauchy sequence in $\mathcal M^{\rm acc}_{p,q}$. Since $L^p(\Omega; L^q(S))$ is a Banach space, there exists a limit $\xi$ in $L^p(\Omega; L^q(S))$ of the sequence $(M_{\infty}^n)_{n\geq 1}$. Define an $L^p$-martingale $M$ by $M_t = \mathbb E(\xi|\mathcal F_{t})$, $t\geq 0$. To prove that $M \in \mathcal M^{\rm acc}_{p,q}$, it is enough to show that for each $x^* \in L^{q'}(S)$ the martingale $\langle M, x^*\rangle$ is quasi-left continuous with accessible jumps. Fix $x^* \in L^{q'}(S)$. Define $N := \langle M, x^*\rangle$ and $N^n:=\langle M^n, x^*\rangle$ for each $n\geq 1$. Then $N^n_{\infty} \to N_{\infty}$ in $L^p(\Omega)$. Let $N=N^c+N^q+N^a$ be the Meyer-Yoeurp decomposition in Lemma~\ref{lem:YoeMayH}. Since $N^n$ is a purely discontinuous martingale with accessible jumps, it follows from the uniqueness statement in Lemma~\ref{lem:YoeMayH} that the Meyer-Yoeurp decomposition of $N-N_n$ is given by $N^c+N^q+(N^a-N_n)$. By \cite[Corollaries 26.15 and 26.16]{Kal}, 
$$[N^c,N^q]=[N^c,N^a-N_n]=[N^q,N^a-N_n]=0\qquad \text{a.s.}$$
It therefore follows from the Burkholder-Davis-Gundy inequality that 
 \begin{align*}
    \mathbb E |N_{\infty}-N^n_{\infty}|^p & \eqsim_{p} \mathbb E [N-N^n]_{\infty}^{\frac p2} \\
    &=\mathbb E \bigl([N]^c_{\infty} + [N]^q_{\infty} + [N-N^n]^a_{\infty}\bigr)^{\frac{p}{2}} \geq \mathbb E \bigl([N]^c_{\infty} + [N]^q_{\infty}\bigr)^{\frac{p}{2}}.
 \end{align*}
By taking $n\to\infty$, we find $\mathbb E \bigl([N]^c_{\infty} + [N]^q_{\infty}\bigr)^{\frac{p}{2}} = 0$ and so $[N]^c_{\infty} = [N]^q_{\infty}=0$ a.s. We conclude that $N$ is purely discontinuous with accessible jumps. Since this holds for any $x^* \in L^{q'}(S)$, $M \in \mathcal M^{\rm acc}_{p,q}$.
\end{proof}
We now turn to the Burkholder-Rosenthal inequalities.
Let $1<p,q<\infty$ and let $M:\mathbb R_+ \times \Omega \to L^q(S)$ be a purely discontinuous martingale with accessible jumps.  Let $\mathcal T = (\tau_n)_{n\geq 0}$ be a sequence of predictable stopping times with disjoint graphs that exhausts the jumps of $M$. We define three expressions 
\begin{equation}\label{eq:normsforaccjump}
 \begin{split}
  \|M\|_{\widetilde S^p_q} &= \Bigl( \mathbb E\Bigl\| \Bigl({ \sum_{n\geq 0} \mathbb E_{\mathcal F_{\tau_n-}}|(\Delta M(\omega)(s))_{\tau_n}|^2 } \Bigr)^{\frac 12}\Bigr\|^p_{L^q(S)} \Bigr)^{\frac 1p},\\
   \|M\|_{\widetilde D^p_{q,q}} &= \Bigl( \mathbb E \Bigl({\sum_{n\geq 0}\mathbb E_{\mathcal F_{\tau_n-}}\|\Delta M(\omega)_{\tau_n}\|^q_{L^q(S)} } \Bigr)^{\frac pq} \Bigr)^{\frac 1p},\\
   \|M\|_{\widetilde D^p_{p,q}} &= \Bigl( \mathbb E \sum_{t\geq 0} \|\Delta M_t\|_{L^q(S)}^p \Bigr)^{\frac 1p}.
   \end{split}
\end{equation}

\begin{proposition}
The expressions in \eqref{eq:normsforaccjump} do not depend on the choice of the family $\mathcal T$.
\end{proposition}

\begin{proof}
 Assume that $\mathcal T' = (\tau'_{m})_{m\geq 0}$ is another family of predictable stopping times with disjoint graphs that exhausts the jumps of $M$. Notice that due to \cite[Proposition I.2.11]{JS}, 
 $$
 \mathcal F_{\tau-}\cap\{\tau = \sigma\} = \mathcal F_{\sigma-}\cap\{\tau = \sigma\}
 $$ 
 for any pair of predictable stopping times $\tau$ and $\sigma$. Therefore for a.e.\ $s\in S$ a.s.
 \begin{align*}
 \sum_{n\geq 0} \mathbb E_{\mathcal F_{\tau_n-}}|(\Delta M(\omega)(s))_{\tau_n}|^2 &\stackrel{(*)}= \sum_{n\geq 0} \mathbb E_{\mathcal F_{\tau_n-}}\Bigl(\sum_{m\geq 0}|(\Delta M(\omega)(s))_{\tau_n}|^2 \mathbf 1_{\tau_n=\tau'_m}\Bigr)\\
 &=\sum_{n\geq 0} \sum_{m\geq 0}\mathbb E_{\mathcal F_{\tau_n-}}\bigl(|(\Delta M(\omega)(s))_{\tau_n}|^2 \mathbf 1_{\tau_n=\tau'_m}\bigr)\\
 &=\sum_{n\geq 0} \sum_{m\geq 0}\mathbb E_{\mathcal F_{\tau'_m-}}\bigl(|(\Delta M(\omega)(s))_{\tau_n}|^2 \mathbf 1_{\tau_n=\tau'_m}\bigr)\\
 &= \sum_{m\geq 0} \mathbb E_{\mathcal F_{\tau'_m-}}\Bigl(\sum_{n\geq 0}|(\Delta M(\omega)(s))_{\tau'_m}|^2 \mathbf 1_{\tau_n=\tau'_m}\Bigr)\\
 &\stackrel{(*)}=\sum_{n\geq 0} \mathbb E_{\mathcal F_{\tau'_m-}}|(\Delta M(\omega)(s))_{\tau'_m}|^2,
 \end{align*}
where $(*)$ holds since
\begin{align*}
 & \mathbb P\{\exists u\geq 0: \Delta M_u\neq 0, u\notin\{\tau_0,\tau_1,\ldots\}\} \\
 &\qquad \qquad \qquad \qquad = \mathbb P\{\exists u\geq 0: \Delta M_u\neq 0, u\notin\{\tau'_0,\tau'_1,\ldots\}\} = 0.
\end{align*}
Therefore we can conclude that $\|M\|_{\widetilde S^p_q}$ does not depend on the choice of the exhausting family. The same holds for $\|M\|_{\widetilde D^p_{q,q}}$ by an analogous argument.
\end{proof}

We let $\widetilde S^p_q$, $\widetilde D^p_{q,q}$ and $\widetilde D^p_{p,q}$ denote the sets of all purely discontinuous martingales with accessible jumps for which the respective expressions in \eqref{eq:normsforaccjump} are finite. We will prove shortly that the expressions in \eqref{eq:normsforaccjump} are norms. For a fixed family $\mathcal T = (\tau_n)_{n\geq 0}$ of predictable stopping times with disjoint graphs we let $\widetilde{S}^{p,\mathcal T}_q$,  $\widetilde{D}^{p,\mathcal T}_{q,q}$ and $\widetilde{D}^{p,\mathcal T}_{p,q}$ be the subsets of $\widetilde{S}^p_q$,  $\widetilde{D}^p_{q,q}$ and $\widetilde{D}^p_{p,q}$ consisting of martingales $M$ with $\{t\in \mathbb R_+: \Delta M_t \neq 0\} \subset {\{\tau_0,\tau_1,\ldots\}}$ a.s.\par
We start by proving a version of the main theorem of this subsection (Theorem~\ref{thm:acessjumpsL^q} below) for a martingales with finitely many jumps.  
\begin{theorem}\label{thm:acessjumpsinTL^q}
 Let $1<p,q<\infty$, $N\geq 1$, $\mathcal T = (\tau_n)_{n= 0}^N$ be a finite family of predictable stopping times with disjoint graphs. Then $\widetilde{S}^{p,\mathcal T}_q$,  $\widetilde{D}^{p,\mathcal T}_{q,q}$ and $\widetilde{D}^{p,\mathcal T}_{p,q}$ are Banach spaces under the norms in \eqref{eq:normsforaccjump}. As a consequence, $A_{p, q}^{\mathcal T}$ given by
 \begin{equation}\label{eq:Apq^T}
 \begin{split}
   \widetilde{S}^{p,\mathcal T}_q \cap \widetilde{D}^{p,\mathcal T}_{q,q} \cap \widetilde{D}^{p,\mathcal T}_{p,q} & \text{ if } \; 2\leq q\leq p<\infty,\\
 \widetilde{S}^{p,\mathcal T}_q \cap (\widetilde{D}^{p,\mathcal T}_{q,q} + \widetilde{D}^{p,\mathcal T}_{p,q}) & \text{ if } \; 2\leq p\leq q<\infty,\\
 (\widetilde{S}^{p,\mathcal T}_q \cap \widetilde{D}^{p,\mathcal T}_{q,q}) + \widetilde{D}^{p,\mathcal T}_{p,q} & \text{ if } \; 1<p<2\leq q<\infty,\\
 (\widetilde{S}^{p,\mathcal T}_q +\widetilde{D}^{p,\mathcal T}_{q,q}) \cap \widetilde{D}^{p,\mathcal T}_{p,q} & \text{ if } \; 1<q<2\leq p<\infty,\\
 \widetilde{S}^{p,\mathcal T}_q + (\widetilde{D}^{p,\mathcal T}_{q,q} \cap \widetilde{D}^{p,\mathcal T}_{p,q}) & \text{ if } \; 1<q\leq p\leq 2,\\
 \widetilde{S}^{p,\mathcal T}_q + \widetilde{D}^{p,\mathcal T}_{q,q} + \widetilde{D}^{p,\mathcal T}_{p,q}& \text{ if } \; 1<p\leq q \leq 2.
 \end{split}
\end{equation}
is a well-defined Banach space. Moreover, $(\mathcal A^{\mathcal T}_{p,q})^{*} = \mathcal A^{\mathcal T}_{p',q'}$ with isomorphism given~by  
 \begin{equation}\label{eq:dualApq^T}
 \begin{split}
  g \mapsto F_g, \;\;\; F_g(f) &= \mathbb E \sum_{ t\in \mathcal T}\langle  \Delta g_t,\Delta f_t\rangle  \stackrel{(*)}= \mathbb E \langle g_{\infty}, f_{\infty}\rangle,\\ 
  \|F_g\|_{(\mathcal A^{\mathcal T}_{p,q})^{*}} &\eqsim_{p,q} \|g\|_{\mathcal A^{\mathcal T}_{p',q'}}.
 \end{split}
 \end{equation}
Finally, for any purely discontinuous $L^p$-martingale $M:\mathbb R_+ \times \Omega \to L^q(S)$ with accessible jumps such that $\{t\in \mathbb R_+:\Delta M_t \neq 0\} \subset { \{\tau_0,\ldots,\tau_N\}}$ a.s., and for all $t\geq 0$, 
 \begin{equation}\label{eq:mainforaccesjumpsinT}
 \Bigl( \mathbb E \sup_{0\leq s\leq t}\|  M_{t}\|^p_{L^q(S)} \Bigr)^{\frac 1p} \eqsim_{p,q}\|M\mathbf 1_{[0,t]}\|_{\mathcal A_{p,q}^{\mathcal T}}.
 \end{equation}
\end{theorem}
The idea of the proof is to identify $\widetilde{S}^{p,\mathcal T}_q$,  $\widetilde{D}^{p,\mathcal T}_{q,q}$ and $\widetilde{D}^{p,\mathcal T}_{p,q}$ with discrete martingale spaces $S^p_q$, $D^p_{q,q}$ and $D^p_{p,q}$ for an appropriate filtration. For this purpose we need two observations.
\begin{lemma}\label{lemma:comptau-}
 Let $F:\mathbb R_+ \times \Omega \to \mathbb R_+$ be a 
 locally integrable c\`adl\`ag adapted process, $\tau$ be a predictable stopping time. Let $G,H : \mathbb R_+ \times \Omega \to \mathbb R_+$ be such that $G_t=F_{\tau}\mathbf 1_{[0,t]}(\tau)$, $H_t =\mathbf1_{[0,t]}(\tau) \mathbb E_{\mathcal F_{\tau-}} F_{\tau}$ for each $t\geq 0$. Then $G-H$ is a local martingale.
\end{lemma}
\begin{proof}
 Without loss of generality suppose that $F$ is integrable. First of all notice that $H$ is a predictable process thanks to \cite[Lemma 25.3(ii)]{Kal}, and $G$ is adapted due to the fact that $G_t = F_{\tau\wedge t}\mathbf 1_{[0,t]}(\tau)$. 
 Fix $t>s\geq 0$. By \cite[Lemma 25.2(i)]{Kal}, $\mathcal F_{s} \cap \{s< \tau\} \subset \mathcal F_{\tau -}$ and $\mathcal F_{s} \cap \{t< \tau\} \subset \mathcal F_{t} \cap \{t< \tau\}\subset \mathcal F_{\tau -}$. Hence,  
 $$
 \mathcal F_{s} \cap \{s< \tau\leq t\} \subset \mathcal F_{\tau -}
 $$ 
 and so
 \begin{align*}
  \mathbb E(G_t&-H_t|\mathcal F_{s}) = \mathbb E(F_{\tau}\mathbf 1_{\{\tau\leq t\}}-\mathbf 1_{\{\tau\leq t\}} \mathbb E_{\mathcal F_{\tau-}} F_{\tau}|\mathcal F_{s})\\
  &=\mathbb E(F_{\tau}\mathbf 1_{\{\tau\leq s\}}-\mathbf 1_{\{\tau\leq s\}} \mathbb E_{\mathcal F_{\tau-}} F_{\tau}|\mathcal F_{s})\\
  &\quad+\mathbb E(F_{\tau}\mathbf 1_{\{s<\tau\leq t\}}-\mathbf 1_{\{s<\tau\leq t\}} \mathbb E_{\mathcal F_{\tau-}} F_{\tau}|\mathcal F_{s})\\
  &= G_s-H_s + \mathbb E( \mathbb E(F_{\tau}- \mathbb E_{\mathcal F_{\tau-}} F_{\tau}|\mathcal F_{\tau-}){ \mathbf 1_{\{s<\tau\leq t\}}}|\mathcal F_{s}\cap { \{s<\tau\leq t\}}) = G_s-H_s.
 \end{align*}
\end{proof}

\begin{corollary}\label{cor:xitomart}
 Let $X$ be a Banach space, $\tau$ be a predictable stopping time, $\xi \in L^1(\Omega; X)$ be $\mathcal F_{\tau}$-measurable such that $\mathbb E_{\mathcal F_{\tau_-}}\xi = 0$. Let $M:\mathbb R_+ \times \Omega \to X$ be such that $M_t = \xi \mathbf 1_{[0, t]}(\tau)$. Then $M$ is a martingale.
\end{corollary}

\begin{proof}
 The case $X  =\mathbb R$ follows from Lemma~\ref{lemma:comptau-} and the fact that $\xi \mathbf 1_{\tau\leq t}$ is \mbox{$\mathcal F_t$-me}\-a\-su\-rable for each $t\geq 0$ by the definition of $\mathcal F_{\tau}$. For the general case we notice that $\langle M, x^*\rangle$ is a martingale for each $x^* \in X$ and since $M$ is integrable it follows that $M$ is a martingale.
\end{proof}

\begin{proof}[Proof of Theorem \ref{thm:acessjumpsinTL^q}]
 Since the $\tau_i$ have disjoint graphs, we can find predictable stopping times $\tau_0',\ldots,\tau_N'$ such that $\{\tau_0(\omega),\ldots,\tau_N(\omega)\} = \{\tau_0'(\omega),\ldots,\tau_N'(\omega)\}$ and $\tau_0'(\omega)<\ldots
<\tau_N'(\omega)$ for a.e.\ $\omega \in \Omega$.
Indeed, we can set $\tau_0'=\min\{\tau_0,\ldots,\tau_N\}$ and
$$\tau_i' = \min(\{\tau_0,\ldots,\tau_N\}\setminus \{\tau_0',\ldots,\tau_i'\}), \qquad 0\leq i\leq N-1.$$
Fix the sequence of $\sigma$-algebras $\mathbb G = (\mathcal G_k)_{k=0}^{2N+1} = (\mathcal F_{\tau_0'-}, \mathcal F_{\tau_0'},\ldots,\mathcal F_{\tau_N'-},\mathcal F_{\tau_N'})$. 
Using \cite[Lemma 25.2]{Kal} and the fact that $(\tau_n')_{n=0}^N$ is a.s.\ a strictly increasing sequence one can show that $\mathbb G$ is a filtration. 

Consider Banach spaces $\hat S^{p,odd}_q$, $\hat D^{p,odd}_{q,q}$ and $\hat D^{p,odd}_{p, q}$ with respect to the filtration $\mathbb G$ that were defined in Remark \ref{rem:oddmartingales}. For any purely discontinuous $L^q$-valued martingale $M$ with accessible jumps in $\mathcal{T}$ we can construct a $\mathbb G$-martingale difference sequence $(d_k)_{n=0}^{2N+1}$ by setting $d_{2n} = 0$, $d_{2n+1} = \Delta M_{\tau_{n}}$ for $n=0,\ldots, N$. Indeed, by \cite[Lemma 26.18]{Kal} (see also \cite[Lemma 2.27]{JS}) for each $n=0,\ldots, N$
\[
  \mathbb E (d_{2n+1}|{\mathcal G_{2n}}) =  \mathbb E ( \Delta M_{\tau_n'}|\mathcal F_{\tau_n'-}) =0.
 \]
 By Lemma \ref{lemma:comptau-}, 
$$\|M\|_{\widetilde{S}^{p,\mathcal T}_q} = \|(d_n)\|_{S^{p,odd}_q}, \;\;\;  \|M\|_{\widetilde{D}^{p,\mathcal T}_{q,q}} = \|(d_n)\|_{D^{p,odd}_{q,q}}, \;\;\; \|M\|_{\widetilde{D}^{p,\mathcal T}_{p,q}} = \|(d_n)\|_{D^{p,odd}_{p,q}}.$$
Moreover, by Corollary \ref{cor:xitomart} any element $(d_k)_{k=0}^{2N+1}$ of $\hat S^{p,odd}_q$,  $D^{p,odd}_{q,q}$, or $D^{p,odd}_{p,q}$ (so in particular, $d_{2n}=0$ for each $n=0,\ldots, N$) can be converted back to an element $M$ of $\widetilde{S}^{p,\mathcal T}_q$, $\widetilde{D}^{p,\mathcal T}_{q,q}$, or $\widetilde{D}^{p,\mathcal T}_{p,q}$, respectively, by defining
\[
 M_t = \sum_{n=0}^N d_{2n+1}\mathbf 1_{[0, t]}( \tau_n'),\;\;\; t\geq 0.
\]
Using this identification, we find that $\widetilde{S}^{p,\mathcal T}_q$, $\widetilde{D}^{p,\mathcal T}_{q,q}$, and $\widetilde{D}^{p,\mathcal T}_{p,q}$ are Banach spaces. As a consequence, 
$\mathcal A_{p, q}^{\mathcal T}$ is a well-defined Banach space that is isometrically isomorphic to $\hat s^{odd}_{p,q}$. 
The duality statement now follows from the duality for $s^{odd}_{p,q}$ and $(*)$ in \eqref{eq:dualApq^T} follows from \eqref{eq:sumofexpect=expectofsum}.

Now let us show \eqref{eq:mainforaccesjumpsinT}. By Doob's maximal inequality,
 \[
  \Bigl( \mathbb E \sup_{0\leq s\leq t}\|  M_{t}\|^p_{L^q(S)} \Bigr)^{\frac 1p} \eqsim_{p}  \Bigl( \mathbb E \|  M_{t}\|^p_{L^q(S)} \Bigr)^{\frac 1p}.
 \]
 Again define a $\mathbb G$-martingale difference sequence $(d_n)_{n=0}^{2N+1}$ by setting $d_{2n} = 0$, $d_{2n+1} = \Delta M_{\tau_{n}}$, where $n=0,\ldots, N$. Then by Remark \ref{rem:oddmartingales}
 $$
\|M_{\infty}\|_{L^p(\Omega; X)} = \Bigl\|\sum_{n=0}^{2N+1}d_k\Bigr\|_{L^p(\Omega; X)} \eqsim_{p,q}\|(d_n)_{n=0}^{2N+1}\|_{\hat s^{odd}_{p,q}} = \|M\|_{\mathcal A_{p,q}^{\mathcal T}}.
$$
\end{proof}
To treat the general case we use an approximation argument based on the following observation.

\begin{lemma}\label{lem:M^ntoMinA_pq}
\label{lem:appApq}  Let $1<p,q<\infty$. Let $M$ be in $\widetilde{S}^{p}_q$, $\widetilde{D}^{p}_{q,q}$, or $\widetilde{D}^{p}_{p,q}$ and let $\mathcal{T}=(\tau_n)_{n\geq 0}$ be any sequence of predictable stopping times with disjoint graphs that exhausts the jumps of $M$. Consider the process $M^n=M^n_{\mathcal{T}}$ defined in \eqref{eq:defofM^n}. Then $M^n\to M$ in $\widetilde{S}^{p}_q$, $\widetilde{D}^{p}_{q,q}$, or $\widetilde{D}^{p}_{p,q}$, respectively. As a consequence, $\widetilde{S}^{p}_q$, $\widetilde{D}^{p}_{q,q}$, and $\widetilde{D}^{p}_{p,q}$ are normed linear spaces and $\mathcal A_{p,q}$ given by
\begin{equation}\label{eq:Apq}
 \begin{split}
   \widetilde{ S}^p_q \cap \widetilde{D}^p_{q,q} \cap \widetilde{D}^p_{p,q} & \text{ if } \; 2\leq q\leq p<\infty,\\
 \widetilde{S}^p_q \cap (\widetilde{D}^p_{q,q} + \widetilde{D}^p_{p,q}) & \text{ if } \; 2\leq p\leq q<\infty,\\
 (\widetilde{S}^p_q \cap \widetilde{D}^p_{q,q}) + \widetilde{D}^p_{p,q} & \text{ if } \; 1<p<2\leq q<\infty,\\
 (\widetilde{S}^p_q +\widetilde{D}^p_{q,q}) \cap \widetilde{D}^p_{p,q} & \text{ if } \; 1<q<2\leq p<\infty,\\
 \widetilde{S}^p_q + (\widetilde{D}^p_{q,q} \cap \widetilde{D}^p_{p,q}) & \text{ if } \; 1<q\leq p\leq 2,\\
 \widetilde{S}^p_q + \widetilde{D}^p_{q,q} + \widetilde{D}^p_{p,q} & \text{ if } \; 1<p\leq q \leq 2.
 \end{split}
\end{equation}
is a well-defined normed linear space. If $M\in\mathcal A_{p,q}$, then 
there exists a sequence of predictable stopping times $\mathcal T$ with disjoint graphs that exhausts the jumps of $M$ so that $M^n_{\mathcal{T}}\to M$ in $\mathcal A_{p,q}$. 
\end{lemma}
\begin{proof}
We prove the two first statements only for $\widetilde{S}^{p}_q$.  
By the dominated convergence theorem, we obtain $M^n\to M$ in $\widetilde{S}^{p}_q$ as well as $\|M^n\|_{\widetilde{S}^{p}_q}\nearrow \|M\|_{\widetilde{S}^{p}_q}$. Suppose now that $M,N\in \widetilde{S}^{p}_q$. By \cite[Lemma I.2.23]{JS}, there exists a sequence $\mathcal{T}=\{\tau_n\}_{n\geq 0}$ of predictable stopping times with disjoint graphs that exhausts the jumps of both $M$ and $N$. Now clearly, $(M+N)^n=M^n+N^n$ and so
\begin{align*}
\|M+N\|_{\widetilde{S}^{p}_q} & = \lim_{n\to \infty} \|M^n+N^n\|_{\widetilde{S}^{p}_q}  \\
& = \lim_{n\to \infty} \|M^n+N^n\|_{\widetilde{S}^{p,\mathcal{T}}_q} \\
& \leq \lim_{n\to \infty} \|M^n\|_{\widetilde{S}^{p,\mathcal{T}}_q}+\|N^n\|_{\widetilde{S}^{p,\mathcal{T}}_q} 
= \|M\|_{\widetilde{S}^{p}_q}+\|N\|_{\widetilde{S}^{p}_q}. 
\end{align*}
Let us prove the final statement if $p\leq q\leq 2$, the other cases are similar. Let $M\in \mathcal A_{p,q}$ and let $M_1\in \widetilde{S}^{p}_q$, $M_2\in \widetilde{D}^{p}_{q,q}$, $M_3\in \widetilde{D}^{p}_{p,q}$ be such that $M=M_1+M_2+M_3$. Let $\mathcal{T}=\{\tau_n\}_{n\geq 0}$ be a sequence of predictable stopping times with disjoint graphs that exhausts the jumps of $M_1$, $M_2$ and $M_3$. Then $M^n=M_1^n+M_2^n+M_3^n$ and by the above,
$$\|M-M^n\|_{\mathcal A_{p,q}} \leq \|M_1-M_1^n\|_{\widetilde{S}^{p,\mathcal{T}}_q} + \|M_2-M_2^n\|_{\widetilde{D}^{p}_{q,q}} + \|M_3-M_3^n\|_{\widetilde{D}^{p}_{p,q}}\to 0$$
as $n\to \infty$. 
\end{proof}

\begin{lemma}\label{lem:dualApq^T}
 Let $1<p,q < \infty$, $N\geq 1$, $\mathcal T = (\tau_n)_{n=0}^N$ be a finite family of predictable stopping times with disjoint graphs. Then $\mathcal A^{\mathcal T}_{p,q}\hookrightarrow \mathcal A_{p,q}$ isometrically.
\end{lemma}
\begin{proof}
We will consider only the case $p\leq q \leq 2$, the other cases can be shown analogously. Let $M \in \mathcal A^{\mathcal T}_{p,q}$. Then automatically $M \in \mathcal A_{p,q}$ and $\|M\|_{\mathcal A^{\mathcal T}_{p,q}} \geq \|M\|_{\mathcal A_{p,q}}$. Let us show the reverse inequality. Fix $\eps>0$, and let $M^1 \in \widetilde{S}^p_q$, $M^2 \in \widetilde{D}^p_{q,q}$ and $M^3 \in \widetilde{D}^p_{p,q}$ be martingales such that $M = M^1 + M^2 + M^3$ and 
\[
 \|M\|_{\mathcal A_{p,q}} \geq \|M^1\|_{\widetilde{S}^p_q} + \|M^2\|_{\widetilde{D}^p_{q,q}} + \|M^3\|_{\widetilde{D}^p_{p,q}}- \eps.
\]
By Lemma \ref{lemma:DeltaMisamart} we can define martingales $\widetilde M^1$, $\widetilde M^2$ and $\widetilde M^3$ by
\begin{equation}\label{eq:defofwidetilde M^i_t}
  \widetilde M^i_t = \sum_{s\in \mathcal T\cap [0, t]}\Delta M^i_s,\;\;\; t\geq 0,\; i=1,2,3.
\end{equation}
Notice that $|\Delta \widetilde M^i_t(\omega)(s)|\leq |\Delta M^i_t(\omega)(s)|$ for each $t\geq 0$, $\omega \in \Omega$, $s\in S$ and $i=1,2,3$. Therefore 
$\widetilde M^1 \in \widetilde{S}^p_q$, $\widetilde M^2 \in \widetilde{D}^p_{q,q}$ and $\widetilde M^3 \in \widetilde{D}^p_{p,q}$ and $\|\widetilde M^1\|_{\widetilde{S}^p_q}\leq \| M^1\|_{\widetilde{S}^p_q}$, $\|\widetilde M^2\|_{\widetilde{D}^p_{q,q}}\leq \| M^2\|_{\widetilde{D}^p_{q,q}}$ and $\|\widetilde M^3\|_{\widetilde{D}^p_{p,q}}\leq \| M^3\|_{\widetilde{D}^p_{p,q}}$. Moreover, $M = \widetilde M^1 + \widetilde M^2 + \widetilde M^3$. Indeed, since all the martingales here are purely discontinuous with accessible jumps, by \eqref{eq:defofwidetilde M^i_t} we find for each $t\geq 0$ a.s.\
\begin{align*}
 M_t = \sum_{s\in \mathcal T \cap [0, t]} \Delta M_s &= \sum_{s\in \mathcal T \cap [0, t]} \Bigl(\Delta M^1_s + \Delta M^2_s + \Delta M^3_s\Bigr)\\
 &=\sum_{s\in \mathcal T \cap [0, t]} \Delta M^1_s + \sum_{s\in \mathcal T \cap [0, t]}\Delta M^2_s + \sum_{s\in \mathcal T \cap [0, t]}\Delta M^3_s\\
 &= \widetilde M^1_t + \widetilde M^2_t + \widetilde M^3_t.
\end{align*}
Therefore
\begin{align*}
  \|M\|_{\mathcal A^{\mathcal T}_{p,q}} &\leq \|\widetilde M^1\|_{\widetilde{S}^p_q} + \|\widetilde M^2\|_{\widetilde{D}^p_{q,q}} + \|\widetilde M^3\|_{\widetilde{D}^p_{p,q}}\\
  &\leq \|M^1\|_{\widetilde{S}^p_q} + \|M^2\|_{\widetilde{D}^p_{q,q}} + \|M^3\|_{\widetilde{D}^p_{p,q}} \leq \|M\|_{\mathcal A_{p,q}} +\eps.
\end{align*}
Since $\eps$ was arbitrary, we conclude that $\|M\|_{\mathcal A^{\mathcal T}_{p,q}}  \leq \|M\|_{\mathcal A_{p,q}}$, and consequently $\|M\|_{\mathcal A^{\mathcal T}_{p,q}}  = \|M\|_{\mathcal A_{p,q}}$.
\end{proof}
We can now readily deduce the main theorem of this section.
\begin{theorem}\label{thm:acessjumpsL^q}
  Let $1<p,q<\infty$, $M:\mathbb R_+ \times \Omega \to L^q(S)$ be a purely discontinuous martingale with accessible jumps. Then for all $t\geq 0$ one has that
 \begin{equation}\label{eq:mainforaccesjumps}
 \Bigl( \mathbb E \sup_{0\leq s\leq t}\|  M_{t}\|^p_{L^q(S)} \Bigr)^{\frac 1p} \eqsim_{p,q}\|M\mathbf 1_{[0,t]}\|_{\mathcal A_{p,q}},
 \end{equation}
where $\mathcal A_{p,q}$ is as in \eqref{eq:Apq}. In particular, $\mathcal A_{p,q}$ is a Banach space of $L^p$-martingales.
\end{theorem}
\begin{proof}

Suppose first that $M\in\mathcal A_{p,q}$. By Lemma~\ref{lem:appApq} there exists a sequence of predictable stopping times with disjoint graphs that exhausts the jumps of $M$ so that $M^n_{\mathcal{T}}\to M$ in $\mathcal A_{p,q}$. In particular, $(M^n_{\mathcal{T}})_{n\geq 0}$ is Cauchy in $\mathcal A_{p,q}$. By Lemma~\ref{lem:dualApq^T} and Theorem~\ref{thm:acessjumpsinTL^q} it follows that it is also Cauchy in $\mathcal{M}^{\rm acc}_{p,q}$. By Proposition~\ref{prop:purdisaccjumpsisBanachspace} $(M^n_{\mathcal{T}})_{n\geq 0}$ converges and clearly the limit is $M$.\par 
Suppose now that $M \in \mathcal{M}^{\rm acc}_{p,q}$. It suffices to show that $M\in \mathcal A_{p,q}$. Indeed, Lemma~\ref{lem:appApq} then shows that there is a sequence of predictable stopping times with disjoint graphs that exhausts the jumps of $M$ so that $M^n_{\mathcal{T}}\to M$ in $\mathcal A_{p,q}$. By Lemma~\ref{lem:limofM^nisM} we also have $M^n_{\mathcal{T}}\to M$ in $\mathcal{M}^{\rm acc}_{p,q}$ and so the lower bound in \eqref{eq:mainforaccesjumps} follows from Lemma~\ref{lem:dualApq^T} and Theorem~\ref{thm:acessjumpsinTL^q}. We will show that $M\in \mathcal A_{p,q}$ in the two cases $2\leq q\leq p$ and $p\leq q\leq 2$, the other cases can be treated analogously.

{\it Case $2\leq q\leq p$.}  We will show that $\|M\|_{\widetilde S_q^p} \lesssim_{p,q} \|M\|_{\mathcal{M}^{\rm acc}_{p,q}}$. The analogous statements for $\widetilde D_{q, q}^p$ and $\widetilde D_{p, q}^p$ can be shown in the same way. 
By Theorem~\ref{thm:acessjumpsinTL^q}, $\|M^n\|_{\widetilde S_q^p} \lesssim_{p,q} \|M^n\|_{\mathcal{M}^{\rm acc}_{p,q}}$. Also, by \eqref{eq:M^m_t^pleqM_t^p} we have $\|M^n\|_{\mathcal{M}^{\rm acc}_{p,q}} \lesssim_{p,q} \|M\|_{\mathcal{M}^{\rm acc}_{p,q}}$ for all $n\geq 0$. Therefore $\|M^n\|_{\widetilde S_q^p} \lesssim_{p,q} \|M\|_{\mathcal{M}^{\rm acc}_{p,q}}$ uniformly in $n$, so by monotone convergence
\begin{align*}
 \|M\|_{\widetilde S_q^p}^p &= \mathbb E\Bigl\| \Bigl({\sum_{m\geq 0} \mathbb E_{\mathcal F_{\tau_m-}}|(\Delta M(\omega)(s))_{\tau_m}|^2 } \Bigr)^{\frac 12}\Bigr\|^p_{L^q(S)} \\
 &=\lim_{n\to \infty}\mathbb E\Bigl\| \Bigl({\sum_{m= 0}^n \mathbb E_{\mathcal F_{\tau_m-}}|(\Delta M(\omega)(s))_{\tau_m}|^2 } \Bigr)^{\frac 12}\Bigr\|^p_{L^q(S)}\\
 &= \lim_{n\to \infty}\|M^n\|^p_{\widetilde S_q^p} \lesssim_{p,q} \|M\|^p_{\mathcal{M}^{\rm acc}_{p,q}}.
\end{align*}

{\it Case $p\leq q\leq 2$.} Observe that $\|M^n\|_{\mathcal A_{p,q}}\eqsim_{p,q}\|M^n\|_{\mathcal{M}^{\rm acc}_{p,q}}$ for each $n\geq 0$ by Theorem~\ref{thm:acessjumpsinTL^q} and since $(M^n)_{n\geq 0}$ is a Cauchy sequence in $\mathcal{M}^{\rm acc}_{p,q}$ due to Lemma \ref{lem:limofM^nisM}, it follows that $(M^n)_{n\geq 0}$ is a Cauchy sequence in $\mathcal A_{p,q}$. Thus there exists a subsequence $(M^{n_k})_{k\geq 0}$ such that
\[
 \|M^{n_{k+1}}-M^{n_k}\|_{\mathcal A_{p,q}}<\frac 1 {2^{k+1}},\;\;\; k\geq 0.
\]
 Let $N^{k} = M^{n_{k}}-M^{n_{k-1}}$, $k\geq 1$, $N^0 = M^{n_0}$. Set $n_{-1}=-1$. By Theorem \ref{thm:acessjumpsinTL^q}, for each $k\geq 0$ there exist $N^{k,1}$, $N^{k,2}$ and $N^{k,3}$ such that $N^{k,1} \in \widetilde S_q^p$, $N^{k,2} \in \widetilde D_{q, q}^p$, $N^{k,3}\in\widetilde D_{p, q}^p$, $N^k = N^{k,1} + N^{k,2} + N^{k,3}$,
\[
 \{t: \Delta N^{k,i}_t\neq 0, i=1,2,3\}\subset \{\tau_{n_{k-1}+1},\ldots,\tau_{n_{k}}\},\;\;\; \text{a.s.,}
\]
and
\begin{equation}\label{eq:estimatesforN^ki}
 \begin{split}
   \|N^{k,1}\|_{\widetilde S_q^p} + \|N^{k,2}\|_{\widetilde D_{q, q}^p} + \|N^{k,3}\|_{\widetilde D_{p, q}^p} &< \frac {1}{2^k},\;\;\;\;\;\; k\geq 1,\\
 \|N^{0,1}\|_{\widetilde S_q^p} + \|N^{0,2}\|_{\widetilde D_{q, q}^p} + \|N^{0,3}\|_{\widetilde D_{p, q}^p}& \leq 2\|M^{ n_0}\|_{\mathcal A_{p,q}}.
 \end{split}
\end{equation}
Let
\[
 M^{m,i}:= \sum_{k=0}^m N^{k,i},\;\;\; m\geq 0,\;\;\; i=1,2,3.
\]
Then by \eqref{eq:estimatesforN^ki}, $(M^{m,1})_{m\geq 0}$, $(M^{m,2})_{m\geq 0}$ and $(M^{m,3})_{m\geq 0}$ are Cauchy sequences in $\widetilde S_q^p$, $\widetilde D_{q, q}^p$ and $\widetilde D_{p, q}^p$ respectively. By construction, each of $M^{m, i}$, $m\geq 0$, $i=1,2,3$, has finitely many jumps occurring in $\{\tau_{0},\ldots,\tau_{n_{m}}\}$, so by Theorem \ref{thm:acessjumpsinTL^q} the sequences $(M^{m,1})_{m\geq 0}$, $(M^{m,2})_{m\geq 0}$ and $(M^{m,3})_{m\geq 0}$ are Cauchy in $\mathcal{M}^{\rm acc}_{p,q}$ as well. Due to Proposition \ref{prop:purdisaccjumpsisBanachspace} there exist $\widetilde M^1$, $\widetilde M^2$ and $\widetilde M^3$ such that $M^{m,i}\to \widetilde M^i$ in $\mathcal{M}^{\rm acc}_{p,q}$ as $m\to \infty$ for each $i=1,2,3$. Since $M^{m,1} + M^{m,2} + M^{m,3}\to M$ in $\mathcal{M}^{\rm acc}_{p,q}$ as $m\to \infty$ by Lemma~\ref{lem:limofM^nisM}, it follows that $M=\widetilde M^1+\widetilde M^2+\widetilde M^3$.

Let us now show that the jumps of $\widetilde M^1$, $\widetilde M^2$ and $\widetilde M^3$ are exhausted by $\mathcal T = (\tau_n)_{n\geq 0}$. Indeed, assume that for some $i=1,2,3$ there exists a predictable stopping time $\tau$ such that $\mathbb P\{\Delta \widetilde M^i_{\tau}\neq 0, \tau \notin \{\tau_0,\tau_1,\ldots\}\}>0$. Then by separability of $X=L^q(S)$ there exists an $x^*\in X^*$ such that
\begin{equation}\label{eq:goodx^*forlimarg}
 \mathbb P\{\langle \Delta \widetilde M^i_{\tau}, x^*\rangle\neq 0, \tau \notin \{\tau_0,\tau_1,\ldots\}\}>0
\end{equation}
and so, by the Burkholder-Davis-Gundy inequality,
\begin{equation}\label{eq:goodx^*application}
 \begin{split}
   \mathbb E |\langle (M^{m, i}-\widetilde M^{i})_{\infty}, x^*\rangle|^p &\eqsim_p \mathbb E [\langle M^{m, i}-\widetilde M^{i}, x^*\rangle]^{\frac p2}_{\infty} \\
 &= \mathbb E \Bigl(\sum_{u\geq 0}| \langle \Delta(M^{m, i}-\widetilde M^{i})_u, x^*\rangle|^2\Bigr)^{\frac p2}\\
 &\geq \mathbb E |\langle \Delta \widetilde M^i_{\tau}, x^*\rangle|^p \mathbf 1_{\tau \notin \{\tau_0,\tau_1,\ldots\}},
 \end{split}
\end{equation}
where the final inequality holds as $\mathbb P\{\Delta M^{m,i}_{\tau}\neq 0, \tau \notin \{\tau_0,\tau_1,\ldots\}\}=0$. But the last expression in \eqref{eq:goodx^*application} does not vanish as $m\to \infty$ because of \eqref{eq:goodx^*forlimarg}, which contradicts with the fact that $M^{m,i}\to \widetilde M^i$ in $\mathcal{M}^{\rm acc}_{p,q}$.

By monotone convergence,
\begin{align*}
 \|\widetilde M^1\|_{\widetilde S_q^p}^p&= \mathbb E\Bigl\| \Bigl({\sum_{n\geq 0} \mathbb E_{\mathcal F_{\tau_n-}}|(\Delta \widetilde M^1(\omega)(s))_{\tau_n}|^2 } \Bigr)^{\frac 12}\Bigr\|^p_{L^q(S)} \\
 &=\lim_{m\to \infty}\mathbb E\Bigl\| \Bigl({\sum_{n=0}^{n_m} \mathbb E_{\mathcal F_{\tau_n-}}|(\Delta \widetilde M^1(\omega)(s))_{\tau_n}|^2 } \Bigr)^{\frac 12}\Bigr\|^p_{L^q(S)}\\
 &= \lim_{m\to \infty}\|M^{m,1}\|^p_{\widetilde S_q^p},
\end{align*}
and the last expression is bounded due to the fact that $M^{m,1}$ is a Cauchy sequence in $\widetilde S_q^p$. By the same reasoning $\widetilde M^2\in \widetilde D_{q, q}^p$ and $\widetilde M^3 \in \widetilde D_{p, q}^p$, so $M \in  \mathcal A_{p,q}$. This completes the proof.
\end{proof}

Theorem~\ref{thm:acessjumpsL^q} and Lemma~\ref{lemma:intcontcontintdisdis} yield the following sharp estimates.
\begin{corollary}\label{thm:mainforaccesjumps}
 Let $1<p,q<\infty$, $M:\mathbb R_+ \times \Omega \to H$ be a purely discontinuous \mbox{$L^p$-mar}\-tin\-gale with accessible jumps, $X = L^q(S)$, $\Phi:\mathbb R_+ \times \Omega \to \mathcal L(H,X)$ be elementary predictable. Then for all $t\geq 0$ one has that
 \begin{equation*}
 \Bigl( \mathbb E \sup_{0\leq s\leq t}\| (\Phi \cdot M)_{t}\|^p_{L^q(S)} \Bigr)^{\frac 1p} \eqsim_{p,q}\|(\Phi \mathbf 1_{[0,t]}) \cdot M\|_{\mathcal A_{p,q}},
 \end{equation*}
where $\mathcal A_{p,q}$ is as given in \eqref{eq:Apq}.
\end{corollary}

\subsection{Quasi-left continuous purely discontinuous martingales}
\label{sec:ItoPDQLC}

We now turn to estimates for the stochastic integral $\Phi\cdot M$ in the case that $M$ is a purely discontinuous quasi-left continuous local martingale. We will first show in Lemma~\ref{lemma:whyPhijisint} that one can (essentially) represent $\Phi\cdot M$ as a stochastic integral $\Phi_H\star \bar{\mu}^M$, where $\bar{\mu}^M$ is the compensated version of the jump measure $\mu^M$ of $M$. Afterwards, in Theorem~\ref{thm:mainintranmeas} we prove sharp bounds for stochastic integrals of the form $f\star \bar{\mu}$, where $\mu$ is any integer-valued random measure with a compensator that is non-atomic in time. By combining these two observations, we immediately find sharp bounds for $\Phi\cdot M$.

\subsubsection{Facts on random measures}
 
Let us start by recalling some necessary definitions and facts concerning random measures. Let $(J, \mathcal J)$ be a measurable space. Then a family $\mu = \{\mu(\omega; \ud t, \ud x), \omega \in \Omega\}$ of nonnegative measures on $(\mathbb R_+ \times J; \mathcal B(\mathbb R_+)\otimes \mathcal J)$ is called a {\it random measure}. A random measure $\mu$ is called {\it integer-valued} if it takes values in $\mathbb N\cup\{\infty\}$, i.e.\ for each $A \in \mathcal B(\mathbb R_+)\otimes \mathcal F\otimes \mathcal J$ one has that $\mu(A) \in \mathbb N\cup\{\infty\}$ a.s., and if $\mu(\{t\}\times J)\in \{0,1\}$ a.s.\ for all $t\geq 0$. We say that $\mu$ is \emph{non-atomic in time} if $\mu(\{t\}\times J) = 0$ a.s.\ for all $t\geq 0$.  

Recall that $\mathcal P$ and $\mathcal O$ denote the predictable and optional $\si$-algebra on $\R_+\ti \Om$ and $\widetilde{\mathcal P}= \mathcal P \otimes \mathcal J$ and $\widetilde {\mathcal O}:=\mathcal O \otimes \mathcal J$ are the induced $\sigma$-algebras on $\widetilde {\Omega} = \mathbb R_+ \times \Omega \times J$. A~process $F:\mathbb R_+\times \Omega \to \mathbb R$ is called {\it optional} if it is $\mathcal O$-measurable. A random measure $\mu$ is called {\it optional} (resp. {\it predictable}) if for any $\widetilde{\mathcal O}$-measurable (resp. $\widetilde{\mathcal P}$-measurable) nonnegative $F:\mathbb R_+ \times \Omega \times J \to \mathbb R_+$ the stochastic integral 
\[
 (F\star \mu)_t(\omega) := \int_{\mathbb R_+\times J}\mathbf 1_{[0,t]}(s)F(s,\omega,x) \mu(\omega;\ud s, \ud x),\;\;\; t\geq 0, \ \omega \in \Omega,
\]
as a function from $\mathbb R_+ \times \Omega$ to $\overline{\mathbb R}_+$ is optional (resp.\ predictable).  Let $X$ be a Banach space. 

Then we can extend stochastic integration to \mbox{$X$-valued} processes in the following way. Let $F:\mathbb R_+ \times \Omega \times J  \to X$, $\mu$ be a random measure. The integral
\[
 (F\star \mu)_t := \int_{\mathbb R_+ \times J} F(s,\cdot, x)\mathbf 1_{[0,t]}(s)\mu(\cdot; \ud s, \ud x),\;\;\; t\geq 0,
\]
is well-defined and optional (resp.\ predictable) if $\mu$ is optional (resp.\ predictable), $F$ is $\widetilde {\mathcal O}$-strongly-measurable (resp.\ $\widetilde {\mathcal P}$-strongly-measurable), and $(\|F\|\star \mu)_{\infty}$ is a.s.\ bounded.

A random measure $\mu$ is called $\widetilde{\mathcal P}$-$\sigma$-finite if there exists an increasing sequence of sets $(A_n)_{n\geq 1}\subset \widetilde{\mathcal P}$ such that $\int_{\mathbb R_+ \times J} \mathbf 1_{A_n}(s,\omega,x)\mu(\omega; \ud s, \ud x)$ is finite a.s.\ and $\cup_n A_n = \mathbb R_+ \times \Omega \times J$. According to \cite[Theorem II.1.8]{JS} every $\widetilde{\mathcal P}$-$\sigma$-finite optional random measure $\mu$ has a {\it compensator}: a~unique $\widetilde{\mathcal P}$-$\sigma$-finite predictable random measure $\nu$ such that $\mathbb E (W \star \mu)_{\infty} = \mathbb E (W \star \nu)_{\infty}$ for each $\widetilde{\mathcal P}$-measurable real-valued nonnegative $W$. We refer the reader to \cite[Chapter II.1]{JS} for more details on random measures. For any optional $\widetilde{\mathcal P}$-$\sigma$-finite measure $\mu$ we define the associated compensated random measure by $\bar{\mu} = \mu -\nu$. 

For each $\widetilde{\mathcal P}$-strongly-measurable $F:\mathbb R_+ \times \Omega \times J \to X$ such that $\mathbb E (\|F\|\star \mu)_{\infty}< \infty$ (or, equivalently, $\mathbb E (\|F\|\star \nu)_{\infty}<\infty$, see the definition of a compensator above) we can define a process $F\star \bar{\mu}$ by $F \star \mu - F \star \nu$. The reader should be warned that in the literature $F \star \bar{\mu}$ is often used to denote the integral of $F$ over the whole $\mathbb R_+$ (i.e.~$(F \star \bar{\mu})_{\infty}$ in our notation). 

We will also need the following lemma.

\begin{lemma}\label{lemma:restmu_1tomu_2}
Let $A \in \widetilde{\mathcal P}$, $\mu_1$ be a $\widetilde {\mathcal P}$-$\sigma$-finite random measure with a compensator $\nu_1$. Then $\mu_2 = \mu_1 \mathbf 1_{A}$ is a $\widetilde {\mathcal P}$-$\sigma$-finite random measure and $\nu_2 = \nu_1 \mathbf 1_{A}$ is a compensator for $\mu_2$.  
\end{lemma}

\begin{proof}
$\mu_2$ is $\widetilde {\mathcal P}$-$\sigma$-finite since $\mu_2 \leq \mu_1$ a.s. Moreover, $\mu_2$ is optional. Indeed, let $F:\mathbb R_+\times \Omega \times J\to \mathbb R_+$ be $\widetilde{\mathcal O}$-measurable. Then 
 $$
 F \star \mu_2 = F \star (\mu_1 \mathbf 1_{A}) = (F\mathbf 1_{A})\star \mu_1,
 $$
 and the last process is obviously optional.
 
Now let us show that $\nu_2 = \nu_1 \mathbf 1_{A}$. Let $F:\mathbb R_+\times \Omega \times J\to \mathbb R$ be simple \mbox{$\widetilde {\mathcal P}$-mea}\-su\-rable. Since $\mu_1$ is $\widetilde {\mathcal P}$-$\sigma$-finite, so are $\nu_1,\mu_1,\nu_2$. Hence, we can assume without loss of generality that $F \star \mu_1$ exists and is integrable. Then $F\star \mu_2 = F \star (\mu_1 \mathbf 1_{A}) = (F\mathbf 1_{A})\star \mu_1$ exists and is integrable. Moreover, 
$$\E(F\star \mu_2)_{\infty} = \E((F\mathbf 1_{A})\star \mu_1)_{\infty} = \E((F\mathbf 1_{A})\star \nu_1)_{\infty} = \E(F\star \nu_2)_{\infty},$$ 
so $\nu_2$ is a compensator of $\mu_2$.
\end{proof}

\subsubsection{Representation of the stochastic integral}

To any purely discontinuous local martingale $M$ we can associate an integer-valued random measure $\mu^M$ on $\mathcal B(\mathbb R_+)\otimes \mathcal B(H)$ by setting 
\begin{equation*}
 \mu^M(\omega; B\times A) := \sum_{u\in B} \mathbf 1_{ A\setminus\{0\}}(\Delta M_u(\omega) ),\;\;\; \omega \in \Omega,
\end{equation*}
for each $B \in \mathcal B(\mathbb R_+)$, $A \in \mathcal B(H)$. That is, $ \mu^M(\omega; B\times A)$ counts the number of jumps within the time set $B$ with size in $A$ on the trajectory belonging to the sample point~$\omega$.

Recall that a process $M :\mathbb R_+ \times \Omega \to H$ is called quasi-left continuous if $\Delta M_{\tau} = 0$ a.s.\ on the set $\{\tau<\infty\}$ for each predictable stopping time $\tau$ (see \cite[Chapter I.2]{JS} for more information). If $M:\mathbb R_+ \times \Omega \to H$ is a quasi-left continuous local martingale, then $\mu^M$ is $\widetilde {\mathcal P}$-$\sigma$-finite and there exists a compensator~$\nu^M$ (see e.g.\ \cite[Proposition II.1.16]{JS} and \cite[Theorem 25.22]{Kal}). If $M$ is, in addition, purely discontinuous, then the following characterization holds thanks to \cite[Corollary~II.1.19]{JS}.
\begin{lemma}\label{lemma:corII.1.19}
 Let $H$ be a separable Hilbert space and $M:\mathbb R_+ \times \Omega \to H$ be a~purely discontinuous local martingale. Let $\mu^M$ and $\nu^M$ be the associated integer-valued random measure and its compensator. Then $M$ is quasi-left continuous if and only if $\nu^M$ is non-atomic in time.
\end{lemma}

Let us now show that $\Phi\cdot M$ can (essentially) be represented as a stochastic integral with respect to $\bar{\mu}_M$.
\begin{lemma}\label{lemma:whyPhijisint}
 Let $X$ be a Banach space, $H$ be a Hilbert space, $1\leq p<\infty$, $M:\mathbb R_+ \times \Omega \to H$ be a purely discontinuous quasi-left continuous local martingale, and $\Phi:\mathbb R_+ \times \Omega \to \mathcal L(H,X)$ be elementary predictable. Define $\Phi_H:\mathbb R_+ \times \Omega \times H \to X$ by
\[
 \Phi_H(t,\omega, h) := \Phi(t,\omega)h,\;\;\; t\geq 0, \omega \in \Omega, h\in H.
\]
Then there exists an increasing sequence $(A_n)_{n\geq 1}\in \widetilde{\mathcal P}$ such that $\cup_n A_n =\mathbb R_+ \times \Omega \times J$, $\Phi_H \mathbf 1_{A_n}$ is integrable with respect to $\bar{\mu}^M$ for each $n\geq 1$, and 
\begin{itemize}
\item[(i)] if $\Phi \cdot M \in L^p(\Omega; X)$ then $(\Phi_H\mathbf 1_{A_n}) \star \bar{\mu}^M \to \Phi \cdot M$ in $L^p(\Omega; X)$;
\item[(ii)] if $\Phi \cdot M \not\in L^p(\Omega; X)$ then $\|(\Phi_H\mathbf 1_{A_n}) \star \bar{\mu}^M\|_{L^p(\Omega;X)}\to\infty$ for $n\to \infty$.
\end{itemize}
\end{lemma}
\begin{proof}
For each $k, l\geq 1$ we define a stopping time $\tau_{k,l}$ by
\[
 \tau_{k,l} = \inf\{t\in \mathbb R_+: \#\{s\in [0,t] : \|\Delta M_s\|\in [1/k, k]\} = l\}.
\]
Since $M$ has c\`adl\`ag trajectories, $\tau_{k,l}$ is a.s.\ well-defined and takes its values in $[0,\infty]$. Moreover, $\tau_{k,l}\to \infty$ for each $k\geq 1$ a.s.\ as $l\to \infty$. 

Set $B_k = \{h\in H:\|h\|\in [1/k,k]\}$. For each $k,l\geq 1$ define $A_{k,l} = \mathbf 1_{[0,\tau_{k.l}]\times B_k}\subset \widetilde{\mathcal P}$. Then $\Phi_H \mathbf 1_{A_{k,l}}$ is integrable with respect to $\mu^M$. Indeed,~a.s.
\[
 \bigl((\Phi_H \mathbf 1_{A_{k,l}}) \star \mu^M\bigr)_{\infty} \leq \sup\|\Phi\| k \bigl(\mathbf 1_{A_{k,l}}\star \mu^M\bigr)_{\infty} \leq \sup\|\Phi\| kl.
\]

Since $\tau_{k,l}\to \infty$ for each $k\geq 1$ a.s.\ as $l\to \infty$, we can find a subsequence $(\tau_{k_n,l_n})_{n\geq 1}$ such that $k_n\geq n$ for each $n\geq 1$ and $\inf_{m\geq n} \tau_{k_m, l_m}\to \infty$ a.s.\ as $n\to \infty$. Let $\tau_n = \inf_{m\geq n} \tau_{k_m, l_m}$ and define $(A_n)_{n\geq 1}\subset \widetilde{\mathcal P}$ by
\[
 A_n = \mathbf 1_{[0,\tau_n]\times B_n}.
\]
Then $\cup_n A_n = \mathbb R_+ \times \Omega \times J$ and $\Phi_H\mathbf 1_{A_n}$ is integrable with respect to $\bar{\mu}^M$ for all $n\geq 1$.

Now prove that $(\Phi_H\mathbf 1_{A_n}) \star \bar{\mu}^M \to \Phi \cdot M$ in $L^p(\Omega; X)$.
Since $\Phi$ is simple, it takes its values in a finite dimensional subspace of $X$, so we can endow $X$ with a Euclidean norm $\tnorm{\cdot}$. First suppose that $(\Phi \cdot M)_{\infty} \notin L^p(\Omega; X)$. By the Burkholder-Davis-Gundy inequality this is equivalent to the fact that $[\Phi\cdot M]_{\infty}^{\frac 12}\notin L^p(\Omega; X)$. Notice that
\begin{align*}
  \mathbb E\tnorm{(\Phi_H \mathbf 1_{A_n})\star \bar{\mu}^M)_{\infty}}^p & \eqsim_p \mathbb E\bigl[(\Phi_H\mathbf 1_{A_n})\star \bar{\mu}^M\bigr]^{\frac p2}_{\infty} \\
 & = \mathbb E \Bigl(\sum_{t\in [0, \tau_n]}\tnorm{\Delta(\Phi \cdot M)_t}^2 \mathbf 1_{\|\Delta M_t\| \in [1/n,n]}\Bigr)^{\frac p2},
\end{align*}
and the last expression monotonically goes to infinity since $\tau_n \to \infty$ a.s.\ and
\[
 \mathbb E \Bigl(\sum_{t\geq 0}\|\Delta(\Phi \cdot M)_t\|^2 \Bigr)^{\frac p2} = \mathbb E [\Phi\cdot M]_{\infty}^{\frac p2} = \infty.
\]
So if $(\Phi \cdot M)_{\infty} \notin L^p(\Omega; X)$, then $\bigl\|\bigl((\Phi_H \mathbf 1_{A_n})\star \bar{\mu}^M\bigr)_{\infty}\bigr\|_{L^p(\Omega; X)} \to \infty$ as $n\to \infty$.

Now assume that $(\Phi \cdot M)_{\infty} \in L^p(\Omega; X)$. Then
\begin{align*}
& \mathbb E \tnorm{(\Phi \cdot M)_{\infty} - ( (\Phi_H\mathbf 1_{A_n}) \star \bar{\mu}^M)_{\infty}}^p \eqsim_p \mathbb E [\Phi \cdot M - (\Phi_H\mathbf 1_{A_n}) \star \bar{\mu}^M]_{\infty}^{\frac p2}\\
 &\qquad =\mathbb E \Bigl( \sum_{t\in [0, \tau_n]}\tnorm{\Delta(\Phi \cdot M)_t}^2 \mathbf 1_{\|\Delta M_t\| \notin [1/n,n]}\\ 
 &\qquad \qquad \qquad +\sum_{t\in (\tau_n,\infty)}\tnorm{\Delta(\Phi \cdot M)_t}^2 \Bigr)^{\frac p2} \to 0,\;\; n\to \infty
\end{align*}
by the dominated convergence theorem.
\end{proof} 
By Lemmas~\ref{lemma:corII.1.19} and \ref{lemma:whyPhijisint} it now suffices to obtain sharp bounds for the stochastic integral $(F\star \bar{\mu})_{\infty}$, where $\mu$ is any optional integer-valued random measure whose compensator $\nu$ is non-atomic in time.

\subsection{Integrals with respect to random measures}

Throughout this section, \emph{$\mu$~denotes an optional integer-valued random measure whose compensator $\nu$ is non-atomic in time}, i.e., $\nu(\{t\}\times J) = 0$ a.s.\ for all $t\geq 0$. 
The following result was first shown in \cite[Theorem 1]{Nov75}.
\begin{lemma}[A.A.\ Novikov]\label{lemma:Nov}
 Let $f:\mathbb R_+ \times \Omega \times J \to \mathbb R$ be $\widetilde{\mathcal P}$-measurable. Then
 \begin{align*}
  \mathbb E | f\star \bar{\mu}|^p &\lesssim_{p} \mathbb E |f|^p\star \nu \text{ if }\; 1\leq p\leq 2,\\
  \mathbb E |f\star \bar{\mu}|^p &\lesssim_{p} (\mathbb E |f|^2 \star \nu)^{\frac p2}+ \mathbb E |f|^p\star \nu \text{ if }\;  p\geq 2.
 \end{align*}
\end{lemma}
The following lemma easily follows from \cite[Theorem II.1.33]{JS} (or from \cite[p.98]{Grig71} and \cite[(6)]{Nov75} as well).
\begin{lemma}\label{lemma:seqmoments}
 Let $H$ be a Hilbert space,  $f:\mathbb R_+ \times \Omega \times J \to H$ be $\widetilde{\mathcal P}$-measurable. Then
 \begin{equation}\label{eq:seqmoments}
  \mathbb E \|f\star\bar{\mu}\|^2 =  \mathbb E \|f\|^2 \star \nu.
 \end{equation}
 Equivalently, for each $\widetilde{\mathcal P}$-measurable $f,g:\mathbb R_+ \times \Omega \times J \to H$ such that $\mathbb E \|f\|^2 \star \nu<\infty$ and $\mathbb E \|g\|^2 \star \nu<\infty$
 \begin{equation}\label{eq:seqmomentsinnprod}
  \mathbb E \langle f\star\bar{\mu},g\star\bar{\mu}\rangle=  \mathbb E \langle f, g\rangle \star \nu.
 \end{equation}
\end{lemma}
\begin{proof}
The case $H = \mathbb R$ can be deduced from \cite[II.1.34]{JS} as $\nu$ is assumed to be non-atomic in time. By applying this special case coordinate-wise, we obtain the general case. 
\end{proof}

\begin{corollary}\label{lemma:equalityfordual}
 Let $X$ be a Banach space, $1<p<\infty$, $\mu$ be a~random measure, $\nu$ be the corresponding compensator, $F:\mathbb R_+ \times \Omega \times J \to X$ and  $G:\mathbb R_+ \times \Omega \times J \to X^*$ be simple $\widetilde {\mathcal P}$-measurable functions. Then for each $A \in \widetilde{\mathcal P}$ such that $\mathbb E (\mathbf 1_{A} \star \mu)_{\infty} <\infty$ the stochastic integrals $(F\mathbf 1_{A})\star \bar{\mu}$ and $(G\mathbf 1_{A})\star \bar{\mu}$ are well-defined and
 \begin{equation}\label{eq:intequal}
  \mathbb E \langle (F\mathbf 1_{A})\star \bar{\mu} , (G\mathbf 1_{A}) \star \bar{\mu}\rangle = \mathbb E(\langle F, G\rangle\mathbf 1_{A})\star \nu.
 \end{equation}
\end{corollary}

\begin{proof}
Without loss of generality we can assume that $X$ is finite dimensional. By Lemma~\ref{lemma:restmu_1tomu_2}, we can also redefine $F := F\mathbf 1_{A}$, $G:= G\mathbf 1_{A}$. First notice that since $\|F\|_{\infty}, \|G\|_{\infty}<\infty$ and $\mathbb E \mu (F\neq 0), \mathbb E \mu(G\neq 0)<\infty$, both integrals $F\star \bar{\mu}$ and $G \star \bar{\mu}$ exist. Moreover, every finite dimensional space is isomorphic to a Hilbert space, so by Lemma \ref{lemma:seqmoments} both $F\star \bar{\mu}$ and $G \star \bar{\mu}$ are $L^2$-integrable, and therefore the left-hand side of \eqref{eq:intequal} is well-defined.

Now let $d$ be the dimension of $X$, $(x_k)_{k=1}^d$ and $(x_k^*)_{k=1}^d$ be bases of $X$ and~$X^*$ respectively. Then there exist simple $\widetilde {\mathcal P}$-measurable $F^1,\ldots,F^d, G^1,\ldots,G^d:\mathbb R_+ \times \Omega \times J \to \mathbb R$ such that $F = F^1x_1 + \cdots + F^d x_d$ and $G = G^1x_1^* + \cdots + G^d x_d^*$. Now \eqref{eq:seqmomentsinnprod} implies
\begin{align*}
 \mathbb E \langle F\star \bar{\mu} , G \star \bar{\mu}\rangle &= \sum_{k,l=1}^d \langle x_k, x_l^*\rangle\mathbb E\bigl (F^k\star \bar{\mu} \cdot G^l\star \bar{\mu}\bigr)=\sum_{k,l=1}^d \langle x_k, x_l^*\rangle \mathbb E(F^k G^l)\star \nu\\
 &=\mathbb E\Bigl(\sum_{k,l=1}^d \langle x_k, x_l^*\rangle F^k G^l\Bigr)\star \nu  = \mathbb E\langle F, G\rangle\star \nu.
\end{align*}
\end{proof}

The following proposition extends Novikov's inequalities presented in Lemma~\ref{lemma:Nov} in the case that $\nu(\mathbb R_+ \times J)\leq 1$ a.s. If $X=L^q(S)$ this result can be seen as a special case of Theorem \ref{thm:mainintranmeas} below. In the proof we will use the measure $\mathbb P \times \nu$ on $\mathcal B(\mathbb R_+)\otimes \mathcal F \otimes \mathcal J$ that is defined by setting 
\[
\mathbb P \times \nu\Big(\bigcup_{i=1}^n A_i\times B_i\Big) := \sum_{i=1}^n \mathbb E(\mathbf 1_{A_i}\nu(B_i)),
\]
for disjoint $A_i \in \mathcal F$ and disjoint $B_i \in \mathcal B(\mathbb R_+) \otimes \mathcal J$, and extending $\mathbb P \times \nu$ to $\mathcal B(\mathbb R_+)\otimes \mathcal F \otimes \mathcal J$ via the Carath\'eodory extension theorem.

\begin{proposition}\label{prop:Nov}
Suppose that $\nu(\mathbb R_+ \times J)\leq 1$ a.s. Let $X$ be a Banach space and $f:\mathbb R_+ \times \Omega \times J \to X$ be simple $\widetilde {\mathcal P}$-measurable. Then for all $1<p<\infty$
 \[
  \mathbb E \|F \star \bar{\mu}\|^p \eqsim_p \mathbb E \|F\|^p  \star \nu.
 \]
\end{proposition}

\begin{proof}
We first prove $\lesssim_p$, and later deduce $\gtrsim_p$ by a duality argument.

{\it Step 1: upper bounds.}
 The case $X = \mathbb R$ follows from Lemma \ref{lemma:Nov} and the fact that $\|\cdot\|_{L^2(\mathbb R_+\times \Omega \times J, \mathbb P\otimes \nu)} \leq \|\cdot\|_{L^p(\mathbb R_+\times \Omega \times J, \mathbb P\otimes \nu)}$ for each $p\geq 2$ since $\mathbb P\otimes \nu(\mathbb R_+ \times \Omega \times J)\leq 1$. Now let $X$ be a general Banach space. Then
 \begin{align*}
   \mathbb E \|F \star \bar{\mu}\|^p &\stackrel{(i)}\lesssim_p    \mathbb E \|F \star \mu\|^p +    \mathbb E \|F \star\nu\|^p \stackrel{(ii)}\leq \mathbb E \bigl|\|F\| \star \mu\bigr|^p +    \mathbb E \bigl|\|F\| \star\nu\bigr|^p\\
   &\stackrel{(iii)}\lesssim_p  \mathbb E \bigl|\|F\| \star \bar{\mu}\bigr|^p +    \mathbb E \bigl|\|F\| \star\nu\bigr|^p \stackrel{(iv)}\lesssim_p \mathbb E \|F\|^p  \star \nu,
 \end{align*}
where $(i)$ and $(iii)$ follow from the fact that $\bar{\mu} = \mu - \nu$ and the triangle inequality, $(ii)$ follows from \cite[Proposition 1.2.2]{HNVW1}, and $(iv)$ follows from the real-valued case and the fact that a.s.
$$
\|\cdot\|_{L^1(\mathbb R_+ \times J; \nu)}\leq \|\cdot\|_{L^p(\mathbb R_+ \times J; \nu)}.
$$

{\it Step 2: lower bounds.} We can assume that $X$ is finite dimensional since $F$ is simple. Let $Y = L^p(\mathbb R_+ \times \Omega \times J, \mathbb P \otimes \nu; X)$. Recall that by \cite[Proposition 1.3.3]{HNVW1} $Y^* = L^{p'}(\mathbb R_+ \times \Omega \times J, \mathbb P \otimes \nu; X^*)$ and $(L^p(\Omega; X))^* = L^{p'}(\Omega; X^*)$. Therefore due to the upper bounds from Step 1 and Corollary \ref{lemma:equalityfordual}
\begin{align*}
 (\mathbb E \|F\|^p \star \nu)^{\frac 1p} &= \sup_{G\in Y^*:\|G\|\leq 1} \mathbb E \langle F, G\rangle\star \nu = \sup_{G\in Y^*:\|G\|\leq 1} \mathbb E \langle F \star \bar{\mu}, G \star \bar{\mu}\rangle\\
 &\lesssim_{p} \sup_{\xi \in L^{p'}(\Omega; X^*):\|\xi\|\leq 1} \mathbb E\langle  F \star \bar{\mu},\xi\rangle = (\mathbb E \|F\star \bar{\mu}\|^p)^{\frac 1p}.
\end{align*}
\end{proof}
\begin{remark}
The condition $\nu(\mathbb R_+ \times J)\leq 1$ a.s.\ is necessary in general. Indeed, let $N$ be a Poisson process with intensity parameter $\lambda$ and let $\mu$ be the random measure on $\R_+\times \{0\}$ defined by $\mu([0,t]\times\{0\})=N_t$. Then the corresponding compensator $\nu$ satisfies $\nu([0,t]\times\{0\})=\lambda t$. In particular,
\[
 \mathbb E |\mathbf 1_{[0, 1]}\star \bar{\mu}|^4 = \mathbb E |N -\lambda |^4 = \sum_{k=0}^{\infty} \frac{(k-\lambda)^4\lambda ^k e^{-\lambda}}{k!} = \lambda(3\lambda +1),
\]
which is not comparable with $\mathbb E |\mathbf 1_{[0,1]}|^4\star \nu = \lambda$ if $\lambda$ is large.

The condition $\nu(\mathbb R_+ \times J)\leq 1$ a.s.\ is however not needed for the upper bounds if $1\leq p\leq 2$ and $X$ is a Hilbert space. Indeed, for $p=1$ 
 \[
  \mathbb E \|F\star \bar{\mu}\| \leq \mathbb E \|F\star \mu\| +  \mathbb E \|F\star \nu\| \leq  \mathbb E \|F\|\star \mu +  \mathbb E \|F\|\star \nu = 2\mathbb E \|F\|\star \nu,
 \]
and for case $p=2$ follows immediately from Lemma \ref{lemma:seqmoments}:
\[
 \mathbb E \|F\star \bar{\mu}\|^2 = \mathbb E \|F\|^2\star \nu.
\]
Therefore by the vector-valued Riesz-Thorin theorem \cite[Theorem 2.2.1]{HNVW1} for each $1\leq p\leq 2$
\[
 (\mathbb E \|F\star \bar{\mu}\|^p)^{\frac 1p}\leq 2 (\mathbb E \|F\|^p\star\nu )^{\frac 1p}.
\]
 \end{remark}

\begin{corollary}\label{cor:condexpmunu}
Suppose that $\nu(\mathbb R_+ \times J)\leq 1$ a.s. Let $X$ be a Banach space, $f:\mathbb R_+ \times \Omega \times J \to X$ be simple $\widetilde{ \mathcal P}$-measurable. Then for each $p\in (1,\infty)$ a.s.
 \begin{equation}\label{eq:condexpmunu}
   (\mathbb E \|F \star \bar{\mu}\|^p|\mathcal F_0) \eqsim_p (\mathbb E \|F\|^p  \star \nu|\mathcal F_0).
 \end{equation}
\end{corollary}

\begin{proof}
 Fix $A\in \mathcal F_0$. Then by Lemma~\ref{lemma:restmu_1tomu_2} and Proposition~\ref{prop:Nov}
 \begin{align*}
  \mathbb E (\|F \star \bar{\mu}\|^p\cdot \mathbf 1_{A}) = \mathbb E \|(F \cdot \mathbf 1_{A})\star \bar{\mu}\|^p \eqsim_p\mathbb E \|F \cdot \mathbf 1_{A}\|^p  \star \nu = \mathbb E (\|F\|^p  \star \nu \cdot \mathbf 1_{A}).
 \end{align*}
Since $A$ is arbitrary, \eqref{eq:condexpmunu} holds.
\end{proof}

For each $m\geq 1$ let $\mathcal P_m$ be the $\sigma$-field on $\mathbb R_+\times \Omega$ generated by all $\mathcal P$-measurable $f:\mathbb R_+\times \Omega \to \mathbb R $ such that $f\big|_{ (\frac{n}{2^m}, \frac{n+1}{2^m}]\times \Omega}$ is $\mathcal B\bigl((\frac{n}{2^m}, \frac{n+1}{2^m}]\bigr) \otimes \mathcal F_{\frac{n}{2^m}}$-measurable for each $n\geq 0$. Then the following theorem holds.

\begin{theorem}\label{thm:condexpwrtP_m}
 Let $f:\mathbb R_+\times \Omega \to \mathbb R$ be bounded and $\mathcal P$-measurable. Then for each $m\geq 1$
 \begin{equation}\label{eq:condexppred}
    \mathbb E(f|\mathcal P_m)(s) = \sum_{n\geq 0}\mathbb E\bigl(f(s)\big| \mathcal F_{\frac{n}{2^m}}\bigr),\;\;\; s\in \Bigl(\frac{n}{2^m}, \frac{n+1}{2^m}\Bigr],n\geq 0.
 \end{equation}
Moreover, $\mathbb E(f|\mathcal P_m) \to f$ a.s.\ on $\mathbb R_+\times \Omega$ as $m\to \infty$.
\end{theorem}
\begin{proof}
 Let us first show \eqref{eq:condexppred}. Fix $m\geq 1$. Fix a simple $\mathcal P_m$-measurable process $g:\mathbb R_+ \times \Omega \to \mathbb R$. Then for each $n\geq 0$ and $s\in (\frac{n}{2^m}, \frac{n+1}{2^m}]$ a random variable $g(s)$ is $\mathcal F_{\frac{n}{2^m}}$-measurable.
Define $\widetilde f:\mathbb R_+\times \Omega \to \mathbb R$ by
\[
 \widetilde f(s) = \sum_{n\geq 0}\mathbb E\bigl(f(s)\big| \mathcal F_{\frac{n}{2^m}}\bigr)\mathbf 1_{s \in (\frac{n}{2^m}, \frac{n+1}{2^m}]},\;\;\; s\geq 0.
\]
Then for each $n\geq 0$ and $s \in (\frac{n}{2^m}, \frac{n+1}{2^m}]$
\begin{align*}
 \mathbb E \bigl[(f(s) - \widetilde f(s))g(s)\bigr] &= \mathbb E\Bigl[ \mathbb E\bigl((f(s) - \widetilde f(s))g(s)\big|\mathcal F_{\frac{n}{2^m}}\bigr) \Bigr]\\
 &=\mathbb E\Bigl[ \mathbb E\bigl((f(s) - \widetilde f(s))\big|\mathcal F_{\frac{n}{2^m}}\bigr)g(s) \Bigr] = 0.
\end{align*}
Therefore
 \begin{align*}
  \mathbb E \int_{\mathbb R_+}\bigl(f(s) -\widetilde f(s)\bigr)g(s)\ud s =\int_{\mathbb R_+} \mathbb E \Bigl[\bigl(f(s) - \widetilde f(s)\bigr)g(s)\Bigr]\ud s=0,
 \end{align*}
and hence \eqref{eq:condexppred} holds. Now notice that $(\mathcal P_m)_{m\geq 1}$ forms a filtration on $\mathbb R_+\times \Omega$, and obviously $\sigma\{\cup_m\mathcal P_m\} = \mathcal P$. Therefore the second part of the theorem follows from the martingale convergence theorem (see e.g.\ \cite[Theorem 7.23]{Kal}).
\end{proof}

\begin{corollary}\label{cor:Ffromcondexpex}
 Let $F:\mathbb R_+ \times \Omega \to \mathbb R_+$ be an increasing predictable function such that $F(t)-F(s) \leq C(t-s)$ a.s.\ for all $0\leq s\leq t$ and for some fixed constant $C\geq 0$  and $F(0)=0$ a.s. Then for each fixed $T\geq 0$
 \[
  F(T) = \lim_{m\to\infty} \sum_{n=0}^{[2^m T]-1}\mathbb E\Bigl[F\Bigl(\frac{n+1}{2^m}\Bigr) - F\Bigl(\frac{n}{2^m}\Bigr)\Big|\mathcal F_{\frac {n}{2^m}}\Bigr],
 \]
where the last limit holds a.s.\ and in $L^p(\Omega)$ for all $1<p<\infty$.
\end{corollary}

For the proof we will need the following lemma.
\begin{lemma}\label{lem:fin[0,C]exists}
 Let $F:\mathbb R_+ \times \Omega \to \mathbb R_+$ be an increasing predictable function such that $F(t)-F(s) \leq C(t-s)$ a.s.\ for all $0\leq s\leq t$ and for some fixed constant $C\geq 0$ and $F(0)=0$ a.s. Then there exists a predictable $f:\mathbb R_+\times \Omega \to \mathbb [0, C]$ such that $F(T) = \int_0^T f(s)\ud s$ for each fixed $T\geq 0$.
\end{lemma}

\begin{proof}
$F$ is a.s.\ differentiable in $t$ because $F$ is Lipschitz, so there exists $f:\mathbb R_+\times \Omega \to \mathbb [0, C]$ such that for a.e.\ $\omega \in \Omega$ and $t\geq 0$
\[
 f(t,\omega) = \lim_{\eps\to 0} \frac{F(t,\omega) - F((t-\eps) \vee 0,\omega)}{\eps}.
\]
Since $F$ is predictable, $t\mapsto F(t) - F((t-\eps) \vee 0)$ is a predictable process as well for each $\eps\geq 0$, so the obtained $f$ is predictable.
\end{proof}

\begin{proof}[Proof of Corollary \ref{cor:Ffromcondexpex}]
 Let $f:\mathbb R_+ \times \Omega \to [0, C]$ be as defined in Lemma \ref{lem:fin[0,C]exists}. Then by Theorem \ref{thm:condexpwrtP_m}, $\mathbb E(f|\mathcal P_m)$ exists and converges to $f$ a.s.\ on $\mathbb R_+ \times \Omega$. Moreover, $f$ is bounded by $C$, so $\mathbb E(f|\mathcal P_m)$ is bounded by $C$ as well. Therefore for each $m\geq 1$ we find using \eqref{eq:condexppred}
 \begin{align*}
   \sum_{n=0}^{[2^m T]-1}\mathbb E\Bigl[F\Bigl(\frac{n+1}{2^m}\Bigr) - F\Bigl(\frac{n}{2^m}\Bigr)\Big|\mathcal F_{\frac {n}{2^m}}\Bigr]
  &=  \sum_{n=0}^{[2^m T]-1}\mathbb E\left[\left.\int_{(\frac{n}{2^m},\frac{n+1}{2^m}]}f(s)\ud s\right|\mathcal F_{\frac {n}{2^m}}\right]\\
  &= \mathbb \sum_{n=0}^{[2^m T]-1}\int_{(\frac{n}{2^m},\frac{n+1}{2^m}]}E\bigl(f(s)\big|\mathcal F_{\frac {n}{2^m}}\bigr)\ud s\\
  &= \int_{\bigl(0,\frac{[2^m T]}{2^m}\bigr]}\mathbb E(f|\mathcal P_m)(s)\ud s,
 \end{align*}
and since $\frac{[2^m T]}{2^m} \to T$ as $m\to \infty$, 
\begin{align*}
 \lim_{m\to \infty}\sum_{n=0}^{[2^m T]-1}\mathbb E\Bigl[F\Bigl(\frac{n+1}{2^m}\Bigr) - F\Bigl(\frac{n}{2^m}\Bigr)\Big|\mathcal F_{\frac {n}{2^m}}\Bigr] &= \lim_{m\to \infty}\int_{\bigl(0,\frac{[2^m T]}{2^m}\bigr]}\mathbb E(f|\mathcal P_m)(s)\ud s\\
 & = \int_{(0,T]} f(s)\ud s = F(T),
\end{align*}
where the limit holds a.s., and since $F(T)\leq CT$ and all the functions above are bounded by $CT$ as well, by the dominated convergence theorem the limit holds in $L^p(\Omega)$ for each $1<p<\infty$.
\end{proof}

In the proof of Theorem~\ref{thm:mainintranmeas} we will use a time-change argument. We recall some necessary definitions and results. A nondecreasing, right-continuous family of stopping times $\tau =(\tau_s)_{s\geq 0}$ is called a  \textit{random time-change}. If $\mathbb F$ is right-continuous, then according to \cite[Lemma~7.3]{Kal} the same holds true
for the  \textit{induced filtration} $\mathbb G = (\mathcal G_s)_{s \geq 0} = (\mathcal F_{\tau_s})_{s\geq 0}$. 

For a random time-change $\tau =(\tau_s)_{s\geq 0}$ and for a random measure $\mu$ we define $\mu\circ \tau$ in the following way:
\[
 \mu\circ\tau ((s,t]\times B) = \mu((\tau_s,\tau_t]\times A),\;\;\; t\geq s\geq 0, A\in\mathcal J.
 \]
$\mu$ is said to be \textit{$\tau$-continuous} if $\mu((\tau_{s-}, \tau_s]\times J) = 0$ a.s.\ for each $s \geq 0$, where we let $\tau_{s-} := \lim_{\eps\to 0} \tau_{s-\eps}$, $\tau_{0-} := \tau_0$. Later we will need the following proposition.
 
\begin{proposition}\label{prop:apptimechange}
 Let $A:\mathbb R_+ \times \Omega\to \mathbb R_+$ be a strictly increasing continuous predictable process such that $A_0 = 0$ and $A_{t} \to \infty$ as $t\to \infty$ a.s. Then
 \[
  \tau_s = \{t: A_t=s\},\;\;\; s\geq 0.
 \]
defines a random time-change $\tau = (\tau_s)_{s\geq 0}$. It satisfies $(A\circ \tau)(t) = (\tau \circ A)(t) = t$ a.s.\ for each $t\geq 0$. Let $\mathbb G = (\mathcal G_s)_{s \geq 0} = (\mathcal F_{\tau_s})_{s\geq 0}$ be the induced filtration. Then $(A_t)_{t\geq 0}$ is a random time-change with respect to $\mathbb G$. Moreover, for any random measure $\mu$ the following hold:
\begin{itemize}
 \item [(i)] if $\mu$ is $\mathbb F$-optional, then $\mu\circ\tau$ is $\mathbb G$-optional,
 \item[(ii)] if $\mu$ is $\mathbb F$-predictable, then $\mu\circ\tau$ is $\mathbb G$-predictable,
 \item [(iii)] if $\mu$ is an $\mathbb F$-optional random measure with a compensator $\nu$, then $\nu\circ\tau$ is a compensator of $\mu\circ\tau$, and for each $\widetilde{\mathcal P}$-measurable simple $F:\mathbb R_+ \times \Omega\times J \to \mathbb R$ such that $\mathbb E (F\star\mu)_{\infty}<\infty$ we have $\mathbb E \bigl((F\circ \tau) \star (\mu\circ\tau)\bigr)_{\infty}<\infty$ and a.s.
 \begin{equation}\label{eq:apptimechangesimple}
  \begin{split}
    (F\star\mu)_{\infty} &= \bigl((F\circ \tau) \star (\mu\circ\tau)\bigr)_{\infty},\\
  (F\star\nu)_{\infty} &= \bigl((F\circ \tau) \star (\nu\circ\tau)\bigr)_{\infty},
  \end{split}
 \end{equation}
 \begin{equation}\label{eq:apptimechange}
  (F\star \bar{\mu})(\tau_s)= ((F\circ \tau)\star (\overline{\mu \circ\tau}))(s),\;\;\; s\geq 0.
 \end{equation}
\end{itemize}
\end{proposition}
\begin{proof}
 First of all notice that since $A$ is strictly increasing and continuous a.s., $s\mapsto \tau_s$ is an a.s.\ continuous function, so any random measure $\mu$ is $\tau$-continuous. Therefore (i) and (ii) follow from \cite[Theorem 10.27(c,d)]{Jac79}. Let us prove (iii). The fact that $\nu\circ\tau$ is a compensator of $\mu\circ\tau$ holds due to \cite[Theorem 10.27(e)]{Jac79}, while the rest follows from \cite[Theorem 10.28]{Jac79}, and in particular \eqref{eq:apptimechangesimple} follows from the~definition of $\mu\circ\tau$ and $\nu\circ\tau$.
\end{proof}
For more information on time-changes for random measures we refer to \cite[Chapter X]{Jac79}.

Let $(S, \Sigma,\rho)$ be a measure space. For $1<p,q<\infty$ we define $\hat{\mathcal S}_q^p$, $\hat{\mathcal D}_{q,q}^p$ and $\hat{\mathcal D}_{p,q}^p$ as the Banach spaces of all functions $F: \mathbb R_+ \times\Omega \times J \to L^q(S)$ that are $\widetilde{\mathcal P}$-measurable and for which the corresponding norms are finite:
\begin{equation}\label{eq:spqhatdpqqhatdppqhat}
\begin{split}
\|F\|_{\hat{\mathcal S}_q^p} &:= \Bigl( \mathbb E \Bigl\| \Bigl(\int_{\mathbb R_+\times J}|F|^2 \ud \nu \Bigr)^{\frac 12}\Bigr\|^p_{L^q(S)} \Bigr)^{\frac 1p},\\
\|F\|_{\hat{\mathcal D}_{q,q}^p} &:=  \Bigl( \mathbb E \Bigl( \int_{\mathbb R_+\times J}\|F\|^q_{L^q(S)} \ud \nu \Bigr)^{\frac pq} \Bigr)^{\frac 1p},\\
\|F\|_{\hat{\mathcal D}_{p,q}^p} &:=  \Bigl( \mathbb E  \int_{\mathbb R_+\times J}\|F\|^p_{L^q(S)} \ud \nu \Bigr)^{\frac 1p}.
\end{split}
\end{equation}
We show in Appendix~\ref{sec:appDuals} that 
$$(\hat{\mathcal S}_q^p)^* = \hat{\mathcal S}_{q'}^{p'}, \qquad (\hat{\mathcal D}_{q,q}^p)^* = \hat{\mathcal D}_{q',q'}^{p'}, 
\qquad (\hat{\mathcal D}_{p,q}^p)^* = \hat{\mathcal D}_{p',q'}^{p'}$$
hold isomorphically with constants depending only on $p$ and $q$.
\begin{theorem}\label{thm:mainintranmeas}
Fix $1<p,q<\infty$. Let $\mu$ be an optional $\widetilde{\mathcal P}$-$\si$-finite random measure on $\mathbb R_+\times \mathcal J$ and suppose that its compensator $\nu$ is non-atomic in time.
 Then for any 
 simple $\widetilde{\mathcal P}$-measurable $F:\mathbb R_+ \times \Omega \times J \to L^q(S)$ and for any $A \in \widetilde{\mathcal P}$ with $\mathbb E \mathbf 1_{A}\star \mu < \infty$
 \begin{equation}\label{eq:mainstochint}
  \Bigl( \mathbb E \sup_{0\leq s\leq t}\| ((F\mathbf 1_{A}) \star \bar{\mu})_s\|^p_{L^q(S)} \Bigr)^{\frac 1p} \eqsim_{p,q}\|F \mathbf 1_{A}\mathbf 1_{[0,t]}\|_{\mathcal I_{p,q}},
 \end{equation}
where $\mathcal I_{p,q}$ is given by
\begin{equation}\label{eq:ipq}
 \begin{split}
   \hat{\mathcal S}^p_q \cap \hat{\mathcal D}^p_{q,q} \cap \hat{\mathcal D}^p_{p,q} & \text{ if } \; 2\leq q\leq p<\infty,\\
 \hat{\mathcal S}^p_q \cap (\hat{\mathcal D}^p_{q,q} + \hat{\mathcal D}^p_{p,q}) & \text{ if } \; 2\leq p\leq q<\infty,\\
 (\hat{\mathcal S}^p_q \cap \hat{\mathcal D}^p_{q,q}) + \hat{\mathcal D}^p_{p,q} & \text{ if } \; 1<p<2\leq q<\infty,\\
 (\hat{\mathcal S}^p_q + \hat{\mathcal D}^p_{q,q}) \cap \hat{\mathcal D}^p_{p,q} & \text{ if } \; 1<q<2\leq p<\infty,\\
 \hat{\mathcal S}^p_q + (\hat{\mathcal D}^p_{q,q} \cap \hat{\mathcal D}^p_{p,q}) & \text{ if } \; 1<q\leq p\leq 2,\\
 \hat{\mathcal S}^p_q + \hat{\mathcal D}^p_{q,q} + \hat{\mathcal D}^p_{p,q} & \text{ if } \; 1<p\leq q \leq 2.
 \end{split}
\end{equation}
\end{theorem}

\begin{proof} By Lemma~\ref{lemma:restmu_1tomu_2} we can assume without loss of generality that $F:= F\mathbf 1_{A}$, $\mu:= \mu\mathbf 1_{A},$ and that there exists a $T\geq 0$ such that $F(t) =0$ for each $t\geq T$. Since $F$ is simple, it is uniformly bounded on $\mathbb R_+ \times \Omega \times J$ and, due to the fact that $\mathbb E \mathbf 1_{A} \star \mu = \mathbb E \mu(\mathbb R_+ \times \Omega) < \infty$, we find $\mathbb E \|F\star \mu\|< \infty$. Consequently $F\star \bar{\mu}$ exists and it is a local martingale. Therefore Doob's maximal inequality implies
\[
 (\mathbb E\|(F\star \bar{\mu})_{t}\|^p)^{\frac 1p} \eqsim_{p}  \Bigl( \mathbb E \sup_{0\leq s\leq t}\| (F\mathbf  \star \bar{\mu})_s\|^p_{L^q(S)} \Bigr)^{\frac 1p}
\]
and so it is enough to show that
\begin{equation}\label{eq:mainstochint1}
(\mathbb E\|(F\star \bar{\mu})_{t}\|^p)^{\frac 1p}  \eqsim_{p,q}\|F \mathbf 1_{[0,t]}\|_{\mathcal I_{p,q}}.
\end{equation}
The proof consists of two steps. In the first step we assume that $\nu$ is absolutely continuous with respect to Lebesgue measure. In this case, we can derive the upper bounds in \eqref{eq:mainstochint1} from the Burkholder-Rosenthal inequalities, Corollary \ref{cor:condexpmunu} and Corollary~\ref{cor:Ffromcondexpex}. The lower bounds then follow by duality. In the second step we deduce the general result via a time-change argument based on Proposition~\ref{prop:apptimechange}.

{\it Step 1: $\nu((s, t]\times J) \leq (t-s)$ for each $t\geq s\geq 0$ a.s.} 
 We will consider the cases $2\leq q\leq p<\infty$ and $1< p\leq q\leq 2$, the proofs in the other cases are similar. 
 
 {\it Case $2\leq q\leq p<\infty$}: Fix $m\geq 1$. Let $F_n := F\mathbf 1_{\left(\frac{n}{2^m}, \frac{n+1}{2^m}\right]}$ for each $n\geq 0$. Then
\[
 (d_n)_{n\geq 0} := \left((F_n \star \bar{\mu})_{\infty}\right)_{n\geq 0}
\]
is an $L^q(S)$-valued martingale difference sequence with respect to a filtration \linebreak $\bigl(\mathcal F_{\frac{n+1}{2^m}}\bigr)_{n\geq 0}$. Theorem \ref{thm:summaryBRIntro} implies
\begin{align*}
 \mathbb E \|(F \star \bar{\mu})_{\infty}\|_{L^q(S)}^p &=  \mathbb E \Bigl\|\sum_{n\geq 0} (F_n \star \bar{\mu})_{\infty}\Bigr\|_{L^q(S)}^p = \mathbb E \Bigl\|\sum_{n\geq 0} d_n\Bigr\|_{L^q(S)}^p\eqsim_{p, q} \|(d_n)\|_{s_{p,q}}^p\\
 &\eqsim_p (\|(d_n)\|_{S^p_q} + \|(d_n)\|_{D^p_{q,q}} + \|(d_n)\|_{D^p_{p,q}})^p.
\end{align*}
To bound $\|(d_n)\|_{S^p_q}$, observe that
\begin{align}\label{eq:dijspq}
    \|(d_{n})\|_{S^p_q} &= \Bigl( \mathbb E \Bigl\| \sum_{n} \mathbb E_{\mathcal F_{\frac{n}{2^m}}}|d_{n}|^2 \Bigr\|^p_{L^q(S)} \Bigr)^{\frac 1p} = \Bigl( \mathbb E \Bigl\| \sum_{n} \mathbb E_{\mathcal F_{\frac{n}{2^m}}}|(F_n \star \bar{\mu})_{\infty}|^2 \Bigr\|^p_{L^q(S)} \Bigr)^{\frac 1p}\nonumber\\
  &\stackrel{(*)}\eqsim_p \Bigl( \mathbb E \Bigl\| \sum_{n} \mathbb E_{\mathcal F_{\frac{n}{2^m}}}(|F_n|^2 \star \nu)_{\infty}\Bigr\|^p_{L^q(S)} \Bigr)^{\frac 1p}\\
  &=\Bigl( \mathbb E \Bigl\| \sum_{n} \mathbb E_{\mathcal F_{\frac{n}{2^m}}}\bigl((|F|^2 \star \nu)_{\frac{n+1}{2^m}} - (|F|^2 \star \nu)_{\frac{n}{2^m}}\bigr)\Bigr\|^p_{L^q(S)} \Bigr)^{\frac 1p},\nonumber
\end{align}
 where $(*)$ holds by Corollary \ref{cor:condexpmunu} and the fact that 
 $$
 \nu\Bigl(\Bigl(\frac{n}{2^m}, \frac{n+1}{2^m}\Bigr]\Bigr) \leq \frac{n+1}{2^m}-\frac{n}{2^m} = \frac{1}{2^m}\leq 1.
 $$ 
 Notice that for a.e.\ $\omega\in\Omega$, all $s\in S$, and each $t\geq u\geq 0$
 \begin{align*}
  (|F|^2 \star \nu)_{t}(s,\omega) - (|F|^2 \star \nu)_{u}(s,\omega)&\leq \sup|F(s)|^2(\nu((u, t]\times J)(\omega))\\
  &\leq \sup|F(s)|^2(t-u),
 \end{align*}
so by Corollary \ref{cor:Ffromcondexpex}
\[
\sum_{n} \mathbb E_{\mathcal F_{\frac{n}{2^m}}}\bigl((|F|^2 \star \nu)_{\frac{n+1}{2^m}} - (|F|^2 \star \nu)_{\frac{n}{2^m}}\bigr) \to (|F|^2\star \nu)_T = (|F|^2\star \nu)_{\infty}
\]
a.s.\ as $m\to \infty$. Therefore thanks to \eqref{eq:dijspq}
\begin{equation}\label{eq:SpqfromNov}
 \begin{split}
   \|(d_{n})\|_{S^p_q}&\eqsim\Bigl( \mathbb E \Bigl\| \sum_{n} \mathbb E_{\mathcal F_{\frac{n}{2^m}}}\bigl((|F|^2 \star \nu)_{\frac{n+1}{2^m}} - (|F|^2 \star \nu)_{\frac{n}{2^m}}\bigr)\Bigr\|^p_{L^q(S)} \Bigr)^{\frac 1p}\\
  &\quad\stackrel{m\to \infty}\longrightarrow (\mathbb E \|(|F|^2\star \nu)_{\infty}\|_{L^q(S)}^p)^{\frac 1p} = \|F\|_{\mathcal S^p_q}.
 \end{split}
\end{equation}
Now let us estimate $\|(d_n)\|_{D_{q,q}^p}$. Analogously to \eqref{eq:dijspq}
\begin{align}\label{eq:dijdpqq}
  \|(d_{n})\|_{D_{q,q}^p} &= \Bigl( \mathbb E \Bigl(\sum_{n}  \mathbb E_{\mathcal F_{\frac{n}{2^m}}}\|d_{n} \|^q_{L^q(S)} \Bigr)^{\frac pq}\Bigr)^{\frac 1p} = \Bigl( \mathbb E  \Bigl(\sum_{n} \mathbb E_{\mathcal F_{\frac{n}{2^m}}}\|( F_n \star \bar{\mu} )_{\infty}\|^q_{L^q(S)} \Bigr)^{\frac pq}\Bigr)^{\frac 1p}\nonumber\\
  &\eqsim \Bigl( \mathbb E  \Bigl(\sum_{n} \mathbb E_{\mathcal F_{\frac{n}{2^m}}}(\| F_n\|^q_{L^q(S)} \star \nu)_{\infty}  \Bigr)^{\frac pq}\Bigr)^{\frac 1p}\\
  &= \Bigl( \mathbb E  \Bigl(\sum_{n} \mathbb E_{\mathcal F_{\frac{n}{2^m}}}\bigl((\| F\|^q_{L^q(S)} \star \nu)_{\frac{n+1}{2^m}} -  (\| F\|^q_{L^q(S)} \star \nu)_{\frac{n}{2^m}}\bigr) \Bigr)^{\frac pq}\Bigr)^{\frac 1p},\nonumber
\end{align}
and similarly to \eqref{eq:SpqfromNov} the last expression converges to $\|F\|_{\mathcal D_{q,q}^p}$. The same can be shown for $\mathcal D_{p,q}^p$.
 
{\it Case $1<p \leq q\leq 2$}: Let $\cI_{\mathrm{elem}}(\widetilde {\mathcal P})$ denote the linear space of all simple $\widetilde{\mathcal P}$-measurable $L^q(S)$-valued functions. This linear space is dense in $\hat{\cS}_{q}^p$, $\hat{\cD}_{p,q}^p$ and $\hat{\cD}_{q,q}^p$. Let $F \in \cI_{\mathrm{elem}}(\widetilde {\mathcal P})$. Fix a decomposition $F=F_1+F_2+F_3$ with $F_{\alpha} \in \cI_{\mathrm{elem}}(\widetilde {\mathcal P})$.

Fix $m\geq 1$ and set $F_{n,\alpha} = F_{\alpha}\mathbf 1_{\left(\frac{n}{2^m}, \frac{n+1}{2^m}\right]}$, $d_{n,\alpha} = F_{n,\alpha}\star\bar{\mu}$, $\alpha = 1,2,3$, so that
$$
(F\star \bar{\mu})_{T} = (F\star \bar{\mu})_{\infty} = \sum_{n} d_{n,1} + d_{n,2} + d_{n,3}.
$$
Then by Theorem~\ref{thm:summaryBRIntro}, \eqref{eq:dijspq}, \eqref{eq:SpqfromNov} and \eqref{eq:dijdpqq} we conclude that
$$
\Big(\E\|(F\star \bar{\mu})_{\infty}\|_{L^q(S)}^p\Big)^{\frac{1}{p}} \lesssim_{p,q} \|F_1\|_{\cS_{q}^p}  + \|F_2\|_{\cD_{p,q}^p} + \|F_3\|_{\cD_{q,q}^p}.
$$
Since $\cI_{\mathrm{elem}}(\widetilde {\mathcal P})$ is dense in $\hat{\cS}_{q}^p$, $\hat{\cD}_{p,q}^p$ and $\hat{\cD}_{q,q}^p$, we conclude by taking the infimum over $F_1,F_2,F_3$ as above that
$$
\Big(\E\|(F\star \bar{\mu})_{\infty}\|_{L^q(S)}^p\Big)^{\frac{1}{p}} \lesssim_{p,q} \|F\|_{\cI_{p,q}}.
$$

{\it The duality argument:} Fix $t< \infty$, $1<p,q<\infty$. Using the upper bounds in \eqref{eq:mainstochint1} we can obtain the stochastic integral $(F\star \bar{\mu})_{t}$ as an $L^p$-limit of the integrals of the corresponding simple approximations of $F$ in $\mathcal I_{p,q}$. Let $Y$ be the closure of the linear subspace $\cup_{F\in \mathcal I_{p,q}}(F\star \bar{\mu})_t$ in $L^p(\Omega;L^q(S))$ and let $X = \mathcal I_{p,q}$. By Corollary \ref{cor:dualIpq}, $X^* = \mathcal I_{p',q'}$. Let $U$ (resp.\ $V$) be the dense subspace of $X$ (resp.\ $X^*$) consisting of all $\widetilde {\mathcal P}$-measurable simple $L^q(S)$-valued (resp.\ $L^{q'}(S)$-valued) functions. Define both $j_0:U\to Y$ and $k_0: V \to Y^*$ by $F \mapsto (F\star \bar{\mu})_{t}$. Note that $k_0$ maps into $Y^*$ since each $(F\star \bar{\mu})_{t}$ is in $L^{p'}(\Omega; L^{q'}(S))$, so it defines a bounded linear functional on $Y$. By the upper bounds in \eqref{eq:mainstochint1}, $j_0$ and $k_0$ are bounded. Moreover, by the definition of $Y$, $\text{ran }j_0$ is dense in $Y$. Finally, by Corollary \ref{lemma:equalityfordual} $\langle F^*, F\rangle = \langle k_0(F^*), j_0(F)\rangle$ for all $F\in U$ and $F^*\in V$. Now \eqref{eq:mainstochint} follows from Lemma \ref{lemma:jkstaff}. 

{\it Step 2: general case.} Recall that, due to our assumptions in the beginning of the proof, $\mathbb E \mu(\mathbb R_+\times \Omega) = \mathbb E \nu(\mathbb R_+ \times \Omega)<\infty$. Since $\nu$ is non-atomic in time, we can define a continuous strictly increasing predictable process $A:\mathbb R_+ \times \Omega \to \mathbb R_+$ by 
\[
 A_t = \nu([0,t]\times J) + t,\;\;\; t\geq 0.
\]
Let $\tau =(\tau_s)_{s\geq 0}$ be the time-change defined in Proposition \ref{prop:apptimechange}. Then according to Proposition \ref{prop:apptimechange} the random measure $\mu_{\tau}:= \mu\circ\tau$ is $\mathbb G$-optional, where $\mathbb G:= (\mathcal G_s)_{s\geq 0} = (\mathcal F_{\tau_s})_{s\geq 0}$. Moreover, $\nu_{\tau}:= \nu\circ\tau$ is $\mathbb G$-predictable and a compensator of $\mu_{\tau}$. Let $G:= F\circ\tau$. Notice that for each $t\geq s\geq 0$ a.s.
\begin{equation}\label{eq:nu_tauisforstep1}
 \begin{split}
    \nu_{\tau}((s,t]\times J) &= \nu((\tau_s, \tau_t]\times J) = \nu((0, \tau_t]\times J) - \nu((0, \tau_s]\times J)\\
  &\leq\nu((0, \tau_t]\times J) - \nu((0, \tau_s]\times J) + (\tau_t - \tau_s)\\
  &= (\nu((0, \tau_t]\times J) + \tau_t) - (\nu((0, \tau_s]\times J) + \tau_s) = t-s.
 \end{split}
\end{equation}
Let $\mathcal I_{p,q}^{\tau}$ be defined as $\mathcal I_{p,q}$ but for the random measure $\nu_{\tau}$. By \eqref{eq:nu_tauisforstep1} Step 1 yields $\mathbb E\|G\star \bar{\mu}_{\tau}\|^p \eqsim_{p,q} \|G\|_{\mathcal I_{p, q}^{\tau}}^p$. Indeed, by \eqref{eq:apptimechange}, $\mathbb E\|(G\star \bar{\mu}_{\tau})_{\infty}\|^p = \mathbb E\|F\star \bar{\mu}\|^p$. Moreover, for given $F_i$ and $G_i=F_i\circ \tau$, $i=1,2,3$, it follows from \eqref{eq:apptimechangesimple} that
\begin{equation*}
\begin{split}
 \mathbb E \Bigl\| \Bigl(\int_{\mathbb R_+\times J}|G_1|^2 \ud \nu_{\tau} \Bigr)^{\frac 12}\Bigr\|^p_{L^q(S)}&=  \mathbb E \Bigl\| \Bigl(\int_{\mathbb R_+\times J}|F_1|^2 \ud \nu \Bigr)^{\frac 12}\Bigr\|^p_{L^q(S)} = \|F_1\|_{\mathcal S_p^q}^p,\\
 \mathbb E \Bigl( \int_{\mathbb R_+\times J}\|G_2\|^q_{L^q(S)} \ud \nu_{\tau} \Bigr)^{\frac pq}&=  \mathbb E \Bigl( \int_{\mathbb R_+\times J}\|F_2\|^q_{L^q(S)} \ud \nu \Bigr)^{\frac pq} = \|F_2\|_{\mathcal D_{q,q}^p}^p,\\
 \mathbb E  \int_{\mathbb R_+\times J}\|G_3\|^p_{L^q(S)} \ud \nu_{\tau}&= \mathbb E  \int_{\mathbb R_+\times J}\|F_3\|^p_{L^q(S)} \ud \nu = \|F_3\|_{\mathcal D_{p,q}^p}^p.
\end{split}
\end{equation*}
Consequently, $\|G\|_{\mathcal I_{p, q}^{\tau}} = \|F\|_{\mathcal I_{p, q}}$. We conclude that $\mathbb E\|F\star \bar{\mu}\|^p \eqsim_{p,q} \|F\|_{\mathcal I_{p, q}}^p$.
\end{proof}

\begin{remark}
Let us compare our result to the literature. The upper bounds in Theorem~\ref{thm:mainintranmeas} were discovered in the scalar-valued case by A.A.\ Novikov in \cite[Theorem 1]{Nov75}. By exploiting an orthonormal basis one can easily extend this result to the Hilbert-space valued integrands, see \cite[Section 3.3]{MaR14} for details. The paper \cite{MaR14} contains several other proofs of the Hilbert-space valued version of Novikov's inequality. In the context of Poisson random measures, Theorem~\ref{thm:mainintranmeas} was obtained in \cite{Dir14}. Some one-sided extensions of the latter result in the context of general Banach spaces were obtained in \cite{DMN12}. However, these bounds, which are based on the martingale type and cotype of the space, are only matching in the Hilbert-space case and not optimal in general (in particular for $L^q$-spaces). A very different proof of the upper bounds in Theorem~\ref{thm:mainintranmeas}, which exploits tools from stochastic analysis, was discovered independently of our work by Marinelli in \cite{Mar13}.
\end{remark}
As a corollary, we obtain the following sharp bounds for stochastic integrals.
\begin{theorem}\label{thm:intwrtpurdismart}
Fix $1<p,q<\infty$. Let $H$ be a Hilbert space, $(S, \Sigma,\rho)$ be a measure space and let $M:\mathbb R_+ \times \Omega \to H$ be a purely discontinuous quasi-left continuous local martingale. Let $\Phi:\mathbb R_+ \times \Omega \to \mathcal L(H,L^q(S))$ be elementary predictable. Then 
 \begin{equation}\label{eq:mainstochintwrtpurdismart}
  \Bigl(\mathbb E \sup_{0\leq s\leq t} \|(\Phi\cdot M)_s\|^p_{L^q(S)} \Bigr)^{\frac 1p} \eqsim_{p,q}\|\Phi_H{ \mathbf{1}_{[0,t]}}\|_{\mathcal I_{p,q}},
 \end{equation}
where $\Phi_H:\mathbb R_+ \times \Omega \times H \to L^q(S)$ is defined by
\[
 \Phi_H(t,\omega, h) := \Phi(t,\omega)h,\;\;\; t\geq0, \omega \in \Omega, h\in H,
\]
and $\mathcal I_{p,q}$ is given as in \eqref{eq:ipq} for $\nu = \nu^M$.
\end{theorem}

\begin{proof}
The result follows from Doob's maximal inequality, Lemma~\ref{lemma:whyPhijisint}, Theorem~\ref{thm:mainintranmeas}, and the fact that $\|\Phi_H \mathbf 1_{A_n}\|_{\mathcal I_{p, q}}\nearrow \|\Phi_H \|_{\mathcal I_{p, q}}$ as $n\to \infty$ by the monotone convergence theorem.
\end{proof}

\subsection{Integration with respect to continuous martingales}
\label{sec:ItoC}

Finally, let us recall the known sharp bounds for $L^q$-valued stochastic integrals with respect to continuous local martingales. These bounds are a special case of the main result in \cite{VY}. To formulate these, we will need $\gamma$-radonifying operators. Let $(\gamma_n')_{n \geq 1}$ be a sequence of independent standard Gaussian random variables
on a probability space $(\Omega', {\mathcal F'}, \mathbb P')$ (we reserve the notation $(\Omega, {\mathcal F}, \mathbb P)$ for the probability space on which our processes live) and let $H$ be a separable Hilbert space. A bounded operator $R \in\mathcal L (H, X)$ is said to be {\it $\gamma$-radonifying} if for some (and then for each) orthonormal basis $(h_n)_{n\geq 1}$ of $H$ the Gaussian series $\sum_{n\geq 1} \gamma_n' Rh_n$ converges in $L^2 (\Omega'; X)$. We then define
$$
\|R\|_{\gamma(H,X)} :=\Bigl(\mathbb E'\Bigl\|\sum_{n\geq 1} \gamma_n' Rh_n \Bigr\|_X^2\Bigr)^{\frac 12}.
$$
This number does not depend on the sequence $(\gamma_n')_{n\geq 1}$ and the basis $(h_n)_{n\geq 1}$, and
defines a norm on the space $\gamma(H, X)$ of all $\gamma$-radonifying operators from $H$ into
$X$. Endowed with this norm, $\gamma(H, X)$ is a Banach space, which is separable if $X$
is separable. Moreover, if $X = L^q(S)$ for some separable measure space $(S, \Sigma, \rho)$, then thanks to the Trace Duality that is presented e.g. in \cite{HNVW2} we have that
\begin{equation}\label{eq:traceduality}
 (\gamma(H, X))^*\simeq \gamma(H^*, X^*).
\end{equation}
We refer to \cite{HNVW2} and the references therein for further details on $\gamma$-radonifying operators.

For $F:\mathbb R_+ \to \mathbb R_+$ nondecreasing, we define a measure $\rho_F$ on $\mathcal B(\mathbb R_+)$ by
\[
 \rho_F((s, t]) = F(t)-F(s),\;\;\; 0\leq s<t<\infty.
\]
If $X$ is a Banach space and $1\leq p\leq \infty$, then we write $L^p(\mathbb R_+,F;X)$ for the Banach space $L^p(\mathbb R_+,\rho_F;X)$.

Let $M:\mathbb R_+ \times \Omega \to H$ be a continuous local martingale. Then thanks to \cite[Chapter 14.3]{MP} one can define a continuous predictable process $[M]:\mathbb R_+ \times \Omega \to \mathbb R$ and a strongly progressively measurable $q_M:\mathbb R_+ \times \Omega \to \mathcal L(H)$ such that $[M]$ is a quadratic variation of $M$ and $\int_0^\cdot \langle q_M(s) h, h\rangle \ud[M]_s$ is a quadratic variation of $[Mh]$ for each $h \in H$. The following theorem immediately follows from \cite[Theorem 4.1]{VY} and formula \cite[(3.9)]{VY}.
\begin{theorem}\label{thm:itoisomcontlocmart}
  Let $H$ be a Hilbert space, $1<p,q<\infty$. Let $M:\mathbb R_+ \times \Omega \to H$ be a continuous local martingale, $\Phi :\mathbb R_+ \times \Omega \to \mathcal L(H,L^q(S))$ be elementary predictable. Then
  \[
  \mathbb E\bigl(\sup_{0\leq s\leq t} \|(\Phi \cdot M)_s\|^p\bigr) \eqsim_{p,q}\mathbb E\|\Phi q_{M}^{\frac12}\mathbf{1}_{[0,t]}\|^p_{\gamma(L^2(\mathbb R_+,[M^{c}];H),L^q(S))}.
  \]
\end{theorem}

\subsection{Integration with respect to general local martingales}
\label{sec:mainResult}

We can now combine the sharp estimates obtained for the three special type of stochastic integrals to obtain sharp estimates for $\Phi\cdot M$, where $M$ is an arbitrary local martingale.

\begin{lemma}\label{lemma:intcontcontintdisdis}
 Let $H$ be a Hilbert space, $X$ be a finite-dimensional space, $M:\mathbb R_+ \times \Omega \to H$ be a local martingale, $\Phi:\mathbb R_+ \times \Omega \to \mathcal L(H,X)$ be elementary predictable, $F:\mathbb R_+ \times \Omega \times H \to X$ be elementary $\widetilde{\mathcal P}$-measurable. Then
 \begin{itemize}
  \item[(i)] if $M$ is continuous, then $\Phi \cdot M$ is continuous,
  \item[(ii)] if $M$ is purely discontinuous quasi-left continuous, then $F \star \bar{\mu}^M$ is purely discontinuous quasi-left continuous,
  \item[(iii)] if $M$ is purely discontinuous with accessible jumps, then $\Phi \cdot M$ is purely discontinuous with accessible jumps.
 \end{itemize}
\end{lemma}
\begin{proof}
 (i) holds since if $M$ is continuous, then the formula \eqref{eq:intelpred} defines an a.s.\ continuous process. 
 
 To prove pure discontinuity in (ii) one has to endow $X$ with a Euclidean norm and notice that if $M$ is purely discontinuous quasi-left continuous then by \cite[Proposition II.1.28]{JS} $[F \star \bar{\mu}^M]_t = \sum_{0\leq s\leq t} \|F(\Delta M)\|^2$ a.s.\ for all $t\geq 0$ since $F \star \nu^M$ is absolutely continuous, so it does not effect on the quadratic variation. Therefore $[F \star \bar{\mu}^M]$ is purely discontinuous, and so $F\star \bar{\mu}^M$ is purely discontinuous by \cite[Theorem 26.14]{Kal}. Quasi-left continuity then follows as $\Delta(F\star \bar{\mu}^M)_{\tau} = F(\Delta M_{\tau}) = 0$ a.s.\ for any predictable stopping time $\tau$.
 
 Pure discontinuity of $\Phi \cdot M$ in (iii) follows from the same argument as in (ii), and the rest can be proven using the fact that a.s.
 \[
 \{t\in \mathbb R_+ : \Delta(\Phi \cdot M)_t \neq 0\} \subset \{t\in \mathbb R_+ : \Delta M_t \neq 0\}.
 \]
\end{proof}

The following observation is fundamental for the duality argument used to prove the lower bounds in Theorem~\ref{thm:preintgenmart}.
\begin{lemma}\label{lemma:contanddiscpart}
Let $H$ be a Hilbert space, $X$ be a Banach space, $M^{c}, M^{q}:\mathbb R_+ \times \Omega \to H$ be continuous and purely discontinuous quasi-left continuous martingales, $M^{a,1}:\mathbb R_+ \times \Omega \to X$, $M^{a,2}:\mathbb R_+ \times \Omega \to X^*$ be purely discontinuous martingales with accessible jumps, $\Phi_1:\mathbb R_+ \times \Omega \to \mathcal L(H, X)$, $\Phi_2:\mathbb R_+ \times \Omega \to \mathcal L(H, X^*)$ be elementary predictable, $F_1:\mathbb R_+ \times \Omega\times H\to X$, $F_2:\mathbb R_+ \times \Omega\times H\to X^*$ be elementary $\widetilde{\mathcal P}$-measurable. Assume that $(\Phi_1 \cdot M^{c})_{\infty}$, $(F_1 \star \bar{\mu}^{M^{q}})_{\infty}$, $M^{a,1}_{\infty} \in L^p(\Omega; X)$ and $(\Phi_2 \cdot M^{c})_{\infty}$, $(F_2 \star \bar{\mu}^{M^{q}})_{\infty}$, $M^{a,2}_{\infty} \in L^{p'}(\Omega; X^*)$ for some $1<p<\infty$. Then, for all $t\geq 0$,
\begin{multline}\label{eq:contdiscint}
\mathbb E\langle(\Phi_1 \cdot M^{c} + F_1 \star \bar{\mu}^{M^{q}} + M^{a,1} )_t,(\Phi_2 \cdot M^{c} + F_2 \star \bar{\mu}^{M^{q}}+  M^{a,2} )_t\rangle \\  =\mathbb E\langle(\Phi_1 \cdot M^{c})_t,(\Phi_2 \cdot M^{c})_t \rangle + \mathbb E\langle(F_1 \star \bar{\mu}^{M^{q}})_t, (F_2 \star \bar{\mu}^{M^{q}})_t\rangle +\mathbb E\langle M^{a,1}_t, M^{a,2}_t \rangle.
\end{multline}
\end{lemma}

\begin{lemma}\label{lem:expectofquasleftcontaccjump0}
 Let $X$ be a Banach space, $X_0 \subset X$ be a finite-dimensional subspace, $1<p<\infty$, $M^q:\mathbb R_+ \times \Omega \to X_0$ be a purely discontinuous quasi-left continuous $L^p$-martingale, $M^q_0=0$, $M^a:\mathbb R_+ \times \Omega \to X^*$ be a purely discontinuous $L^{p'}$-martingale with accessible jumps. Then $\mathbb E \langle M^q_t, M^a_t \rangle = 0$ for each~$t\geq 0$.
\end{lemma}

\begin{proof}
 Let $d$ be the dimension of $X_0$, $x_1, \ldots, x_d$ be a basis of $X_0$. Then there exist purely discontinuous quasi-left continuous $L^p$-martingales $M^{q,1}, \ldots, M^{q, d} :\mathbb R_+ \times \Omega \to \mathbb R$ such that $M^q = M^{q,1}x_1 + \cdots + M^{q, d}x_d$. Thus for any $i=1,\ldots, d$ and any purely discontinuous $L^{p'}$-martingale $N:\mathbb R_+ \times \Omega \to \mathbb R$ with accessible jumps $[M^{q, i}, N] = 0$ a.s.\ by \cite[Corollary 26.16]{Kal}. Hence \cite[Proposition~I.4.50(a)]{JS} implies that $M^{q, i} N$ is a local martingale, and due to integrability it is a martingale. Notice also that all $M^{q,i}$ start at zero, therefore
 \[
  \mathbb E \langle M^q_t, M^a_t \rangle = \sum_{i=1}^d\mathbb E M^{q, i}_t\langle x_i, M^a_t\rangle = \sum_{i=1}^d\mathbb E M^{q, i}_0\langle x_i, M^a_0\rangle = 0.
 \]
\end{proof}

\begin{proof}[Proof of Lemma \ref{lemma:contanddiscpart}]
 Since all the integrands $\Phi_1$, $\Phi_2$, $F_1$, $F_2$ are elementary, one can suppose that $X$ and $X^*$ are finite dimensional, so we can endow these spaces with Euclidean norms. Since by Lemma \ref{lemma:intcontcontintdisdis} $\Phi_1 \cdot M^{c}$ and $\Phi_2 \cdot M^{c}$ are continuous, $F_1 \star \bar{\mu}^{M^{q}}$, $F_1 \star \bar{\mu}^{M^{q}}$, $M^{a,1}$ and $M^{a,2}$ are purely discontinuous, then \cite[Definition I.4.11]{JS} implies that for each $t\geq 0$
 \begin{align*}
  \mathbb E\langle(\Phi_1 \cdot M^{c} )_t,( F_2 \star \bar{\mu}^{M^{q}})_t\rangle 
  &=\mathbb E[\Phi_1 \cdot M^{c}, F_2 \star \bar{\mu}^{M^{q}}]_t = 0,\\
   \mathbb E\langle(\Phi_2 \cdot M^{c} )_t,( F_1 \star \bar{\mu}^{M^{q}})_t\rangle 
  &=\mathbb E[\Phi_2 \cdot M^{c}, F_1 \star \bar{\mu}^{M^{q}}]_t = 0,\\
  \mathbb E\langle(\Phi_1 \cdot M^{c} )_t,M^{a,2}_t\rangle 
  &=\mathbb E[\Phi_1 \cdot M^{c},  M^{a,2}]_t = 0,\\
   \mathbb E\langle(\Phi_2 \cdot M^{c} )_t,M^{a,1}_t\rangle 
  &=\mathbb E[\Phi_2 \cdot M^{c}, M^{a,1}]_t = 0.
 \end{align*}
 Moreover, thanks to Lemma \ref{lemma:intcontcontintdisdis} and Lemma \ref{lem:expectofquasleftcontaccjump0}
 \begin{align*}
  \mathbb E\langle  M^{a,1} _t,( F_2 \star \bar{\mu}^{M^{q}})_t\rangle 
  =\mathbb E\langle  M^{a,2}_t,( F_1 \star \bar{\mu}^{M^{q}})_t\rangle =0,
 \end{align*}
so \eqref{eq:contdiscint} easily follows.
\end{proof}

\begin{theorem}\label{thm:preintgenmart}
 Let $H$ be a Hilbert space, $1<p,q<\infty$. Let $M^{c}$, $M^{q}:\mathbb R_+ \times \Omega \to H$ be continuous and purely discontinuous quasi-left continuous local martingales, $M^a:\mathbb R_+ \times \Omega \to L^q(S)$ be a purely discontinuous $L^p$-martingale with accessible jumps, $\Phi :\mathbb R_+ \times \Omega \to \mathcal L(H,L^q(S))$ be elementary predictable, $F: \mathbb R_+ \times \Omega\times H \to L^q(S)$ be elementary $\widetilde{\mathcal P}$-measurable. If $\Phi\cdot M^{c}$ and $F\star \bar{\mu}^{M^{q} }$ are $L^p$-martingales, then
 \begin{align}
 \label{eqn:preintgenmart}
 & \Bigl(\mathbb E \sup_{0\leq s\leq t} \bigl\|\bigl( \Phi\cdot M^{c} + F\star \bar{\mu}^{M^{q} }+ M^{a}\bigr)_{\infty} \bigr\|^p_{L^q(S)} \Bigr)^{\frac 1p}\\ 
 & \qquad \eqsim_{p,q}\bigl(\mathbb E\|\Phi q_{M^{c}}^{\frac12}\|^p_{\gamma(L^2(\mathbb R_+,[M^{c}];H),X)}\bigr)^{\frac 1p}+\|F\|_{\mathcal I_{p,q}} + \|  M^{a}\|_{\mathcal A_{p,q}},\nonumber
 \end{align}
 where $\mathcal I_{p,q}$ is given as in \eqref{eq:ipq} for $\nu = \nu^{M^{\rm q}}$, $\mathcal A_{p,q}$ is given as in \eqref{eq:Apq}.
\end{theorem}

\begin{proof} 
The estimate $\lesssim_{p,q}$ follows from the triangle inequality and Theorems~\ref{thm:itoisomcontlocmart}, \ref{thm:mainintranmeas} and \ref{thm:acessjumpsL^q}. Let us now prove $\gtrsim_{p,q}$ via duality. Without loss of generality due to the proof of Theorem \ref{thm:acessjumpsL^q} and due to Lemma~\ref{lem:dualApq^T} we can assume that there exists $N\geq 1$ and a sequence of predictable stopping times $\mathcal T =(\tau_n)_{n=0}^N$ such that $M$ has a.s.\ at most $N$ jumps and a.s.\ $\{t\in \mathbb R_+: \Delta M_t \neq 0\}\subset \{\tau_0,\ldots,\tau_N\}$. Define the Banach space 
 \[
  X:= L^{p}(\Omega;\gamma(L^2(\mathbb R_+,[M^{c}];H),L^q(S))) \times \mathcal I_{p,q} \times \mathcal A_{p, q}^{\mathcal T}
 \]
and let $Y$ be the closure of the lineare subspace $\cup_{(\Phi, F, M^a)\in X}(\Phi \cdot M^{c}+F\star \bar{\mu}^{M^{\rm d}} +M^a)_{\infty}$ in $L^p(\Omega;L^q(S))$. Then by \cite[Proposition 1.3.3]{HNVW1}, the Trace duality \eqref{eq:traceduality}, Corollary~\ref{cor:dualIpq} and the duality statement in Theorem~\ref{thm:acessjumpsinTL^q}
\[
X^* = L^{p'}(\Omega;\gamma(L^2(\mathbb R_+,[M^{c}];H),L^{q'}(S))) \times\mathcal I_{p',q'} \times \mathcal A_{p',q'}^{\mathcal T}.
\]
By the upper bounds in \eqref{eqn:preintgenmart}, the maps $j:X\to Y$ and $k:X^* \to Y^*$ defined via $(\Phi, F, M^a) \mapsto (\Phi \cdot M^{c}+F\star \bar{\mu}^{M^{\rm d}}  + M^a)_{\infty}$ are both continuous linear mappings. Let $x = (\Phi_1, F_1, M^a_1) \in X$, $x^*= (\Phi_2, F_2, M^a_2) \in X^*$ be such that $\Phi_1$ and $\Phi_2$ are elementary predictable,  and $F_1$ and $F_2$ are elementary $\widetilde{\mathcal P}$-measurable. Then $\langle \tilde x^*, \tilde x\rangle = \langle k(\tilde x^*), j(\tilde x)\rangle$ by Lemma~\ref{lemma:contanddiscpart} and \eqref{eq:dualApq^T} and so Lemma \ref{lemma:jkstaff} yields $\gtrsim_{p,q}$ in \eqref{eqn:preintgenmart}.
\end{proof}

\begin{theorem}
\label{thm:mainSI}
 Let $H$ be a Hilbert space, $1<p,q<\infty$. Let $M:\mathbb R_+ \times \Omega \to H$ be a local martingale, $M^{c}, M^{q}, M^d:\mathbb R_+ \times \Omega \to H$ be local martingales such that $M^c_0 = M^q_0=0$, $M^c$ is continuous, $M^q$ is purely discontinuous quasi-left continuous, $M^a$ is purely discontinuous with accessible jumps, $M = M^{c} + M^{q} + M^a$. Let $\Phi :\mathbb R_+ \times \Omega \to \mathcal L(H,L^q(S))$ be elementary predictable. Then,
 \begin{multline}\label{eq:mainSI}
  \Bigl(\mathbb E \sup_{0\leq s\leq t} \| (\Phi\cdot M)_s \|^p_{L^q(S)} \Bigr)^{\frac 1p}\\ 
  \eqsim_{p,q}\Bigl(\mathbb E\|\Phi q_{M^{c}}^{\frac12}\mathbf 1_{[0,t]}\|^p_{\gamma(L^2(\mathbb R_+,[M^{c}];H),X)}\Bigr)^{\frac 1p}\\
  +\|\Phi_H \mathbf 1_{[0,t]}\|_{\mathcal I_{p,q}}  + \|(\Phi\mathbf 1_{[0,t]}) \cdot M^a\|_{\mathcal A_{p,q}},
 \end{multline}
where $\Phi_H:\mathbb R_+ \times \Omega \times H \to L^q(S)$ is defined by
\[
 \Phi_H(t,\omega, h) := \Phi(t,\omega)h,\;\;\; t\geq0, \omega \in \Omega, h\in H,
\]
$\mathcal I_{p,q}$ is given as in \eqref{eq:ipq} for $\nu = \nu^{M^{q}}$, and $\mathcal A_{p, q}$ is as defined in \eqref{eq:Apq}. 
\end{theorem}

To prove this theorem we will need a lemma.
\begin{lemma}\label{lem:MeyYoedecfindim}
 Let $X$ be a finite-dimensional Banach space, $1<p<\infty$, $M:\mathbb R_+ \times \Omega \to X$ be a local martingale. Then there exist local martingales $M^c, M^q, M^a:\mathbb R_+ \times \Omega \to X$ such that $M^c_0 = M^q_0=0$, $M^c$ is continuous, $M^q$ is purely discontinuous quasi-left continuous, $M^a$ is purely discontinuous with accessible jumps, $M = M^c+M^a + M^q$ and then for each $t\geq 0$
 \[
  \mathbb E\|M_t\|^p \eqsim_{p, X} \mathbb E \|M^c_t\|^p + \mathbb E \|M^q_t\|^p + \mathbb E \|M^a_t\|^p.
 \]
 In other words, $M$ is an $L^p$-martingale  if and only if each of $M^c$, $M^q$ and $M^a$ is an $L^p$-martingale.
\end{lemma}
\begin{proof}
 The proof for a general UMD Banach space can be found in \cite{Y17MartDec}. Here we present a much simpler proof that works only in finite dimension. Since $X$ is finite-dimensional, we can endow it with a Euclidean norm $\tnorm{\cdot}$. Then the existence of the decomposition $M = M^c+M^a + M^q$ follows from Lemma \ref{lem:YoeMayH}, and for each $t\geq 0$ due to the Burkholder-Davis-Gundy inequality and Lemma \ref{lem:YoeMayH}
\begin{align*}
 \mathbb E\|M_t\|^p &\eqsim_{p,X}\mathbb E\tnorm{M_t}^p \eqsim_{p} \mathbb E [M]_t^{\frac p2} = \mathbb E \bigl([M^c]_t + [M^q]_t + [M^a]_t\bigr)^{\frac p2}\\
 &\eqsim_{p} \mathbb E [M^c]_t^{\frac p2} + \mathbb E [M^q]_t^{\frac p2} + \mathbb E [M^a]_t^{\frac p2} \eqsim_p \mathbb E \tnorm{M^c_t}^p + \mathbb E \tnorm{M^q_t}^p + \mathbb E \tnorm{M^a_t}^p \\ &\eqsim_{p, X}\mathbb E \|M^c_t\|^p + \mathbb E \|M^q_t\|^p + \mathbb E \|M^a_t\|^p.
\end{align*}
\end{proof}
\begin{proof}[Proof of Theorem \ref{thm:mainSI}]
 First of all notice that $\Phi \cdot M$ is an $L^q(S)$-valued local martingale, so by the Doob's maximal inequality
 \begin{equation}\label{eq:doob'smaximalmainSI}
    \mathbb E \sup_{0\leq s\leq t} \| (\Phi\cdot M)_s \|^p_{L^q(S)} \eqsim_p \mathbb E  \| (\Phi\cdot M)_t \|^p_{L^q(S)}.
 \end{equation}
Since $\Phi$ is elementary, we can assume that $X$ is finite dimensional. Consequently, \eqref{eq:mainSI} holds by \eqref{eq:doob'smaximalmainSI}, Lemma \ref{lem:MeyYoedecfindim}, Lemma \ref{lemma:intcontcontintdisdis} and Theorem \ref{thm:preintgenmart}.
\end{proof}

\begin{remark}
Let $M = (M_n)_{n\geq 0}$ be a discrete $L^q$-valued martingale. Then due to the {\it Strong Doob maximal inequality} (also known as the {\it Fefferman-Stein inequality}), presented e.g.\ in \cite[Theorem 3.2.7]{HNVW1} and \cite[Theorem 2.6]{AntThesis}, 
 \[
  \Big(\mathbb E \Bigl(\int_S|\sup_{n\geq 0}M_n(s)|^q\ud s\Bigr)^{\frac pq}\Big)^{\frac{1}{p}} \eqsim_{p,q} (\mathbb E \sup_{n\geq 0}\|M_n\|_{L^q(S)}^p)^{\frac{1}{p}}.
 \]
As a consequence, for any continuous time martingale $M:\mathbb R_+\times \Omega \to L^q(S)$
\[
 \Big(\mathbb E \|\sup_{t\geq 0}M_t\|_{L^q(S)}^p\Big)^{\frac{1}{p}} \eqsim_{p,q} (\mathbb E \sup_{t\geq 0}\|M_t\|_{L^q(S)}^p)^{\frac{1}{p}}. 
\]
Indeed, this follows by the existence of a pointwise c\`adl\`ag version of $M$ and by approximating $M$ by a discrete-time martingale. Thus, all the sharp bounds for stochastic integrals proved in this section, in particular Theorems~\ref{thm:acessjumpsL^q}, \ref{thm:mainintranmeas}, \ref{thm:intwrtpurdismart}, \ref{thm:itoisomcontlocmart}, and, finally, Theorems~\ref{thm:preintgenmart} and \ref{thm:mainSI}, remain valid if we move the supremum over time inside the $L^q$-norm.
\end{remark}

\section{Burkholder-Rosenthal inequalities in noncommutative $L^q$-spaces}
\label{sec:NCLq}

In this section we discuss an extension of Theorem~\ref{thm:summaryBRIntro} to martingales taking values in noncommutative $L^q$-spaces. Let $\cM$ be a von Neumann algebra acting on a complex Hilbert space $H$, which is equipped with a normal, semi-finite faithful trace $\tr$. A closed,
densely defined linear operator $x$ on $H$ is \emph{affiliated} with the von Neumann algebra $\cM$ if $ux=xu$ for any unitary element $u$ in the commutant $\cM '$ of
$\cM$. For such an operator we define its \emph{distribution function} by
$$d(v;x) = \tr(e^{|x|}(v,\infty)) \qquad (v\geq 0),$$
where $e^{|x|}$ is the spectral measure of $|x|$. The \emph{decreasing rearrangement} or \emph{generalized singular value function} of $x$ is defined by
$$\mu_t(x) = \inf\{v>0 \ : \ d(v;x)\leq t\} \qquad (t\geq 0).$$
We call $x$ \emph{$\tr$-measurable} if $d(v;x)<\infty$ for some $v>0$. Let $S(\tr)$ denote the linear space of all $\tr$-measurable operators. One can show that
$S(\tr)$ is a metrizable, complete topological $*$-algebra with respect to the measure topology. Moreover, the trace $\tr$ extends to a trace (again denoted by $\tr$) on
the set $S(\tr)_{+}$ of positive $\tr$-measurable operators by setting
\begin{equation}
\label{eqn:traceStau} \tr(x) = \int_0^{\infty} \mu_t(x) \ dt, \qquad x \in S(\tr)_+.
\end{equation}
For $1\leq q<\infty$ we define
\begin{equation}
\label{eqn:NCLqNorm} \|x\|_{L^q(\cM)} = (\tr(|x|^q))^{\frac{1}{q}}, \qquad x \in S(\tr).
\end{equation}
The space $L^{q}(\cM)$ of all $x \in S(\tr)$ satisfying $\|x\|_{L^q(\cM)}<\infty$ is a Banach space and called the \textit{noncommutative $L^q$-space} associated with the pair
$(\cM,\tr)$. This class of spaces includes classical $L^q$-spaces with respect to a $\si$-finite measure space $(S,\Si,\rho)$. Identify $f \in L^{\infty}(S)$ with the multiplication operator
$$M_f(h) := fh, \qquad h \in L^2(S).$$
on the Hilbert space $L^2(S)$. One can show that
$$\cM := \{M_f \ : \ f \in L^{\infty}(S)\}$$
is a von Neumann subalgebra of $B(L^2(S))$. Now identify $L^{\infty}(S)$ with $\cM$. The functional
$$\tr(f) : = \int_{S} \ f \ d\rho, \qquad f \in L^{\infty}(S)_+$$
defines a normal, semi-finite faithful trace on $\cM$ and $L^{q}(\cM)$ coincides with the Lebesgue space $L^q(S)$, where the functions in $L^q(S)$
are identified with, in general unbounded, multiplication operators on $L^2(S)$.\par
Schatten spaces also arise as noncommutative $L^q$-spaces: if $H$ is a separable Hilbert space with an orthonormal basis $(e_i)_{i\geq 1}$, then the standard trace
\begin{equation*}
\label{eqn:traceMatrices} \mathrm{Tr}(x) = \sum_{i\geq 1} \langle x e_i,e_i\rangle
\end{equation*}
is a normal, semi-finite, faithful trace and $L^q(B(H),\text{Tr})$ is the $q$-th Schatten space. For more information on noncommutative $L^q$-spaces we refer to \cite{PiX03}.\par
Let us now recall the essential facts on noncommutative $L^q$-spaces needed to extend Theorem~\ref{thm:summaryBRIntro}. For $1\leq q<\infty$ and $\frac{1}{q} + \frac{1}{q'} = 1$, the familiar duality $L^q(\cM)^*=L^{q'}(\cM)$ holds isometrically, with the duality bracket given by $\langle
x,y\rangle = \tr(xy)$. Moreover, $L^q(\cM)$ is a UMD Banach space if and only if $1<q<\infty$. In addition, we will use conditional sequence spaces. Let $1\leq q<\infty$. Suppose that $\cN$ is a von Neumann algebra equipped with a normal, semi-finite faithful trace $\si$ and let $(\cN_i)$ be an increasing sequence of von Neumann subalgebras such that $\si|_{\cN_i}$ is again semi-finite. Let $\cE_i:\cN\rightarrow \cN_i$ be the conditional expectation with respect to $\cN_i$. For a finite sequence
$(x_i)$ in $\cN$ we define
\begin{equation*}
\label{eqn:colRowNormLqConditional}
\begin{split}
\|(x_i)\|_{L^q(\cN;(\cE_i),\ell^2_c)} & = \Big\|\Big(\sum_i \cE_i|x_i|^2\Big)^{\frac{1}{2}}\Big\|_{L^q(\cN)};\\
\|(x_i)\|_{L^q(\cN;(\cE_i),\ell^2_r)} & = \Big\|\Big(\sum_i \cE_i|x_i^*|^2\Big)^{\frac{1}{2}}\Big\|_{L^q(\cN)}.
\end{split}
\end{equation*}
Using techniques from Hilbert $C^*$-modules it was shown by M. Junge \cite{Jun02} that
\begin{equation*}
\begin{split}
& \{(x_i)_{i=1}^n \ : \ x_i \in \cN, \ n\geq 1, \ \|(x_i)\|_{L^q(\cN;(\cE_i),\ell^2_c)}<\infty\} \ \mathrm{and}\\
& \{(x_i)_{i=1}^n \ : \ x_i \in \cN, \ n\geq 1, \ \|(x_i)\|_{L^q(\cN;(\cE_i),\ell^2_r)}<\infty\}
\end{split}
\end{equation*}
are normed linear spaces. By taking the completion of these spaces we obtain the \textit{conditional column} and \textit{row space}, respectively. Moreover, for
$1<q,q'<\infty$ with $\frac{1}{q} + \frac{1}{q'} = 1$
$$(L^q(\cN;(\cE_i),\ell^2_c))^* = L^{q'}(\cN;(\cE_i),\ell^2_r), \qquad (L^q(\cN;(\cE_i),\ell^2_r))^* = L^{q'}(\cN;(\cE_i),\ell^2_c)$$
holds isometrically, with the duality bracket given by
$$\langle(x_i),(y_i)\rangle = \sum_i \tr(x_i y_i).$$
In Section~\ref{sec:B-RProof} we used a variation of these results in the following special case. Let $(\Om,\cF,\bP)$ be a probability space, $\cM$ be a semi-finite von Neumann algebra and $\cN$ be the tensor product von Neumann algebra $L^{\infty}(\Om)\vNT\cM$, equipped with the tensor product trace $\E\ot\tr$. Recall that, for any $1\leq q<\infty$, the map defined on simple functions in the Bochner space $L^q(\Om;L^q(\cM))$ by
$$I_q\Big(\sum_i \chi_{A_i} x_i\Big) = \sum_i \chi_{A_i} \ot x_i$$
extends to an isometric isomorphism
\begin{equation}
\label{eqn:tensorBochnerIden} L^q(\Om;L^q(\cM)) = L^q(L^{\infty}(\Om)\vNT\cM).
\end{equation}
Let $\bF=(\cF_i)_{i\geq 0}$ be a filtration in $\cF$, $\cN_i = L^{\infty}(\cF_{i-1})\vNT\cM$ and let $\cE_i$ be the associated conditional expectation. Under the identification (\ref{eqn:tensorBochnerIden}), the element $\cE_i(x)$ coincides with $\E_{i-1}(x)$, whenever $x \in L^{\infty}(\Om)\ot\cM$. In particular, for any finite sequence $(f_i)$ in $L^{\infty}(\Omega)\ot\cM$,
\begin{equation*}
\label{eqn:abrevCondNorm} \|(f_i)\|_{L^q(\cN;(\cE_{i}),l^2_c)} = \Big\|\Big(\sum_i \E_{i-1}|f_i|^2\Big)^{\frac{1}{2}}\Big\|_{L^q(\cM)}.
\end{equation*}
In particular, this shows that the space $S_q^q$ considered in Section~\ref{sec:B-RProof} is a Banach space and $(S_q^q)^*=S_{q'}^{q'}$ isomorphically if $1<q<\infty$. By a straightforward modification of the arguments in \cite{Jun02}, one can show that the expressions in 
\begin{equation}
\label{eqn:SpqNormsNC}
\begin{split}
\|(f_i)\|_{S_{q,c}^p} & = \Big(\E\Big\|\Big(\sum_i \E_{i-1}|f_i|^2\Big)^{\frac{1}{2}}\Big\|_{L^q(\cM)}^p\Big)^{\frac{1}{p}},\\
\|(f_i)\|_{S_{q,r}^p} & = \Big(\E\Big\|\Big(\sum_i \E_{i-1}|f_i^*|^2\Big)^{\frac{1}{2}}\Big\|_{L^q(\cM)}^p\Big)^{\frac{1}{p}}.
\end{split}
\end{equation}
define two norms and that for $1<p,q<\infty$
\begin{equation}
\label{eqn:SqdualNC}
(S_{q,c}^p)^* = S_{q',r}^{p'}, \qquad (S_{q,c}^p)^* = S_{q',r}^{p'}
\end{equation}
isomorphically (and constants depending only on $p$ and $q$), with associated duality bracket given by
$$\langle(x_i),(y_i)\rangle = \sum_i \E\tr(x_i y_i).$$ 
We can now formulate the extension of Theorem~\ref{thm:summaryBRIntro}. In analogy with \eqref{eqn:DpqNorms} we define two norms on the linear space of all finite sequences $(f_i)$ of random variables in $L^{\infty}(\Om;L^q(\cM))$ by
\begin{equation*}
\begin{split}
\|(f_i)\|_{D_{q,q}^p} & = \Big(\E\Big(\sum_i \E_{i-1}\|f_i\|_{L^q(\cM)}^q\Big)^{\frac{p}{q}}\Big)^{\frac{1}{p}}, \\
\|(f_i)\|_{D_{p,q}^p} & = \Big(\sum_i\E\|f_i\|_{L^q(\cM)}^p\Big)^{\frac{1}{p}},
\end{split}
\end{equation*}
Clearly these expression define two norms and we let $D_{p,q}^p$ and $D_{q,q}^p$ denote the completions in these norms. Since $L^q(\cM)$ is reflexive for $1<q<\infty$, it follows from Theorem~\ref{thm:(H^{s_q}_p(X))^*= H^{s_{q'}}_{p'}(X^*)} that $(D_{q,q}^p)^*=D_{q',q'}^{p'}$ if $1<p,q<\infty$. Since $L^q(\cM)$ is a UMD spaces for any $1<q<\infty$, one can follow the proof of Theorem~\ref{thm:summaryBRIntro} verbatim (using the noncommutative version of Theorem~\ref{thm:summaryRosIntro} from \cite{Dir14} and replacing the use of Jensen's inequality by Kadison's inequality) to obtain the following noncommutative version. 
\begin{theorem}
\label{thm:BRNC} Let $1<p,q<\infty$ and let $\cM$ be a semi-finite von Neumann algebra. If $(d_i)$ is an $L^q(\cM)$-valued martingale difference sequence, then
\begin{equation*}
\Big(\E\Big\|\sum_i d_i\Big\|_{L^q(\cM)}^p\Big)^{\frac{1}{p}} \eqsim_{p,q} \|(d_i)\|_{s_{p,q}},
\end{equation*}
where $s_{p,q}$ is given by
\begin{align*}
S_{q,c}^p \cap S_{q,r}^p \cap D_{q,q}^p \cap D_{p,q}^p & \ \ \mathrm{if} \ \ 2\leq q\leq p<\infty;\\
S_{q,c}^p \cap S_{q,r}^p \cap (D_{q,q}^p + D_{p,q}^p) & \ \ \mathrm{if} \ \ 2\leq p\leq q<\infty;\\
(S_{q,c}^p \cap S_{q,r}^p \cap D_{q,q}^p) + D_{p,q}^p & \ \ \mathrm{if} \ \ 1<p<2\leq q<\infty;\\
(S_{q,c}^p + S_{q,r}^p + D_{q,q}^p) \cap D_{p,q}^p & \ \ \mathrm{if} \ \ 1<q<2\leq p<\infty;\\
S_{q,c}^p + S_{q,r}^p + (D_{q,q}^p \cap D_{p,q}^p) & \ \ \mathrm{if} \ \ 1<q\leq p\leq 2;\\
S_{q,c}^p + S_{q,r}^p + D_{q,q}^p + D_{p,q}^p & \ \ \mathrm{if} \ \ 1<p\leq q\leq 2.
\end{align*}
Consequently, if $\cF=\si(\cup_{i\geq 0}\cF_i)$, then the map $f \mapsto (\E_i f - \E_{i-1} f)_{i\geq 1}$ induces an isomorphism between $L^p_0(\Om;L^q(\cM))$, the subspace of mean-zero random variables in $L^p(\Om;L^q(\cM))$, and $s_{p,q}$.
\end{theorem}

\appendix

\section{Duals of $\mathcal S_q^p$, $\mathcal D_{q,q}^p$, $\mathcal D_{p,q}^p$, $\hat{\mathcal S}_q^p$, $\hat{\mathcal D}_{q,q}^p$, and $\hat{\mathcal D}_{p,q}^p$}
\label{sec:appDuals}

In this section we will find the duals of $\mathcal S_q^p$, $\mathcal D_{q,q}^p$, $\mathcal D_{p,q}^p$, $\hat{\mathcal S}_q^p$, $\hat{\mathcal D}_{q,q}^p$, and $\hat{\mathcal D}_{p,q}^p$ for all $1<p,q<\infty$. As a consequence, we show the duality for the space $\cI_{p,q}$ that was used to prove the lower bounds in Theorem~\ref{thm:mainintranmeas}.

\subsection{$\mathcal D_{q,q}^p$ and $\mathcal D_{p,q}^p$ spaces}

Let $X$ be a Banach space and consider any random measure $\nu$ on $\R_+\times J$. In sequel we will assume that $\int_{\mathbb R_+ \times J} \mathbf 1_{A}\ud \nu$ is an \mbox{$\overline{\mathbb R}_+$-valued} random variable for each $\mathcal B(\mathbb R_+) \otimes \mathcal J$-measurable $A\subset \mathbb R_+ \times J$. Notice that this condition always holds for any optional random measure $\nu$. Indeed, without loss of generality we may assume that there exist $A_{\R_+}\in \mathcal B(\mathbb R_+)$ and $A_J \in \mathcal J$ such that $A = A_{\R_+}\times A_J$. Let $\widetilde A = A \times \Omega$. Then $\widetilde A \in \widetilde {\mathcal O}$ (since $A_{\R_+}\times \Omega \in \mathcal O$), therefore $\mathbf 1_{\widetilde A}\star \nu$ is an optional process, and
\[
 \int_{\mathbb R_+ \times J} \mathbf 1_{A}\ud \nu = \lim_{t\to\infty}(\mathbf 1_{\widetilde A}\star \nu)_t
\]
is an \mbox{$\overline{\mathbb R}_+$-valued} $\mathcal F$-measurable function as a monotone limit of \mbox{$\overline{\mathbb R}_+$-valued} \mbox{$\mathcal F$-mea}\-su\-rable functions.

We define $\mathcal D^p_q(X)$ to be the space of all $\mathcal B(\mathbb R_+)\otimes \mathcal F \otimes \mathcal J$-strongly measurable functions $f: \mathbb R_+ \times\Omega \times J \to X$ such that
\[
 \|f\|_{\mathcal D_{q}^p(X)} :=  \Bigl( \mathbb E \Bigl( \int_{\mathbb R_+\times J}\|f\|^q_X \ud \nu \Bigr)^{\frac pq}\Bigr)^{\frac 1p}<\infty. 
\]
Recall that the measure $\mathbb P \otimes \nu$ on $\mathcal B(\mathbb R_+)\otimes \mathcal F \otimes \mathcal J$ is defined by setting 
\[
\mathbb P \otimes \nu\Big(\bigcup_{i=1}^n A_i\times B_i\Big) := \sum_{i=1}^n \mathbb E(\mathbf 1_{A_i}\nu(B_i)),
\]
for disjoint $A_i \in \mathcal F$ and disjoint $B_i \in \mathcal B(\mathbb R_+) \otimes \mathcal J$ and extending $\mathbb P \times \nu$ to $\mathcal B(\mathbb R_+)\otimes \mathcal F \otimes \mathcal J$ via the Carath\'eodory extension theorem. 

The following result is well-known if $\nu$ is a deterministic measure. The argument for random measures is similar and provided for the reader's convenience. 
\begin{theorem}\label{thm:dualofD_q^p(X)}
 Let $1<p,q<\infty$, $X$ be reflexive. Then $(\mathcal D_{q}^p(X))^*  =\mathcal D_{q'}^{p'}(X^*)$. Moreover
 \begin{equation}\label{eq:dualofD_q^p(X)}
  \|\phi\|_{\mathcal D_{q'}^{p'}(X^*)} = \|\phi\|_{(\mathcal D_{q}^p(X))^*},\;\;\; \phi\in \mathcal D_{q'}^{p'}(X^*).
 \end{equation}
\end{theorem}
\begin{proof}
 First we suppose that $\mathbb E\nu(\mathbb R_+ \times J) < \infty$. By approximation we can assume that $\nu(\mathbb R_+ \times J)\leq N$ a.s., for some $N\in \mathbb{N}$. In this case we can proceed with a~standard argument using the Radon-Nikodym property of $X^*$. Let $F\in (\mathcal D_{q}^p(X))^*$. On $\mathcal B(\mathbb R_+)\otimes \mathcal F \otimes \mathcal J$ we can define and $X^*$-valued measure $\theta$ by setting 
 \[
  \langle \theta(A), x\rangle:= F(\mathbf 1_{A}\cdot x) \qquad (\mathcal B(\mathbb R_+)\otimes A\in \mathcal F \otimes \mathcal J, x\in X).
 \]
It is straightforward to verify that $\theta$ is $\sigma$-additive and absolutely continuous with respect to $\mathbb P \times \nu$. Moreover, $\theta$ is of finite variation. Indeed, if $A_1,\ldots,A_n$ is a disjoint partition of $\mathbb R_+\times \Omega \times J$, then 
\begin{align}
\label{eqn:totvarBound}
\sum_{i=1}^n \|\theta(A_i)\| & = \sup_{(x_i)_{i=1}^n \subset B_X} \sum_{i=1}^n F(\mathbf 1_{A_i} x_i) \nonumber \\
& =  \sup_{(x_i)_{i=1}^n \subset B_X} F\Big(\sum_{i=1}^n \mathbf 1_{A_i} x_i\Big)  \nonumber\\
& \leq \|F\|_{(\mathcal D_{q}^p(X))^*} \sup_{(x_i)_{i=1}^n \subset B_X} \Big(\E\Big(\int_{\R_+\ti J} \Big\|\sum_{i=1}^n \mathbf 1_{A_i} x_i\Big\|_X^q \ d\nu\Big)^{p/q}\Big)^{1/p}  \nonumber\\
& = \|F\|_{(\mathcal D_{q}^p(X))^*} \sup_{(x_i)_{i=1}^n \subset B_X} \Big(\E\Big(\int_{\R_+\ti J} \sum_{i=1}^n \mathbf 1_{A_i} \|x_i\|_X^q \ d\nu\Big)^{p/q}\Big)^{1/p}  \nonumber\\
& \leq \|F\|_{(\mathcal D_{q}^p(X))^*} (\E\nu(\R_+\ti J)^{p/q})^{1/p}.
\end{align}
(Here $B_X$ is a unit ball in $X$). By the Radon-Nikodym property of $X^*$, there exists an $\mathcal B(\mathbb R_+) \otimes \mathcal F \otimes \mathcal J$-strongly measurable $X^*$-valued function $f$ such that 
$$F(g) = F_f(g) = \mathbb E \int_{\mathbb R_+ \times J}\langle f,g \rangle\ud \nu$$ 
for each $g\in \mathcal D_{q}^p(X)$. By H\"older's inequality, it is immediate that 
$$\|F\|_{(\mathcal D_{q}^p(X))^*} \leq \|f\|_{\mathcal D_{q'}^{p'}(X^*)}.$$
To show the reverse estimate, we may assume that $f\in \mathcal D_{q'}^{p'}(X^*)$ has norm $1$ and show that $\|F_{f}\|_{(\mathcal D_{q}^p(X))^*}\geq 1$. By approximation, we may furthermore assume that $f$ is simple, i.e., 
$$f=\sum_{m,n} \mathbf 1_{A_n} \ \mathbf 1_{B_{nm}} \  x_{nm}^*,$$
for $A_n\in \cF$ and $B_{nm} \in \cB(\R_+)\ot \cJ$ disjoint and $x_{nm}^*\in X^*$. Define
\begin{equation}
\label{eqn:normFunction}
g = \sum_{m,n} \mathbf 1_{A_n} \Big(\sum_m \nu(B_{nm})\|x_{nm}^*\|_{X^*}^{q'}\Big)^{\frac{p'}{q'}-1} \ \mathbf 1_{B_{nm}} \ x_{nm} \|x_{nm}^*\|_{X^*}^{q'-1},
\end{equation}
where the $x_{nm}\in X$ satisfy the condition in Lemma \ref{lemma:P_eps}, i.e.\ for some $0<\eps<1$
$$(1-\eps)\|x_{nm}^*\|\leq \langle x_{nm},x_{nm^*}\rangle, \qquad \|x_{nm}\|_X=1.$$
By assumption,
$$\|f\|_{D_{q'}^{p'}(X^*)}^{p'} = \sum_{n} \bP(A_n)\Big(\sum_m \nu(B_{nm})\|x_{nm}^*\|_{X^*}^{q'}\Big)^{p'/q'} = 1.$$
Therefore, also
\begin{align*}
\|g\|_{D_q^p(X)}^q & = \sum_{n} \bP(A_n)\Big(\sum_m \nu(B_{nm})\|x_{nm}^*\|_{X^*}^{q'}\Big)^{\Big(\frac{p'}{q'}-1\Big)p} \Big(\sum_m \nu(B_{nm})\|x_{nm}^*\|_{X^*}^{(q'-1)q}\Big)^{\frac{p}{q}} \\
& = \sum_{n} \bP(A_n)\Big(\sum_m \nu(B_{nm})\|x_{nm}^*\|^{q'}\Big)^{p'/q'} = 1,
\end{align*}
as 
\begin{equation}
\label{eqn:HoldConj}
qq'=q+q' \qquad \frac{pp'}{q'} - p + \frac{p}{q} = \frac{p'}{q'}.
\end{equation}
Moreover,
\begin{align*}
F_f(g) & = \E\int \langle f,g\rangle \ d\nu \\
& = \sum_{n} \bP(A_n)\sum_m \nu(B_{nm}) \langle x_{nm},x_{nm^*}\rangle \|x_{nm}^*\|_{X^*}^{q'-1} \Big(\sum_m \nu(B_{nm})\|x_{nm}^*\|_{X^*}^{q'}\Big)^{\frac{p'}{q'}-1} \\
& \geq  \sum_{n} \bP(A_n)\sum_m \nu(B_{nm}) (1-\eps) \|x_{nm}^*\|_{X^*}^{q'} \Big(\sum_m \nu(B_{nm})\|x_{nm}^*\|_{X^*}^{q'}\Big)^{\frac{p'}{q'}-1} \\
& = (1-\eps) \sum_{n} \bP(A_n)\Big(\sum_m \nu(B_{nm})\|x_{nm}^*\|_{X^*}^{q'}\Big)^{\frac{p'}{q'}} = (1-\eps).
\end{align*}  
Since $\eps$ was arbitrary, $\|F_{f}\|_{(\mathcal D_{q}^p(X))^*}\geq 1$.\par
Let now $\mathbb E\nu(\mathbb R_+ \times J) = \infty$ and assume that $\bP\times \nu$ is $\si$-finite. Then there exists a sequence $(A_n)_{n\geq 1}\subset \mathcal B(\mathbb R_+)\otimes  \mathcal F \otimes \mathcal J$ such that $A_n \subset A_{n+1}$ for each $n\geq 1$, $\cup_{n\geq 1}A_n =  \mathbb R_+ \times \Omega \times J$, and $\mathbb P \otimes \nu(A_n)<\infty$ for each $n\geq 1$. Let $\nu_n := \nu \cdot\mathbf 1_{A_n}$. Then each $F\in (\mathcal D_{q}^p(X))^*$ can be considered as a linear functional on the closed subspace of $\mathcal D_{q}^p(X)$ consisting of all functions with support in $A_n$. By the previous part of the proof, for each $n\geq 1$ there exists $f_n \in \mathcal D_{q'}^{p'}(X^*)$ with support in $A_n$ such that 
$$F(g \cdot \mathbf 1_{A_n}) = F_{f_n}(g\cdot \mathbf 1_{A_n}) = \mathbb E \int_{\mathbb R_+ \times J}\langle f_n,g \rangle \mathbf 1_{A_n}\ud \nu$$ 
and 
$$\|f_n\|_{\mathcal D_{q'}^{p'}(X^*)}\leq \|F_{f_n}\|_{(\mathcal D_{q}^p(X))^*} \leq \|F\|_{(\mathcal D_{q}^p(X))^*}.$$
Obviously $f_{n+1}\mathbf 1_{A_n} = f_n$ for each $n\geq 1$, hence there exists $f: \Omega \times \mathbb R_+ \times J \to X^*$ such that $f\mathbf 1_{A_n} = f_n$ for each $n\geq 1$. But then Fatou's lemma implies
$$\|f\|_{\mathcal D_{q'}^{p'}(X^*)} \leq \liminf_{n\to \infty}\|f_n\|_{\mathcal D_{q'}^{p'}(X^*)}\leq \|F\|_{(\mathcal D_{q}^p(X))^*},$$ 
so $f\in \mathcal D_{q'}^{p'}(X^*)$. On the other hand, by H\"older's inequality 
$$\|F\|_{(\mathcal D_{q}^p(X))^*}\leq \|f\|_{\mathcal D_{q'}^{p'}(X^*)}.$$ 
Since the bounded linear functionals $F$ and $F_f$ agree on a dense subset of $\mathcal D_{q}^p(X)$, it follows that $F=F_f$ and \eqref{eq:dualofD_q^p(X)} holds.\par
Finally, let $\nu$ be general. Let $F\in (\mathcal D_{q}^p(X))^*$ be of norm $1$. Let $\eps_n\downarrow 0$ and let $(g_n)_{n\geq 1}$ be a sequence in the unit sphere of $\mathcal D_{q}^p(X)$ satisfying $F(g_n)\geq (1-\eps_n)$. By strong measurability of the $g_n$, there exists an $A \in \mathcal B(\mathbb R_+) \otimes \mathcal F \otimes \mathcal J$ so that $\bP\times \nu$ is $\si$-finite on $A$ and $g_n=0$ on $A^c$ $\bP\times \nu$-a.e. Let $\tilde{F} \in (\mathcal D_{q}^p(X))^*$ be defined by $\tilde{F}(g)=F(g\mathbf 1_A)$. The previous part of the proof shows that there exists an $f \in \mathcal D_{q'}^{p'}(X^*)$ so that $\tilde{F}=F_f$ and $\|\tilde{F}\|_{(\mathcal D_{q}^p(X))^*} = \|f\|_{\mathcal D_{q'}^{p'}(X^*)}$. It remains to show that $F=\tilde{F}$. To prove this, suppose that there exists a $g_0\in \mathcal D_{q}^p(X)$ of norm $1$ with $\text{supp}(g_0)\subset A^c$ and $F(g_0)=\delta>0$. Let $0<\lambda<1$. Then, for any $n\geq 1$,
$$\|(1-\lambda^p)^{1/p} g_0 + \lambda g_n\|_{\mathcal D_{q}^p(X)}^p = (1-\lambda^p) \|g_0\|_{\mathcal D_{q}^p(X)}^p + \lambda^p \|g_n\|_{\mathcal D_{q}^p(X)}^p = 1$$
and 
$$F((1-\lambda^p)^{1/p} g_0 + \lambda g_n)  \geq (1-\lambda^p)^{1/p}\delta + \lambda (1-\eps_n).$$
As a consequence, 
$$\|F\|\geq \sup_{0<\lambda<1} (1-\lambda^p)^{1/p}\delta + \lambda.$$
One easily checks that the supremum is attained in 
$$\lambda = \Big(1+\delta^{1/(1-\frac{1}{p})}\Big)^{-1/p}$$
and so $\|F\|>1$, a contradiction. 
\end{proof}
We now turn to proving a similar duality statement for $\hat{\mathcal D}_{q}^p(X)$, the space of all $\widetilde{\cP}$-measurable functions in $\mathcal D_{q}^p(X)$. In the proof we will use the following `reverse' version of the dual Doob inequality \cite[Lemma~2.10]{DMN12}.
\begin{lemma}[Reverse dual Doob inequality]
\label{lem:dualDoobRev} Fix $0<p\leq 1$. Let $\bF=(\cF_n)_{n\geq 0}$ be a filtration and let $(\E_n)_{n\geq 0}$
be the associated sequence of conditional expectations.
If $(f_n)_{n\geq 0}$ is a sequence of non-negative random variables in $L^1(\bP)$, then
$$\Big(\E\Big|\sum_{n\geq 0} f_n\Big|^p\Big)^{\frac{1}{p}} \leq p^{-1} \Big(\E\Big|\sum_{n\geq 0}
\E_n f_n\Big|^p\Big)^{\frac{1}{p}}.$$
\end{lemma}
\begin{theorem}\label{thm:dualofL^p_P(L^q(X))}
Let $X$ be a reflexive space and let $\nu$ be a predictable, $\widetilde{\cP}$-$\si$-finite random measure on $\cB(\R_+)\ot \cJ$ that is non-atomic in time. Then, for $1<p,q<\infty$, 
$$(L^p_{\widetilde{\cP}}(\bP;L^q(\nu;X)))^* = L^{p'}_{\widetilde{\cP}}(\bP;L^{q'}(\nu;X^*))$$
with isomorphism given by 
$$g\mapsto F_g,\qquad F_g(h) = \E\int_{\R_+\ti J} \langle g,h\rangle d\nu \qquad (g \in L^{p'}_{\widetilde{\cP}}(\bP;L^{q'}(\nu)), h\in L^p_{\widetilde{\cP}}(\bP;L^q(\nu))).$$
Moreover,
\begin{equation}
\label{eqn:predDualityNE}
\min\Big\{\Big(\frac{p}{q}\Big)^{1/q} \frac{q'}{p'}, \Big(\frac{p'}{q'}\Big)^{1/q'} \frac{q}{p}\Big\} \|g\|_{L^{p'}_{\widetilde{\cP}}(\bP;L^{q'}(\nu;X^*))} \leq \|F_g\|\leq \|g\|_{L^{p'}_{\widetilde{\cP}}(\bP;L^{q'}(\nu;X^*))}.
\end{equation}
\end{theorem}
\begin{proof}
\emph{Step 1: reduction.} It suffices to prove the result for $p\leq q$. Indeed, once this is known we can deduce the case $q\leq p$ as follows. Observe that $L^{p'}_{\widetilde{\cP}}(\bP;L^{q'}(\nu;X^*))$ is a closed subspace of $\cD^{p'}_{q'}(X^*)=L^{p'}(\bP;L^{q'}(\nu;X^*))$. By Theorem~\ref{thm:dualofD_q^p(X)}, $\cD^{p'}_{q'}(X^*)$ is reflexive and therefore $L^{p'}_{\widetilde{\cP}}(\bP;L^{q'}(\nu;X^*))$ is reflexive as well. Therefore, as $p'\leq q'$,
$$(L^p_{\widetilde{\cP}}(\bP;L^q(\nu;X)))^* = L^{p'}_{\widetilde{\cP}}(\bP;L^{q'}(\nu;X^*))^{**} = L^{p'}_{\widetilde{\cP}}(\bP;L^{q'}(\nu;X^*)).$$
Hence, if $F\in (L^p_{\widetilde{\cP}}(\bP;L^q(\nu;X)))^*$, then there exists an $f\in L^{p'}_{\widetilde{\cP}}(\bP;L^{q'}(\nu;X^*))$ so that for any $g\in L^p_{\widetilde{\cP}}(\bP;L^q(\nu;X))$
$$F(g) = F_g(f) = \E\int_{\R_+\ti J} \langle f,g\rangle d\nu.$$
Moreover, the bounds \eqref{eqn:predDualityNE} follow from Lemma~\ref{lemma:simple}. Thus, for the remainder of the proof, we can assume that $p\leq q$.\par
\emph{Step 2: norm estimates.} Let us now show that \eqref{eqn:predDualityNE} holds. Since the upper bound is immediate from H\"older's inequality, we only need to show that for any $g \in L^{p'}_{\widetilde{\cP}}(\bP;L^{q'}(\nu;X^*))$,
\begin{equation}
\label{eqn:normingPred}
\|F_g\|\geq \Big(\frac{p}{q}\Big)^{1/q} \frac{q'}{p'} \|g\|_{L^{p'}(\bP;L^{q'}(\nu;X^*))}.
\end{equation}
It suffices to show this on a dense subset of $L^{p'}_{\widetilde{\cP}}(\bP;L^{q'}(\nu;X^*))$. Indeed, suppose that $g_n\to g$ in $L^{p'}_{\widetilde{\cP}}(\bP;L^{q'}(\nu;X^*))$ and that \eqref{eqn:normingPred} holds for $g_n$, for all $n\geq 1$. Then,
$$\Big(\frac{p}{q}\Big)^{1/q} \frac{q'}{p'} \|g_n\|_{L^{p'}(\bP;L^{q'}(\nu;X^*))} \leq \|F_{g_n}\| \leq \|F_{g}\| + \|g-g_n\|_{L^{p'}(\bP;L^{q'}(\nu;X^*))},$$
and by taking limits on both sides we see that $g$ also satisfies \eqref{eqn:normingPred}.\par
Let us first assume that 
\begin{equation}
\label{eqn:nuAbsCont}
\nu((s,t]\ti J)\leq (t-s)\qquad \text{a.s.}, \qquad \text{for all } 0\leq s\leq t 
\end{equation}
By the previous discussion, we may assume that $\|g\|_{L^{p'}(\bP;L^{q'}(\nu;X^*))}=1$ and that $g$ is of the form
$$g= \sum_{n=0}^{N_{m_*}}\sum_{\ell=0}^L \mathbf 1_{(n/2^{m_*},(n+1)/2^{m_*}]}\mathbf 1_{{B_{\ell}}}g_{n\ell},$$
where $N_{m_*}<\infty$, $g_{n\ell}$ is simple and $\cF_{n/2^{m_*}}$-measurable for all $n$ and $\ell$, and the ${B_{\ell}}$ are disjoint sets in $\cJ$ of finite $\bP\otimes\nu$-measure. For $m\geq m_*$ define
$$g^{(m)} = \sum_{n=0}^{N_m}\sum_{\ell=0}^L \mathbf 1_{(n/2^m,(n+1)/2^{m}]}\mathbf 1_{{B_{\ell}}}g_{n\ell}^{(m)}$$
so that $g^{(m)}=g$. Then clearly, $g_{n\ell}^{(m)}$ is $\cF_{n/2^m}$-measurable for all $n$ and $\ell$. Let us now fix an $m\geq m_*$. We define, for any $0\leq k\leq N_m$,
$$\bar{s}_{q'}^k(g) = \Big(\sum_{n=0}^{k}\sum_{\ell=0}^L \|g_{n\ell}^{(m)}\|^{q'}\E_{n/2^m}\nu((n/2^m,(n+1)/2^{m}]\ti {B_{\ell}})\Big)^{1/q'}$$
and set 
$$\alpha=(\E\bar{s}_{q'}^{N_m}(g^{(m)})^{p'})^{1/p'}.$$
Let $P_{\eps}$ be as in Lemma~\ref{lemma:P_eps}. We define a $\widetilde{\cP}$-measurable function $h$ by
$$h = \sum_{n=0}^{N_m}\sum_{\ell=0}^L \mathbf 1_{(n/2^m,(n+1)/2^{m}]}\mathbf 1_{{B_{\ell}}}h_{n\ell}$$   
where, for $0\leq n\leq N_m$ and $0\leq \ell\leq L$, $h_{n\ell}$ is the $\cF_{n/2^m}$-measurable function
$$h_{n\ell} = \frac{1}{\alpha^{p'-1}} (\bar{s}_{q'}^n(g^{(m)}))^{p'-q'} \|g_{n\ell}^{(m)}\|^{q'-1} P_{\eps} g_{n\ell}^{(m)}.$$
Since $p/q\leq 1$, Lemma~\ref{lem:dualDoobRev} implies
\begin{align*}
\|h\|_{L^p(\bP;L^q(\nu))} & = \Big(\E\Big(\sum_{n=0}^{N_m}\sum_{\ell=0}^L \|h_{n\ell}\|^q \nu((n/2^m,(n+1)/2^{m}]\ti {B_{\ell}})\Big)^{p/q}\Big)^{1/p} \\
& \leq \Big(\frac{q}{p}\Big)^{1/q} \Big(\E\Big(\sum_{n=0}^{N_m}\sum_{\ell=0}^L \|h_{n\ell}\|^q \E_{n/2^m}\nu((n/2^m,(n+1)/2^{m}]\ti {B_{\ell}})\Big)^{p/q}\Big)^{1/p}\\
& = \Big(\frac{q}{p}\Big)^{1/q} (\E\bar{s}_q^{N_m}(h)^p)^{1/p}.
\end{align*}
Now observe that 
\begin{align*}
\bar{s}_q^{N_m}(h)^q & =  \sum_{n=0}^{N_m}\sum_{\ell=0}^L \|h_{n\ell}\|^q \E_{n/2^m}\nu((n/2^m,(n+1)/2^{m}]\ti {B_{\ell}}) \\
& \leq \frac{1}{\alpha^{(p'-1)q}} \sum_{n=0}^{N_m}\sum_{\ell=0}^L \|g_{n\ell}^{(m)}\|^{(q'-1)q} \bar{s}_{q'}^n(g^{(m)})^{(p'-q')q}  \E_{n/2^m}\nu((n/2^m,(n+1)/2^{m}]\ti {B_{\ell}}) \\
& \leq \frac{1}{\alpha^{(p'-1)q}} \bar{s}_{q'}^{N_m}(g^{(m)})^{(p'-q')q} \sum_{n=0}^{N_m}\sum_{\ell=0}^L \|g_{n\ell}^{(m)}\|^{q'}  \E_{n/2^m}\nu((n/2^m,(n+1)/2^{m}]\ti {B_{\ell}}) \\
& = \frac{1}{\alpha^{(p'-1)q}} \bar{s}_{q'}^{N_m}(g^{(m)})^{p'q-q'q + q'}.
\end{align*}
Using \eqref{eqn:HoldConj} it follows that 
\begin{align*}
\|h\|_{L^p(\bP;L^q(\nu))}^p & \leq \Big(\frac{q}{p}\Big)^{p/q} \frac{1}{\alpha^{(p'-1)p}} \bar{s}_{q'}^{N_m}(g^{(m)})^{(p'q-q'q + q')p/q} \\
& =  \Big(\frac{q}{p}\Big)^{p/q} \frac{1}{\alpha^{p'}} \E\bar{s}_{q'}^{N_m}(g^{(m)})^{p'}= \Big(\frac{q}{p}\Big)^{p/q}.
\end{align*}
Moreover, by Lemma~\ref{lemma:P_eps},
\begin{align*}
F_g(h) & = \E \sum_{n=0}^{N_m}\sum_{\ell=0}^L \langle g_{n\ell}^{(m)},h_{n\ell}\rangle \nu((n/2^m,(n+1)/2^{m}]\ti {B_{\ell}}) \\
& = \E \sum_{n=0}^{N_m}\sum_{\ell=0}^L \langle g_{n\ell}^{(m)},h_{n\ell}\rangle \E_{n/2^m}\nu((n/2^m,(n+1)/2^{m}]\ti {B_{\ell}}) \\
& \geq (1-\eps) \frac{1}{\alpha^{p'-1}} \E \sum_{n=0}^{N_m}\sum_{\ell=0}^L \|g_{n\ell}^{(m)}\|^{q'} \bar{s}_{q'}^n(g^{(m)})^{p'-q'} \E_{n/2^m}\nu((n/2^m,(n+1)/2^{m}]\ti {B_{\ell}}) \\
& = (1-\eps) \frac{1}{\alpha^{p'-1}} \E \sum_{n=0}^{N_m} \bar{s}_{q'}^n(g^{(m)})^{p'-q'}(\bar{s}_{q'}^n(g^{(m)})^{q'} - \bar{s}_{q'}^{n-1}(g^{(m)})^{q'}). 
\end{align*}
Now apply \eqref{eqn:MVTalpha} for $\alpha=p'/q'\geq 1$ and $x=\bar{s}_{q'}^n(g^{(m)})^{q'}/\bar{s}_{q'}^{n-1}(g^{(m)})^{q'}\geq 1$ to obtain
\begin{align*}
F_g(h) & \geq (1-\eps) \frac{1}{\alpha^{p'-1}} \E \sum_{n=0}^{N_m} \frac{q'}{p'}\bigl(\bar{s}_{q'}^n(g^{(m)})^{p'} - \bar{s}_{q'}^{n-1}(g^{(m)})^{p'}\bigr) \\
& = (1-\eps) \frac{q'}{p'}\frac{1}{\alpha^{p'-1}} \E\bar{s}_{q'}^{N_m}(g^{(m)})^{p'} \\
& = (1-\eps) \frac{q'}{p'}\bigl(\E\bar{s}_{q'}^{N_m}(g^{(m)})^{p'}\bigr)^{1/p'} \\
& = (1-\eps) \frac{q'}{p'} \Big(\E\Big(\sum_{n=0}^{N_m}\sum_{\ell=0}^L \|g_{n\ell}^{(m)}\|^{q'}\E_{n/2^m}\nu\bigl((n/2^m,(n+1)/2^{m}]\ti {B_{\ell}}\bigr)\Big)^{p'/q'}\Big)^{1/p'} \\
& = (1-\eps) \frac{q'}{p'} \Big(\E\Big(\sum_{n=0}^{N_m} \E_{n/2^m}\bigl((\|g\|^{q'} \star \nu)_{(n+1)/2^{m}} -  (\|g\|^{q'} \star \nu)_{n/2^{m}}\bigr)\Big)^{p'/q'}\Big)^{1/p'}
\end{align*}
In conclusion, for any $m\geq m_*$ we find
$$\|F_g\| \geq \Big(\frac{p}{q}\Big)^{1/q} \frac{q'}{p'} \Big(\E\Big(\sum_{n=0}^{N_m} \E_{n/2^m}((\|g\|^{q'} \star \nu)_{(n+1)/2^{m}} -  (\|g\|^{q'} \star \nu)_{n/2^{m}})\Big)^{p'/q'}\Big)^{1/p'}.$$
Taking $m\to\infty$, we find using Corollary~\ref{cor:Ffromcondexpex} that
$$\|F_g\| \geq \Big(\frac{p}{q}\Big)^{1/q} \frac{q'}{p'} \|g\|_{L^{p'}_{\widetilde{\cP}}(\bP;L^{q'}(\nu;X^*))}.$$
Let us now remove the additional restriction \eqref{eqn:nuAbsCont} on $\nu$. In this case, we define a strictly increasing, predictable, continuous process
$$A_t := \nu([0,t]\ti J) + t, \qquad t\geq 0$$
and a random time change $\tau=(\tau_s)_{s\geq 0}$ by
$$\tau_s = \{t \ : \ A_t=s\}.$$
By Proposition~\ref{prop:apptimechange}, $A\circ \tau(t)=t$ a.s.\ for any $t\geq 0$, and hence by continuity of $A$ and $\tau$, a.s.\ $A\circ \tau(t)=t$ for all $t\geq 0$. As was noted in \eqref{eq:nu_tauisforstep1}, we have $\nu_{\tau}((s,t]\times J)\leq t-s$ a.s.\ for all $s\leq t$. By Proposition~\ref{prop:apptimechange}, we can now write
\begin{align*}
\|F_g\| & = \sup_{\|h\|_{L_{\widetilde{\cP}}^p(\bP;L^q(\nu;X))}\leq 1} \E\int_{\R_+\ti J} \langle g,h\rangle d\nu \\
& \geq \sup_{\|\tilde{h}\circ A\|_{L_{\widetilde{\cP}}^p(\bP;L^q(\nu;X))}\leq 1} \E\int_{\R_+\ti J} \langle g,\tilde{h}\circ A\rangle d\nu\\
& = \sup_{\|\tilde{h}\|_{L_{\widetilde{\cP}}^p(\bP;L^q(\nu_{\tau};X))}\leq 1}\E\int_{\R_+\ti J} \langle g\circ\tau,\tilde{h}\rangle d\nu_{\tau}.
\end{align*}
Applying the previous part of the proof for $\nu=\nu_{\tau}$, we find 
\begin{align*}
\|F_g\| \geq \Big(\frac{p}{q}\Big)^{1/q} \frac{q'}{p'} \|g\circ \tau\|_{L^{p'}_{\widetilde{\cP}}(\bP;L^{q'}(\nu_{\tau};X^*))} = \|g\|_{L^{p'}_{\widetilde{\cP}}(\bP;L^{q'}(\nu;X^*))}.
\end{align*}
This completes our proof of \eqref{eqn:predDualityNE}.\par
\emph{Step 3: representation of linear functionals.} It now remains to show that every $F\in (L^p_{\widetilde{\cP}}(\bP;L^q(\nu;X)))^*$ is of the form $F_g$ for a suitable $\widetilde{\cP}$-measurable function $g$. We will first assume that $\mathbb E\nu(\mathbb R_+ \times J) < \infty$. On $\widetilde{\cP}$ we can define an $X^*$-valued measure $\theta$ by setting 
 \[
  \langle \theta(A), x\rangle:= F(\mathbf 1_{A}\cdot x) \qquad (A\in \widetilde{\cP}, x\in X).
 \]
Then $\theta$ is $\sigma$-additive, absolutely continuous with respect to $\mathbb P \times \nu$. Moreover, by the same calculation as in \eqref{eqn:totvarBound}, for any disjoint partition $A_1,\ldots,A_n\in \widetilde{\cP}$ of $\R_+\ti\Om\ti J$,
\begin{align*}
\sum_{i=1}^n \|\theta(A_i)\| & \leq \|F\|_{(L^p_{\widetilde{\cP}}(\bP;L^q(\nu;X)))^*} (\E\nu(\R_+\ti J)^{p/q})^{1/p} \\
& \leq \|F\|_{(L^p_{\widetilde{\cP}}(\bP;L^q(\nu;X)))^*} (\E\nu(\R_+\ti J))^{1/q},
\end{align*}
so $\theta$ is of finite variation. By the Radon-Nikodym property of $X^*$, there exists a $\widetilde{\cP}$-measurable $X^*$-valued function $g$ such that 
$$F(h) = F_g(h) = \mathbb E \int_{\mathbb R_+ \times J}\langle g,h \rangle\ud \nu$$ 
for each $h\in L^p_{\widetilde{\cP}}(\bP;L^q(\nu;X))$. The extension to the general case, where $\nu$ is \mbox{$\widetilde{\cP}$-$\si$-fi}\-nite, can now be obtained in the same way as in the proof of Theorem~\ref{thm:dualofD_q^p(X)}. 
\end{proof}
\begin{remark}
The reader may wonder whether the duality
$$(L^p_{\widetilde{\cP}}(\bP;L^q(\nu;X)))^* = L^{p'}_{\widetilde{\cP}}(\bP;L^{q'}(\nu;X^*))$$
remains valid if $\nu$ is any random measure and $\widetilde{\cP}$ is replaced by an arbitrary sub-$\si$-algebra of $\cB(\R_+)\otimes \cF\otimes \cJ$. It turns out that, surprisingly, one cannot expect such a general result. Indeed, it was pointed out by Pisier \cite{PisPC} that there exist two probability spaces $(\Om_1,\cF_1,\bP_1)$, $(\Om_2,\cF_2,\bP_2)$ and a sub-$\si$-algebra $\cG$ of $\cF_1\otimes \cF_2$, so that the duality 
$$(L^p_{\cG}(\bP_1;L^q(\bP_2)))^* = L^{p'}_{\cG}(\bP_1;L^{q'}(\bP_2))$$
does not even hold isomorphically. This counterexample in particular shows that the duality results claimed in \cite{LYZ12} are not valid without imposing additional assumptions.
\end{remark}

\subsection{$\mathcal S_q^p$ and $\hat{\mathcal S}_q^p$ spaces}

Let $\nu$ be any random measure on $\mathcal B(\mathbb R_+) \otimes \mathcal J$. Recall that $\mathcal S_q^p$ is the space of all $\mathcal B(\mathbb R_+)\otimes \mathcal F \otimes \mathcal J$-strongly measurable functions $f:\mathbb R_+ \otimes \Omega\otimes J \to L^q(S)$ satisfying
\begin{equation}
\label{eqn:SpqDefApp}
\|f\|_{\mathcal S_q^p} = \Bigl( \mathbb E \Bigl\| \Bigl(\int_{\mathbb R_+\times J}|f|^2 \ud \nu \Bigr)^{\frac 12}\Bigr\|^p_{L^q(S)} \Bigr)^{\frac 1p}<\infty.
\end{equation}
The proof of the following result is analogous to Theorem~\ref{thm:dualofD_q^p(X)}. We leave the details to the reader.
\begin{theorem}\label{thm:dualofS^p_q}
 Let $1<p,q<\infty$. Then $(\mathcal S_{q}^p)^*  =\mathcal S_{q'}^{p'}$ and
 \begin{equation*}
  \|f\|_{\mathcal S_{q'}^{p'}} \eqsim_{p,q} \|f\|_{(\mathcal S_{q}^p)^*},\;\;\; f\in \mathcal S_{q'}^{p'}.
 \end{equation*}
\end{theorem}
Let us now prove the desired duality for $\hat{\mathcal S}_q^p$, the subspace of all $\widetilde{\cP}$-strongly measurable functions in $\mathcal S_{q}^p$.
\begin{theorem}\label{thm:hatspqdual}
 Let $1<p,q<\infty$. Suppose that $\nu$ is a predictable, $\widetilde{\cP}$-$\si$-finite random measure on $\cB(\R_+)\ot \cJ$ that is non-atomic in time. Then $(\hat{\mathcal S}_q^p)^* = \hat{\mathcal S}_{q'}^{p'}$ and
 \begin{equation}\label{eq:spqnormequiv}
  \|f\|_{\hat{\mathcal S}_{q'}^{p'}} \eqsim_{p,q}\|f\|_{(\hat{\mathcal S}_q^p)^*}, \;\;\; f \in  \hat{\mathcal S}_{q'}^{p'}.
 \end{equation}
\end{theorem}

For the proof of Theorem \ref{thm:hatspqdual} we will the following assertion. Given a filtration $\mathbb F = (\mathcal F_n)_{n\geq 0}$ and $1<p,q<\infty$, we define $ Q_q^p$ to be the Banach space of all adapted $L^q(S)$-valued sequences $(f_n)_{n\geq 0}$ satisfying
\begin{equation}\label{eq:defofQ^p_q}
  \|(f_n)_{n\geq 0}\|_{ Q_q^p}:= \Bigl( \mathbb E \Bigl\|\Bigl(\sum_{n=0}^{\infty}|f_n|^2\Bigr)^{\frac 12} \Bigr\|^{p}_{L^q(S)}\Bigr)^{\frac 1p}<\infty.
\end{equation}

\begin{proposition}\label{prop:dualofQ_q^p}
 Let $1<p,q<\infty$. Then $(Q_q^p)^* = Q_{q'}^{p'}$ isomorphically, with duality bracket given by 
 \begin{align*}
  \langle (f_n)_{n\geq 0}, (g_n)_{n\geq 0} \rangle := \mathbb E \sum_{n=0}^{\infty} \langle f_n, g_n\rangle \qquad ((g_n)_{n\geq 0} \in Q_{q'}^{p'}, \ (f_n)_{n\geq 0} \in Q_{q}^{p}).
 \end{align*}
Moreover, 
\[
 \|(g_n)_{n\geq 0}\|_{ Q_{q'}^{p'}} \eqsim_{p,q} \|(g_n)_{n\geq 0}\|_{ (Q_{q}^{p})^*}. 
\]
\end{proposition}
\begin{proof}
Consider the filtration $\mathbb G = (\mathcal G_n)_{n\geq 0} = (\mathcal F_{n+1})_{n\geq 0}$. Let $S_q^p$ be the conditional sequence space defined in \eqref{eqn:SpqNorm} for the filtration $\mathbb G$. First notice that $Q_q^p$ is a closed subspace and
$$\|(f_n)_{n\geq 0}\|_{ Q_q^p} = \|(f_n)_{n\geq 0}\|_{S_q^p}, \qquad \text{for all } (f_n)_{n\geq 0} \in Q_q^p.$$
Let $F$ be in $(Q_q^p)^*$. Then by the Hahn-Banach theorem and \eqref{eqn:SqdualNC} there exists $\tilde g = (\tilde g_n)_{n\geq 0} \in S_{q'}^{p'}$ such that $\|\tilde g\|_{S_{q'}^{p'}} \eqsim_{p, q} \|F\|_{(Q_q^p)^*}$ and
 \[
  F(f) = \mathbb E \sum_{n=1}^{\infty}\langle f_n, \tilde g_n \rangle,\;\;\; f=(f_n)_{n\geq 0}\in Q_q^p.
 \]
Now let $(g_n)_{n\geq 0}$ be the $\mathbb F$-adapted $L^q(S)$-valued sequence defined by $g_n = \mathbb E_{n} \tilde g_n$ for $n\geq 0$ (recall that $\mathbb E_n(\cdot) := \mathbb E(\cdot | \mathcal F_n)$). Then, on the one hand, the conditional Jensen inequality yields
\begin{align*}
  \|(g_n)_{n\geq 0}\|^{p'}_{Q_{q'}^{p'}} &= \|(g_n)_{n\geq 0}\|^{p'}_{S_{q'}^{p'}} = \mathbb E \Bigl\|\Bigl( \sum_{n=1}^{\infty} |g_n|^2 \Bigr)^{\frac 12}\Bigr\|_{L^{q'}(S)}^{p'} = \mathbb E \Bigl\|\Bigl( \sum_{n=1}^{\infty} |\mathbb E_{n} \tilde g_n|^2 \Bigr)^{\frac 12}\Bigr\|_{L^{q'}(S)}^{p'}\\
  &\leq \mathbb E \Bigl\|\Bigl( \sum_{n=1}^{\infty}\mathbb E_{n} | \tilde g_n|^2 \Bigr)^{\frac 12}\Bigr\|_{L^{q'}(S)}^{p'}
  =  \|(\tilde g_n)_{n\geq 0}\|^{p'}_{S_{q'}^{p'}},
\end{align*}
and, on the other hand, for each $f=(f_n)_{n\geq 0}\in Q_q^p$ the $\mathbb F$-adaptedness of $(f_n)_{n\geq 0}$ implies
\[
 F(f) = \mathbb E \sum_{n=1}^{\infty}\langle f_n, \tilde g_n \rangle = \mathbb E \sum_{n=1}^{\infty}\mathbb E_{n}\langle f_n, \tilde g_n \rangle =  \mathbb E \sum_{n=1}^{\infty}\langle f_n,\mathbb E_{n} \tilde g_n \rangle = \mathbb E \sum_{n=1}^{\infty}\langle f_n, g_n \rangle.
\]
Therefore, for each $F \in (Q_q^p)^*$ there exists a $(g_n)_{n\geq 0}\in Q_{p'}^{q'}$ such that 
\begin{align*}
 F(f) = \mathbb E \sum_{n\geq 0} \langle f_n, g_n\rangle,\;\;\; &f=(f_n)_{n\geq 0} \in Q_q^p,\\
 \|(g_n)_{n\geq 0}\|_{Q_{p'}^{q'}} &\lesssim_{p, q} \|F\|_{(Q_q^p)^*}.
\end{align*}
The inequality $\|F\|_{(Q_q^p)^*}\leq \|(g_n)_{n\geq 0}\|_{Q_{p'}^{q'}}$ follows immediately from H\"older's inequality.
\end{proof}

\begin{proof}[Proof of Theorem \ref{thm:hatspqdual}]
 The proof contains two parts. In the first part, consisting of several steps, we will show that $\|f\|_{\hat{\mathcal S}_{q'}^{p'}}\eqsim_{p, q}\|f\|_{(\hat{\mathcal S}_{q}^{p})^*}$. In the second part we show that $(\hat{\mathcal S}^p_q)^* = \hat{\mathcal S}^{p'}_{q'}$.
 
 {\it Step 1: $J$ is finite, $\nu$ is Lebesgue.}  Let $J = \{j_1,\ldots, j_K\}$, $\nu(\omega)$ be the product of Lebesgue measure and the counting measure on $ \mathbb R_+\times J$ for all $\omega\in \Omega$ (i.e.\ $\nu((s, t]\times j_k) = t-s$ for each $k=1,\ldots, K$ and $t\geq s\geq 0$). Fix $f\in \hat{\mathcal S}_{q'}^{p'}$. Without loss of generality we can assume that $f$ is simple and that there exist $N, M\geq 1$ and a sequence of random variables $(f_{k,m})_{k=1, m=0}^{k=K, m=M}$ such that $f_{k,m}$ is $\mathcal F_{\frac{m}{N}}$-measurable and $f(t, j_k) = f_{k,m}$ for each $k=1,\ldots,K$, $m=0,\ldots, M$, and $t\in (\frac{m}{N}, \frac{m+1}{N}]$. Let $\mathbb G = (\mathcal G_{k,m})_{k=1, m=0}^{k=K, m=M} := (\mathcal F_{\frac{m}{N}})_{k=1, m=0}^{k=K, m=M}$. Then $\mathbb G$ forms a filtration with respect to the reverse lexicographic order on the pairs $(k,m)$, $1\leq k\leq K$ and $0\leq m\leq M$. Let $Q^{p'}_{q'}$ be as defined in \eqref{eq:defofQ^p_q} for $\mathbb G$. Then
 \begin{equation}\label{eq:phiseqandfunc}
   \|f\|_{\hat{\mathcal S}_{q'}^{p'}} = \frac 1{\sqrt N}\bigl\|(f_{k,m})_{k=1, m=0}^{k=K, m=M}\bigr\|_{Q_{q'}^{p'}}.
 \end{equation}
By Proposition \ref{prop:dualofQ_q^p} there exists a $\mathbb G$-adapted $(g_{k,m})_{k=1, m=0}^{k=K, m=M}\in Q^{p}_{q}$ such that 
\[
  \bigl\|(g_{k,m})_{k=1, m=0}^{k=K, m=M}\bigr\|_{Q^{p}_{q}}=1                                                                                                                                \]
and 
\[
 \bigl\langle (f_{k,m})_{k=1, m=0}^{k=K, m=M}, (g_{k,m})_{k=1, m=0}^{k=K, m=M}\bigr\rangle \eqsim_{p, q} \bigl\|(f_{k,m})_{k=1, m=0}^{k=K, m=M}\bigr\|_{Q^{p'}_{q'}}.
\]
Let $g:\mathbb R_+ \times \Omega \times J \to L^q(S)$ be defined by setting $g(t, j_k) = \sqrt{N}g_{k, m}$ for each $k=1,\ldots,K$, $m=1,\ldots, M$, and $t\in (\frac{m}{N}, \frac{m+1}{N}]$. Then $g \in \hat{\mathcal S}_{q}^{p}$, and analogously to~\eqref{eq:phiseqandfunc}
 \begin{equation*}
   \|g\|_{\hat{\mathcal S}_{q}^{p}} = \bigl\|(g_{k,m})_{k=1, m=0}^{k=K, m=M}\bigr\|_{Q_{q}^{p}}=1.
 \end{equation*}
Moreover,
\begin{align*}
 \langle f,g\rangle &= \mathbb E \int_{\mathbb R_+ \times J} \langle f(t, j), g(t, j)\rangle\ud t \ud j =\frac 1{\sqrt N} \mathbb E \sum_{k=1, m=0}^{k=K, m=M}\langle f_{k,m}, g_{k,m}\rangle\\
 &\eqsim_{p, q}\frac 1{\sqrt N} \bigl\|(f_{k,m})_{k=1, m=0}^{k=K, m=M}\bigr\|_{Q^{p'}_{q'}} = \|f\|_{\hat{\mathcal S}_{q'}^{p'}},
\end{align*}
which finishes the proof.

{\it Step 2: $J$ is finite, $\nu((s, t]\times J)\leq t-s$ a.s.\ for each $t\geq s\geq 0$.} Let $\nu_0$ be the product of Lebesgue measure and the counting measure on $\mathbb R_+ \times J$ (see Step~1). Then clearly $\bP\otimes \nu$ is absolutely continuous with respect to $\bP\otimes \nu_0$ and by the Radon-Nikodym theorem there exists a $\widetilde {\mathcal P}$-measurable density $\phi:\mathbb R_+ \times \Omega \times J \to \mathbb R_+$ such that $d(\bP\otimes \nu) = \phi \ d(\bP\otimes \nu_0)$. Fix $f \in \hat{\mathcal S}_{q'}^{p'}$. Let $\hat{\mathcal S}_{q'}^{p',\nu_0}$ be as defined in \eqref{eqn:SpqDefApp} for the random measure~$\nu_0$. Then $f_0 := f \cdot \sqrt {\phi}\in \hat{\mathcal S}_{q'}^{p',\nu_0}$, and $\|f\|_{\hat{\mathcal S}_{q'}^{p'}} = \|f_0\|_{\hat{\mathcal S}_{q'}^{p',\nu_0}}$. By Step 1 there exists a $g_0\in \hat{\mathcal S}_{q}^{p,\nu_0}$ such that $\|g_0\|_{\hat{\mathcal S}_{q}^{p,\nu_0}} = 1$ and $\langle f_0, g_0\rangle \eqsim_{p, q}\|f_0\|_{\hat{\mathcal S}_{q'}^{p',\nu_0}}$. Let $g = g_0 \mathbf 1_{\phi\neq 0} \frac 1{\sqrt{\phi}}$. Then
\begin{align*}
 \langle f, g\rangle &= \mathbb E \int_{\mathbb R_+ \times J}\langle f, g\rangle \ud \nu = \mathbb E \int_{\mathbb R_+ \times J}\langle f, g\rangle \phi\ud \nu_0 = \mathbb E \int_{\mathbb R_+ \times J}\langle f \sqrt{\phi}, g\sqrt{\phi}\rangle \ud \nu_0\\
 &=\mathbb E \int_{\mathbb R_+ \times J}\langle f_0, g_0\rangle \ud \nu_0 = \langle f_0, g_0\rangle\eqsim_{p, q}\|f_0\|_{\hat{\mathcal S}_{q'}^{p',\nu_0}} = \|f\|_{\hat{\mathcal S}_{q'}^{p'}}
\end{align*}
and
\begin{align*}
 \|g\|_{\hat{\mathcal S}_{q}^{p}} &= \Bigl(\mathbb E \Bigl\| \Bigl(\int_{\mathbb R_+ \times J} |g|^2 \ud \nu\Bigr)^{\frac 12} \Bigr\|^{p} \Bigr)^{\frac 1p} = \Bigl(\mathbb E \Bigl\| \Bigl(\int_{\mathbb R_+ \times J} |g_0|^2\mathbf 1_{\phi \neq 0} \frac {1}{\phi} \ud \nu\Bigr)^{\frac 12} \Bigr\|^{p} \Bigr)^{\frac 1p}\\
 &=\Bigl(\mathbb E \Bigl\| \Bigl(\int_{\mathbb R_+ \times J} |g_0|^2\mathbf 1_{\phi \neq 0} \ud \nu_0\Bigr)^{\frac 12} \Bigr\|^{p} \Bigr)^{\frac 1p} \leq \Bigl(\mathbb E \Bigl\| \Bigl(\int_{\mathbb R_+ \times J} |g_0|^2\ud \nu_0\Bigr)^{\frac 12} \Bigr\|^{p} \Bigr)^{\frac 1p} \\
 &=\|g_0\|_{\hat{\mathcal S}_{q}^{p,\nu_0}} = 1.
\end{align*}
Therefore $ \|f\|_{\hat{\mathcal S}_{q'}^{p'}} \eqsim_{p, q}  \|f\|_{(\hat{\mathcal S}_{q}^{p})^*}$.

{\it Step 3: $J$ is finite, $\nu$ is general.} Without loss of generality we can assume that $\mathbb E \nu(\mathbb R_+ \times J)<\infty$. Then by a time-change argument as was used in the proof of Theorem \ref{thm:dualofL^p_P(L^q(X))}, we can assume that $\nu((s, t]\times J)\leq t-s$ a.s.\ for each $t\geq s\geq 0$, and apply Step 2.

{\it Step 4: $J$ is general, $\nu$ is general.}  Without loss of generality assume that $\mathbb E\nu(\mathbb R_+ \times J)<\infty$. Let $f$ be simple $\widetilde{\mathcal P}$-measurable. Then there exists a $K\geq 1$ and a partition $J = J_1\cup \cdots \cup J_K$ of $J$ into disjoint sets such that $f(i) = f(j)$ for all $i, j\in J_k$ and each $k=1,\ldots, K$.  Fix $j_k\in J_k$, $k=1,\ldots,K$, and define $\widetilde J = \{j_1, \ldots, j_K\}$. Let $\widetilde {\nu}$ be a new random measure on $\mathbb R_+ \times \Omega \times \widetilde J$ defined by 
$$\widetilde {\nu}(A\times \{j_k\}) = \nu(A\times J_k), \qquad A\in \cP, \ k=1,\ldots, K.$$ 
Let $\hat{\mathcal S}_{q'}^{p',\widetilde {\nu}}$ be as constructed in \eqref{eqn:SpqDefApp} for the measure $\widetilde {\nu}$. Let $\widetilde{f}\in \hat{\mathcal S}_{q'}^{p',\widetilde {\nu}}$ be such that $\widetilde {f}(j_k) = f(j_k)$ for each $k=1,\ldots, K$. Then $\|\widetilde {f}\|_{\hat{\mathcal S}_{q'}^{p',\widetilde {\nu}}} = \|{f}\|_{\hat{\mathcal S}_{q'}^{p'}}$. By Step~3 there exists a $\widetilde {g} \in \hat{\mathcal S}_{q}^{p,\widetilde {\nu}}$ such that $\|\widetilde {g}\|_{\hat{\mathcal S}_{q}^{p,\widetilde {\nu}}} = 1$ and $\langle \widetilde{f}, \widetilde{g}\rangle \eqsim_{p, q}\|\widetilde{f}\|_{\hat{\mathcal S}_{q'}^{p',\widetilde {\nu}}}$.

Define $g\in \hat{\mathcal S}_{q}^{p}$ by setting $g(j) = \widetilde{g}(j_k)$ for each $k=1,\ldots, K$ and $j\in J_k$. Then $\|g\|_{\hat{\mathcal S}_{q}^{p}} = \|\widetilde{g}\|_{\hat{\mathcal S}_{q}^{p, \widetilde{\nu}}}= 1$. Moreover,
\begin{align*}
 \langle f, g\rangle &= \mathbb E \int_{\mathbb R_+ \times J}\langle f(t, j), g(t, j)\rangle\ud \nu(t, j) = \mathbb E \sum_{k=1}^K \int_{\mathbb R_+ \times J_k}\langle f(t, j), g(t, j)\rangle\ud \nu(t, j) \\
 &= \mathbb E \int_{\mathbb R_+ \times \widetilde J}\langle\widetilde{f}(t, j), \widetilde {g}(t, j)\rangle\ud \widetilde{\nu}(t, j)
 \eqsim_{p, q} \|\widetilde{f}\|_{\hat{\mathcal S}_{q'}^{p',\widetilde {\nu}}} = \|{f}\|_{\hat{\mathcal S}_{q'}^{p'}}.
\end{align*}
Hence, $\|{f}\|_{\hat{\mathcal S}_{q'}^{p'}} \eqsim_{p, q} \|{f}\|_{(\hat{\mathcal S}_{q}^{p})^*}$.

{\it Step 5: $(\hat{\mathcal S}^p_q)^* = \hat{\mathcal S}^{p'}_{q'}$.} In Step 4 we proved that $\hat{\mathcal S}^{p'}_{q'}\hookrightarrow (\hat{\mathcal S}^p_q)^*$ isomorphically, so it remains to show that $(\hat{\mathcal S}^p_q)^* = \hat{\mathcal S}^{p'}_{q'}$. This identity follows from the same Radon-Nikodym argument that was presented in Step 3 of the proof of Theorem \ref{thm:dualofL^p_P(L^q(X))}.
\end{proof}

\begin{corollary}\label{cor:dualIpq}
 Let $1<p,q<\infty$. Then $\mathcal I_{p,q}^* = \mathcal I_{p',q'}$, where $\mathcal I_{p,q}$ is as defined in~\eqref{eq:ipq}, and
 \begin{equation}\label{eq:normeqipq}
    \|f\|_{\mathcal I_{p',q'}} \eqsim_{p,q}\|f\|_{\mathcal I_{p,q}^*}, \;\;\; f \in  \mathcal I_{p',q'}.
 \end{equation}
\end{corollary}

\begin{proof}
 The result follows by combining Theorem~\ref{thm:dualofL^p_P(L^q(X))} (for $X=L^q(S)$), Theorem \ref{thm:hatspqdual} and \eqref{eqn:dualofsum}.
\end{proof}

\section*{Acknowledgements}

The authors would like to thank Carlo Marinelli, Jan van Neerven and Mark Veraar for helpful discussions on the topic of this paper.

\bibliographystyle{plain}

\end{document}